\documentclass[reqno,oneside,12pt]{amsart}

\usepackage[T1]{fontenc}
\usepackage{times, mathptm}
\usepackage{amscd, color, enumerate}
\usepackage[svgnames]{xcolor}
\usepackage{amssymb}
\usepackage{mathtools}
\usepackage{stackrel}

\usepackage[colorlinks=true, pdfstartview=FitV, linkcolor=blue, citecolor=blue, urlcolor=blue]{hyperref}

\title{Solvable Automorphism Groups of Varieties}

\author[Cantat et al]{Serge Cantat, Hanspeter Kraft, \\Andriy Regeta, Immanuel van Santen}

\address{Univ. Rennes, CNRS, IRMAR - UMR 6625,\newline\indent
35000 Rennes, Frances}
\email{serge.cantat@univ-rennes1.fr}
\address{Departement Mathematik und Informatik,\newline\indent
Universit\"at Basel,  Spiegelgasse~1, CH-4051 Basel}
\email{hanspeter.kraft@unibas.ch}
\address{Dipartimento di Matematica ``Tullio Levi-Civita'',\newline\indent
Universit\`a di Padova, Via Trieste 63, I-35121 Padova}
\email{andriyregeta@gmail.com}
\address{Departement Mathematik und Statistik,\newline\indent
Universit\"at Bern, Sidlerstrasse 5, CH-3012 Bern}
\email{immanuel.van.santen@gmail.com}
\date{2024-2026}

\theoremstyle{plain}
\newtheorem{thm}{Theorem}[section]
\newtheorem{mthm}{Theorem}

\newtheorem{mcor}[mthm]{Corollary}
\newtheorem{ramthm}{Ramanujam's Theorem\!}

\newtheorem{prop}[thm]{Proposition} 
\newtheorem{lem}[thm]{Lemma}
\newtheorem*{lem*}{Lemma}

\newtheorem{cor}[thm]{Corollary}

\theoremstyle{definition}
\newtheorem{defn}[thm]{Definition}
\newtheorem{exa}[thm]{Example}
\newtheorem{exas}[thm]{Examples}
\newtheorem{rem}[thm]{Remark}
\newtheorem{cons}[thm]{Construction}

\newcommand{\into}{\hookrightarrow}

\newcommand{\ZZ}{{\mathbb Z}}
\newcommand{\NN}{{\mathbb N}}
\newcommand{\CC}{{\mathbb C}}
\newcommand{\RR}{{\mathbb R}}
\newcommand{\An}{{\mathbb A}^{n}}
\newcommand{\Aone}{{\mathbb A}^{1}}
\newcommand{\Atwo}{{\mathbb A}^{2}}
\newcommand{\Ast}{{\mathbb A}^\star}
\newcommand{\bbA}{{\mathbb A}}
\newcommand{\FF}{{\mathbb F}}

\newcommand{\NNN}{{\mathcal{N}}}
\newcommand{\VVV}{{\mathcal{V}}}
\newcommand{\WWW}{\mathcal W}
\newcommand{\UUU}{{\mathcal{U}}}
\newcommand{\bUUU}{{\overline{\mathcal{U}}}}
\newcommand{\HHH}{\mathcal H}
\newcommand{\bHHH}{\overline{\HHH}}
\newcommand{\KKK}{\mathcal K}
\newcommand{\OOO}{\mathcal{O}}
\newcommand{\GGG}{\mathcal{G}}
\newcommand{\FFF}{\mathcal{F}}
\newcommand{\bGGG}{\overline{\GGG}}
\newcommand{\BBB}{\mathcal{B}}

\newcommand{\PPP}{\mathcal P}
\newcommand{\simto}{\xrightarrow{\sim}}
\newcommand{\be}{
\begin{enumerate}}
\newcommand{\ee}{
\end{enumerate}}

\renewcommand{\phi}{\varphi}

\newcommand{\kst}{{\kk^\times}}

\DeclareMathOperator{\id}{id}
\DeclareMathOperator{\GL}{GL}
\DeclareMathOperator{\SL}{SL}

\DeclareMathOperator{\SO}{SO}
\DeclareMathOperator{\Aut}{Aut}

\DeclareMathOperator{\Spec}{Spec}
\DeclareMathOperator{\Quot}{Quot}
\DeclareMathOperator{\Aff}{Aff}

\newcommand{\aff}{\textrm{aff}}

\DeclareMathOperator{\pr}{pr}
\DeclareMathOperator{\B}{B}
\newcommand{\Ga}{\mathbf{G}_{a}}
\newcommand{\Gm}{\mathbf{G}_{m}}

\DeclareMathOperator{\Tr}{Tr}

\DeclareMathOperator{\End}{End}
\DeclareMathOperator{\Hom}{Hom}
\DeclareMathOperator{\Char}{char}

\DeclareMathOperator{\diag}{diag}

\DeclareMathOperator{\dl}{dl}
\DeclareMathOperator{\mdo}{{{\mathrm{mdo}}}}

\DeclareMathOperator{\Jonq}{Jonq}
\DeclareMathOperator{\codim}{codim}

\newcommand{\quot}{/\!\!/}

\renewcommand{\subset}{\subseteq}
\newcommand{\kk}{\Bbbk}
\newcommand{\KK}{{\mathbb K}}

\newcommand{\Xm}{X_{\text{\it{max}}}}
\newcommand{\MMM}{{\mathcal M}}

\newcommand{\SSS}{{\mathcal S}}

\newcommand{\Gaff}{G_{\text{\Tiny aff}}}
\newcommand{\Uaff}{U_{\text{\Tiny aff}}}
\newcommand{\Gant}{G_{\text{\Tiny ant}}}

\newcommand{\algebraicallyTame}{tempered}

\newcommand{\LP}{{\sf{L}}}
\newcommand{\TP}{{\sf{T}}}

\DeclarePairedDelimiter\ceil{\lceil}{\rceil}

\newcommand{\Xg}{X_{\eta}}
\newcommand{\Xgg}{X_{\bar\eta}}

\newcommand{\ps}{\par\smallskip}

\newcommand{\set}[2]{\left\{\,#1 \ | \ #2\,\right\}}

\addtocounter{section}{0} 
\numberwithin{equation}{section} 


\begin{document}

\begin{abstract}
Let $X$ be a variety of dimension $n$, and let
$\Aut(X)$ be its automorphism group.
When $X$ is quasi-affine, we prove that a solvable subgroup of $\Aut(X)$ that is generated by an irreducible family of automorphisms containing the identity is an algebraic subgroup.

Our main applications concern arbitrary varieties. First, every connected solvable  subgroup of $\Aut(X)$ is contained in a Borel subgroup and its derived length is  $\leq n+1$. Second, the notion of solvable and unipotent radicals are well defined for any subgroup of $\Aut(X)$. Third, if $X$ is quasi-affine and connected and $\BBB\subset \Aut(X)$ is a Borel subgroup of derived length $n+1$, then $X$ is isomorphic to the affine $n$-space $\An$ and $\BBB$  is conjugate to the Jonquières subgroup.
\end{abstract}

\thanks{The work of S.C. is supported by  the European Research Council (ERC GOAT 101053021). The work of A.R. was partially supported by  DFG, project number 509752046.
The work of I.v.S. was supported by the SNSF grant 10000940.}

\maketitle

\setcounter{tocdepth}{1}
\tableofcontents

\vfill
\pagebreak

\section{Introduction}

Let $\kk$ be an algebraically closed field. By convention, algebraic varieties, subvarieties, morphisms, etc. will always be defined over $\kk$. Throughout this article, $X$ 
denotes an algebraic variety, not necessarily irreducible, and $\Aut(X)$ its automorphism group.

\subsection{Algebraic families of automorphisms} \label{par:intro_algebraic_families}
A good example to keep in mind is the affine space of dimension $n$ which we denote by $\bbA^n$. Thanks to Jung's theorem (see~\cite{Ju42, Lamy:Book}), $\Aut(\bbA^2)$ is relatively well understood, but when $n\geq 3$, its properties are still rather mysterious. 

 \begin{defn}[see \cite{Ra1964A-note-on-automorp}]\label{Def.morphism} Let $A$ be an algebraic variety. A map $\rho\colon A\to \Aut(X)$ is a {\emph{morphism}} (or an {\emph{algebraic family}} of automorphisms) if  the map $A\times X\to X$ defined by 
$(a,x)\mapsto \rho(a)(x)$ is a morphism of algebraic varieties. A subset $V\subset \Aut(X)$ is {\emph{constructible}} if it is the image $\rho(A)$ of a morphism. If we can choose $A$ to be irreducible,  $V$ is an \emph{irreducible constructible} subset. 
\end{defn}

\begin{defn}\label{Defi.Algebraically_generated} 
A subgroup $G$ of $\Aut(X)$ is {\emph{algebraically generated}} if there is a morphism $\rho\colon A\to \Aut(X)$ such that (i) the algebraic variety $A$ is irreducible, (ii) 
$\rho(A)$ contains the identity $\id_X$, and (iii)
$G$ is generated  by $\rho(A)$ as a subgroup of $\Aut(X)$.
We write $G=\langle \rho(A)\rangle$.
\end{defn}
Equivalently, $G$ is algebraically generated if $G=\langle V\rangle$ for some irreducible constructible subset $V$ of $\Aut(X)$ containing $\id_X$.
We insist on assumptions (i) and (ii). They are crucial and could be easily forgotten, because they are implicit in the vocabulary ``algebraically generated'' (see Example~\ref{exa:introduction} below).  

\begin{defn}\label{defi:introduction_connected}
A subgroup $G$ of $\Aut(X)$ is  {\emph{connected}} if for every $g\in G$ 
there is an algebraic family $\rho\colon A\to \Aut(X)$ parametrized by an irreducible variety~$A$, such that $\rho(A)\subset G$ and $\rho(A)$ contains $\id_X$ and $g$. 
\end{defn}

The dimension $\dim(S)$ of a subset $S\subset \Aut(X)$  is the supremum of $\dim(A)$ over
all injective morphisms $\iota \colon A \to \Aut(X)$ with $\iota(A)\subset S$. If $\dim(S)<\infty$, then $S$ is  {\emph{finite dimensional}}. For instance,
$\Aut(\bbA^n)$ is finite dimensional if and only if $n\leq 1$.
The next result endows finite dimensional subgroups of $\Aut(X)$ with a canonical structure of an algebraic group.  We refer to \cite{Ra1964A-note-on-automorp} for the original version of this theorem, and to  \cite{Note_on_Ramanujam} for the generalization we state here which includes the case of reducible varieties.

\begin{ramthm}\hypertarget{Ramanujam.thm}{}
Let $X$ be an algebraic variety. Let $G$ be a subgroup of $\Aut(X)$ containing a connected and
finite dimensional subgroup of finite index. Then $G$ carries the structure of an
algebraic group which is uniquely determined by the following universal property:
\begin{enumerate}
\item[{\rm{(UP)}}] 
{The action $G\times X \to X$ is an action of the algebraic group $G$, and if $\mu\colon A \to \Aut(X)$ is a morphism such that $\mu(A) \subset G$, then the induced map 
$\mu\colon A \to G$ is a morphism of algebraic varieties.}
\end{enumerate}
\end{ramthm}
Such a group will be called an \textit{algebraic subgroup} of $\Aut(X)$. If an algebraic group acts algebraically on X, then its image is an algebraic subgroup, see Corollary~1.3 in \cite{Note_on_Ramanujam}.

A subgroup $\GGG$ of $\Aut(X)$ is \emph{nested} if it is a countable increasing union of algebraic subgroups: $\GGG = \bigcup_k G_k$. 
It is \emph{\algebraicallyTame{}} if it is connected and every algebraically generated subgroup is an algebraic group.
We shall see in Theorem~\ref{thm-nested-tame} that a group is \algebraicallyTame{} if and only if it is nested by connected algebraic subgroups.
 
\subsection{Solvable, algebraically generated groups} 
Recall that the derived subgroup $[G,G]$ of a group $G$ is the subgroup generated by all commutators $[g,h]=ghg^{-1}h^{-1}$
with $g,h\in G$.  Its derived series $G^{(n)}$ is defined recursively by $G^{(0)}:=G$ and $G^{(n+1)}:=[G^{(n)}, G^{(n)}]$. The group $G$ is \emph{solvable} if $G^{(n)}=\{e_G\}$ for some $n\geq 0$, and the {\emph{derived length}} $\dl(G)$ is the smallest such $n$.

Recall that a quasi-affine variety is an open subvariety of an affine variety.
\begin{mthm}\label{main-thm-A}
If $X$ is a quasi-affine variety, then 
every algebraically generated solvable  subgroup of $\Aut(X)$ is an 
affine algebraic subgroup. Thus,
a connected solvable subgroup is \algebraicallyTame{}.
\end{mthm}

This extends Theorem~B of~\cite{CaReXi2023Families-of-commut} from the affine and commutative to the quasi-affine and solvable setting. 
The proof  is given in \S\,\ref{Subsec.Proof_Main_Thm}.
 
\begin{exas} \label{exa:introduction} 
(1) Take $\kk=\CC$, the field of complex numbers. Consider the rotation $g\in \Aut(\bbA^1_\CC)$ defined by $g(x)=e^{2{\mathsf{i}}\pi\theta} x$ for some irrational~$\theta$. Then, $\langle g\rangle$ is an infinite cyclic group, it is dense in a copy of $\SO_2(\RR)\subset \Aut(\bbA^1_\CC)$ for the euclidean topology, and in a copy of $\Gm(\CC)\subset \Aut(\bbA^1_\CC)$ for the Zariski topology. The group $\langle g\rangle$ is {\emph{not}} algebraically generated in our sense. 

(2) Let $h\in \Aut(\bbA^2)$ be a Hénon map (see~\cite{Friedland-Milnor, Lamy:Book}), for instance 
$h(x,y)=(y+x^2+c, x)$ for some $c$ in $\kk$.
Set $G=\SL_2(\kk)\subset \Aut(\bbA^2)$.  Then, $W:= hGh^{-1}G$ is an irreducible constructible subset of $\Aut(\bbA^2)$ containing the identity. 
For a general choice of $g\in \SL_2(\kk)$, Jung's theorem implies that the automorphism $f=hgh^{-1}g$ satisfies $\deg(f^n)=\deg(h)^{2n}$; in particular, $\deg(f^n)$ is not bounded, so $f$ is not contained in an algebraic subgroup of $\Aut(\bbA^2)$. 
Thus, $\langle W\rangle$ is an algebraically generated (non-solvable) group that is not algebraic. 

(3) Let~$V \subseteq \textrm{Bir}(\bbA^2)$ be the irreducible constructible set
of birational transformations 
$(x, y) \mapstochar\dashrightarrow (x, y + \frac{\lambda}{x-\mu})$,
where $(\lambda, \mu) \in \kk^2$. The group 
$\langle V \rangle$ is abelian. It
is not algebraic, because the elements $\frac{1}{x - \mu} \in \kk(x)$ 
for $\mu \in \kk$ are linearly independent over $\kk$. Thus, Theorem~\ref{main-thm-A} fails for birational transformations.
\end{exas} 

\subsection{Connected solvable subgroups}\label{connected_solvable_introduction.sec}
An element $g\in \Aut(X)$ is said to be {\emph{algebraic}} if it is contained in an algebraic subgroup $G$ of $\Aut(X)$. 
When
$X$ is quasi-affine, $G$ is automatically an affine algebraic group
(see Lemma~\ref{quasiaffine-gives-affine.lem}), and we shall say that $g$ is {\emph{unipotent}}  if $g$ is  unipotent as an element of $G$. Then a subgroup  $\UUU \subseteq \Aut(X)$, $X$ again quasi-affine, is  \emph{unipotent} if every element of $\UUU$ is unipotent.
For instance the automorphisms $(x,y)\mapsto (x,y+p(x))$ with $p\in \kk[x]$ form a unipotent subgroup
of $\Aut(\bbA^2)$. We refer to \S\,\ref{Rosdecompandunipelem}  for a more detailed discussion of these notions.

An \emph{algebraic torus} (shortly a \emph{torus}) in $\Aut(X)$ is an algebraic subgroup isomorphic to $(\Gm)^s$ for some $s\geq 1$ where $\Gm$ denotes the multiplicative group $\kk^\times:=(\kk\setminus\{0\},\times)$. The following result is our first application of Theorem~\ref{main-thm-A}; its proof is given at the end of \S\,\ref{sec.Levi}.

\begin{mcor}
\label{corB}
Let $X$ be a quasi-affine
variety, and let $\HHH \subseteq \Aut(X)$ be a connected solvable subgroup. 
Let $\UUU \subseteq \HHH$ be the set of all unipotent elements in $\HHH$.
Then $\UUU$ is a tempered unipotent subgroup and 
$\HHH$ is a semi-direct product of $\UUU$ and an algebraic torus 
$T\subset \Aut(X)$.
\end{mcor}

The normal subgroup $\UUU\subset \HHH$ is called the \emph{unipotent radical} of $\HHH$.
Unipotent and solvable radicals will be discussed in more generality  in Section~\ref{unipotent_radical.sec}.

\subsection{Jonquières subgroups}\label{jonquieres.sec}
We denote by $\Jonq(n)\subset \Aut(\bbA^n)$ the subgroup of \emph{Jonquières automorphisms}, that is, automorphisms of type 
\begin{equation*}
f(x_1, \ldots, x_n)=(a_1x_1+p_1, a_2x_2+p_2(x_1), \ldots, a_n x_n+p_n(x_1, \ldots, x_{n-1}))
\end{equation*} 
where the $a_i$ are non-zero elements of $\kk$, $p_1$ is in $\kk$, and for $i\geq 2$ the $p_i$ are polynomial functions in the variables $x_1, \ldots, x_{i-1}$. Thus, 
$\Jonq(n)$ is the group of automorphisms $f \in \Aut(\An)$ such that $f^*\kk[x_1,\ldots,x_j] = \kk[x_1,\ldots,x_j]$ for $j=1,\ldots,n$.
Equivalently, $\Jonq(n)$ is the stabilizer of the coflag
\begin{equation}
\label{Eq.coflag}
\bbA^n \to \bbA^{n-1} \to \cdots \to \bbA^2 \to \bbA^1
\end{equation}
given by the linear projections $(x_1,\ldots,x_{j-1},x_{j})\mapsto (x_1,\ldots,x_{j-1})$. 
By Lemma 3.1 and 3.2 of~\cite{FuPo2018On-the-maximality-}, $\Jonq(n)$ is connected, nested, and solvable of derived length $n+1$.

Let $T_n\subset \GL_n(\kk)\subset \Aut(\An)$ be the standard
$n$-dimensional torus (acting by diagonal matrices).
Let $\Jonq_u(n)\subset \Jonq(n)$ be the subgroup of elements $f$ as above with trivial diagonal part, i.e., all $a_i$ equal to $1$.
Then 
\begin{equation*}
\Jonq(n) = T_n\ltimes \Jonq_u(n),
\end{equation*}
$\Jonq_u(n)=[\Jonq(n),\Jonq(n)]$,
and  $\Jonq_u(n)$ is nested by connected unipotent algebraic subgroups.  
Proposition~3.4 of~\cite{FuPo2018On-the-maximality-} shows that $\Jonq(n)$ is 
maximal among all solvable subgroups of $\Aut(\bbA^n)$. In particular, it is a Borel subgroup in the sense of \S\,\ref{Borel-subgroups} below. 
The next theorem, proven in  \S\,\ref{solvable.subgroups.sec}, complements this result and answers the question in~\cite[p. 394]{FuPo2018On-the-maximality-} on the derived length of solvable subgroups.

\begin{mthm}\label{main-thm-derived-length}
Let $X$ be an algebraic variety of dimension $n$, and let $\HHH$ be a connected solvable subgroup of  $\Aut(X)$. 
\begin{enumerate}[\rm (1)] 
\item The derived length of $\HHH$ is at most $n+1$.
\item Assume in addition that $X$ is connected and quasi-affine. 
If the unipotent radical $\UUU$ of $\HHH$ has derived length $\geq n$, then $X \simeq \An$ and $\HHH$ is conjugate to a subgroup of $\Jonq(n)$. 
\end{enumerate}
\end{mthm}

The proof is given at the end of \S\,\ref{solvable.subgroups.sec}.
Assertion~(1) should be compared with analogous results obtained for Lie algebras of vector fields in characteristic $0$ in~\cite{Epstein-Thurston:1979, Makedonskyi-Petravchuk:2014}. 
In (2), $X$ must be connected because $\Aut(\An) = \Aut(\An\cup Y)$ for the disjoint union $\An\cup Y$ if $\Aut(Y)$ is trivial; and $X$ must be quasi-affine, as shown in Example~\ref{Exa.Smooth_proj_dl_n_plus_1}.

\subsection{Borel subgroups}\label{Borel-subgroups}
Recall that a {\emph{Borel subgroup}} of an algebraic group $G$ is a subgroup which is maximal among all connected and solvable subgroups of~$G$. Over an algebraically closed field, there is a single conjugacy class of Borel subgroups. For instance, the upper triangular matrices form a Borel subgroup of $\GL_n(\kk)$ of derived length $1+\ceil*{\log_2(n)}$ (see~\cite[\S18, Theorem~1]{Su1976Matrix-groups}).
By analogy, a subgroup of $\Aut(X)$ is a \textit{Borel subgroup} 
if it is maximal among connected solvable subgroups. 
We shall obtain the following results (see Theorem~\ref{main-thm-derived-length}, Corollary~\ref{Cor:Borel}, and Theorem~\ref{Tn-subgroup.thm}).
\begin{itemize}
\item {\textit{Any connected solvable subgroup of $\Aut(X)$  is contained in a  Borel subgroup}}  
{\textit{and its derived length is $\leq \dim(X)+1$}}.  
\item {\textit{Any Borel subgroup of $\Aut(\bbA^n)$ of derived length $n+1$ is conjugate to $\Jonq(n)$}}. 
\item  {\textit{If a Borel subgroup of $\Aut(\bbA^n)$ contains an $n$-dimensional torus, then it is conjugate to $\Jonq(n)$}}. 
\end{itemize} 

These results suggest interesting similarities between automorphism groups 
and linear algebraic groups, at least for  Borel subgroups, but there are also important differences, as shown in the next two items. 

\subsubsection{}
For the plane, the analogy is excellent: all Borel subgroups of $\Aut(\bbA^2)$ are conjugate to $\Jonq(2)$ (see Proposition \ref{Prop.Borel_dim2} below, and~\cite{Berest-Eshmatov-Eshmatov} when $\kk = \CC$). But the situation changes drastically in dimension $n\geq 3$.
In that case
 $\Aut(\An$) contains a Borel subgroup that is not conjugate to  $\Jonq(n)$ (see Corollary~1.5 of~\cite{ReUrSa2025Group-theoretical-} and 
Proposition~\ref{Prop.Exotic_Borels} below).
In characteristic $0$, we do not know whether the number of conjugacy classes of Borel subgroups of $\Aut(\bbA^n)$ is finite or infinite. But in positive characteristic, we shall use Gupta's counterexamples to the ``cancellation problem'' to show that $\Aut(\bbA^n)$ contains infinitely many conjugacy classes of Borel subgroups if~$n\geq 4$ (see Theorem~\ref{infmanyBorels}).

\subsubsection{} 
When $n\geq 2$, most elements of $\Aut(\bbA^n)$ are not algebraic, hence Borel subgroups do not cover $\Aut(\An)$. 
In fact,
when $X$ is affine, $\kk$  uncountable of characteristic zero, and $\Aut(X)$ connected, then {\emph{
 $\Aut(X)$ is covered by its Borel subgroups if and only if  all elements of $\Aut(X)$ are algebraic, if and only if $\Aut(X)$ coincides with a Borel subgroup}}. This follows from Theorem 1.1 in \cite{perepechko2024automorphism} and Theorem~\ref{main-thm-A}.

\subsection{Nested subgroups}
If a subgroup $\GGG$ of $\Aut(X)$ is \algebraicallyTame{}, it contains a lot of algebraic subgroups, but it is not clear whether it contains a countable
sequence $G_1 \subset G_2 \subset \cdots$ of algebraic subgroups covering~$\GGG$, i.e., whether $\GGG$ is nested. Conversely, if $\GGG$ is nested it is not obvious that every algebraically generated subgroup is algebraic. However, the following equivalence will be derived from Theorem~\ref{main-thm-A}, see Theorem~\ref{Thm.Equivalence_alg_tame_nested}.

\begin{mthm}\label{thm-nested-tame}
Let $X$ be an algebraic variety. A subgroup of $\Aut(X)$ is \algebraicallyTame{} if and only if it is nested by connected algebraic subgroups.
\end{mthm}

Our final application of Theorem~\ref{main-thm-A}, given in \S\,\ref{proof_perepechko_sec}, is a proof of one of the main results of~\cite{Pe2024Structure-of-conne}, answering positively \cite[Question 4]{KZ24}. 
To state it, recall that if $X$ is affine, then $\Aut(X)$ is an ind-group with a natural
topology (see \S\,\ref{ind-algebraic-vs-indgroups.sec} below).

\begin{mthm}\label{main-thm-Perepechko}
Assume that $\Char\kk = 0$. Let $X$ be an affine variety and let $\GGG \subset \Aut(X)$ be a connected nested subgroup. 
Then $\GGG$ is closed in $\Aut(X)$, and $\GGG = L \ltimes \UUU$, with 
$L$ a reductive algebraic 
subgroup  and $\UUU$ a closed nested unipotent subgroup of $\Aut(X)$.
\end{mthm}

The assumption $\Char(\kk)= 0$ is necessary as shown by Example \ref{exa.perepechko}.

\subsection{Organization of the paper} 
Theorem~\ref{main-thm-A} is proven by induction on the derived length of the group and the dimension of $X$. It relies on a version of Rosenlicht's theorem due to Popov (see~\cite{Po2014On-infinite-dimens, ReUrSa2025Group-theoretical-}) which is described in \S\,\ref{section:a_rosenlicht_theorem}. We also need a version of the Lie-Kolchin Theorem (see Theorem~\ref{Lie-Kolchin.thm} in \S\,\ref{lie-kolchin.par}). Then, Theorem~\ref{main-thm-A} is proven in \S\,\ref{quasi-affine.sec} and~\ref{maintheoremproof.sec}. 
Corollary~\ref{corB}, Theorems~\ref{main-thm-derived-length},~\ref{thm-nested-tame} and~\ref{main-thm-Perepechko}, and the main properties of Borel subgroups are obtained in \S\,\ref{applications.sec}. 

\subsection*{Notation} The characteristic of $\kk$ is arbitrary, except otherwise stated. 
If $X$ is a variety over $\kk$,  $\OOO_X$ is its structural sheaf of regular functions and $\OOO(X)$ the $\kk$-algebra of regular functions on $X$. If $X$ is irreducible, then $\kk(X)$ denotes its field of rational functions. 
If $G$ is a group acting by automorphisms on a ring~$R$, we let $R^G$ be the ring of invariants.
 
\subsection*{Acknowledgement} We thank Jérémy Blanc, 
David Bourqui, Michel Brion, Matthieu Romagny, and Christian Urech for interesting discussions.

\section{Ind-groups and Rosenlicht's theorem}\label{section:a_rosenlicht_theorem}
\vspace{0.1cm}
\begin{center}
\begin{minipage}{12.0cm}
{\sl{
We first discuss ind-groups and their properties, and then complete  \S\,\ref{par:intro_algebraic_families} with additional properties of $\Aut(X)$, in particular with Popov's version of Rosenlicht's theorem. 
}}
\end{minipage}
\end{center}

\subsection{Ind-groups} \label{ind-algebraic-vs-indgroups.sec}
An ind-group $\GGG$ is an ind-variety $\GGG = \bigcup_{k\geq 1} \GGG_k$ with a group structure such that multiplication and inverse are ind-morphisms (see \S\,1.1.1 and \S\,1.2.1 of \cite{FuKr2018On-the-geometry-of}). 
The $\GGG_k$ are algebraic varieties, but in general they are not subgroups. In this paper an ind-group $\GGG = \bigcup_{k\geq 1} \GGG_k$ will always be an {\it affine} ind-group which means that all $\GGG_k$ are affine varieties.

An ind-variety $\VVV = \bigcup_k \VVV_k$  
carries a natural topology in which a subset $S \subset \VVV$ is closed if all $S \cap \VVV_k \subset \VVV_k$ are closed subsets. A subset $ S \subset \VVV$ is \emph{algebraic} if it is locally closed and contained in some $\VVV_k$.

Let $\VVV = \bigcup_k\VVV_k$ and $\WWW = \bigcup_\ell\WWW_\ell$ be ind-varieties. A map $\phi\colon \VVV \to \WWW$ is a \textit{morphism} if for every $k$ there is an $\ell$ such that $\phi(\VVV_k) \subset \WWW_\ell$ and the induced map $\VVV_k \to \WWW_\ell$ is a morphism of varieties. 
An ascending filtration $\UUU_1 \subset \UUU_2 \subset  \UUU_3 \subset \cdots \subset \VVV$ of $\VVV$ by closed algebraic subsets is \textit{admissible} if each $\VVV_k$ is contained in some $\UUU_j$. This means that the identity is an isomorphism from the ind-variety $\VVV = \bigcup_k \VVV_k$ to the ind-variety $\VVV = \bigcup_j \UUU_j$
(cf. \S\,1.1 in ~\cite{FuKr2018On-the-geometry-of}).

\begin{lem}\label{connected-indgroups.lem}
For an ind-group $\GGG$ the following assertions are equivalent.
\begin{enumerate}[\rm (a)]
\item
$\GGG$ is connected with respect to the ind-topology.
\item
For every $g \in \GGG$ there is an irreducible closed algebraic subset $C\subset \GGG$ which contains $g$ and the neutral element $e_\GGG$ of $\GGG$.
\item
There is an admissible filtration $\GGG = \bigcup_k \GGG_k$ such that each $\GGG_k$ is irreducible and contains $e_\GGG$.
\end{enumerate}
\end{lem}
\begin{proof}
The equivalence of (a) and (b) follows from 
\cite[Proposition~2.2.1]{FuKr2018On-the-geometry-of}. 
It is clear that (c) implies (b). For the reverse implication we have to show that every closed algebraic subset $B \subset \GGG$ is contained in an irreducible closed algebraic subset containing $e_\GGG$. Decompose $B$ in its irreducible components $B_i$, $i=1,\ldots,k$,
and choose $b_i$ in $B_i$. By (b) there is an irreducible closed subvariety $C_i$ containing $e_\GGG$ and $b_i^{-1}$. Then $\overline{B_i C_i}$ is an irreducible algebraic subset containing $B_i$ and $e_\GGG$. Hence, 
$\overline{(B_1 C_1)(B_2C_2)\cdots(B_kC_k)}$ is irreducible and contains $B$ and $e_\GGG$.
\end{proof}

We denote by $\GGG^\circ$ the connected component of the identity in $\GGG$. It is a  closed subgroup of $\GGG$.

\begin{lem}
\label{alg-gen-subgroup.lem}
Let $G$ be an algebraic group and $A \subset G$ an irreducible constructible subset containing $e_G$. Then $H:=\langle A \rangle\subset G$ is a connected closed subgroup, and $H = (A \,A^{-1})^n$ for all $n \geq 2 \dim H$.
\end{lem}

\begin{proof}
The algebraic subgroup $\overline{H}\subset G$ is connected. The increasing sequence of irreducible closed subsets 
$\overline{A A^{-1}} \subset\overline{(A A^{-1})^2} \subset \overline{(A A^{-1})^3} \subset \cdots$ stabilizes, 
and thus $\overline{(A A^{-1})^d} = \overline{H}$
for
$d:=\dim \overline{H}$. It follows that $(AA^{-1})^{d} \subset \overline{H}$ is a dense constructible subset, and thus $(AA^{-1})^{2d} =  \overline{H}$, by Lemma~7.4 in \cite{Hu1975Linear-algebraic-g}.
\end{proof}

\subsection{Nested and strongly nested ind-groups}\label{strongly_nested.sec}
Let $\GGG = \bigcup_k\GGG_k$ be an ind-group. Following the definition from \S~\ref{par:intro_algebraic_families}, $\GGG$  is nested if there is a countable ascending filtration $G_1 \subset G_2 \subset \cdots \subset \GGG$ by algebraic groups. 
This does not imply that  this filtration is admissible. E.g., the algebraic group 
$\SL_2(\overline{\FF_p})$ is nested by the finite subgroups $\SL_2(\FF_{p^k})$. 
We will call an ind-group \textit{strongly nested} if it admits an admissible filtration by closed algebraic subgroups. Note that for an uncountable $\kk$ an ascending filtration by algebraic subgroups is always admissible
(Lemma 1.3.1 in \cite{FuKr2018On-the-geometry-of}). In particular, a nested ind-group is strongly nested in this case. 

If $X$ is affine,  $\Aut(X)$ has a natural structure of an ind-group such that morphisms $A \to \Aut(X)$ in the sense of Definition~\ref{Def.morphism} correspond to
morphisms of ind-varieties (see \cite[Theorem~5.1.1]{FuKr2018On-the-geometry-of}).
A closed subgroup $\GGG$ of $\Aut(X)$ carries an ind-group structure coming from 
$\Aut(X)$, and $\GGG$ is connected with respect to this structure if and only if 
it  is connected in the sense of Definition~\ref{defi:introduction_connected}.
(We use the equivalence of (a) and (b) in Lemma~\ref{connected-indgroups.lem}.) 

Corollary~\ref{nested-strongly-nested.cor} will  show that a closed subgroup of $\Aut(X)$ which is nested by connected algebraic groups is strongly nested and \algebraicallyTame{}.

\subsection{Connected component}\label{connected_subgroups.sec}
Let $X$ be a variety and let $\HHH \subseteq \Aut(X)$ be a subgroup.
We denote by $\HHH^\circ \subseteq \HHH$ the subgroup generated by the images of all
morphisms $\rho \colon A \to \Aut(X)$ such that $A$ is irreducible and $\id_X \in \rho(A) \subseteq \HHH$;
 we call $\HHH^\circ$ the \emph{connected component of the identity in $\HHH$}.
If $X$ is affine and $\HHH$ is a closed subgroup of the ind-group $\Aut(X)$, this definition agrees with the usual notion of connected component of the identity introduced above in \S\,\ref{ind-algebraic-vs-indgroups.sec}.

\subsection{Orbits}\label{orbits-notation.sec}
Let $\rho\colon A\to \Aut(X)$ be a morphism. 
For simplicity, for $(a,x)\in A\times X$, we write $ax$ instead of $\rho(a)(x)$. The orbit of a point $x\in X$ under $A$~is 
\begin{equation}
Ax :=\{ax\mid a\in A\} \subset X.
\end{equation}
If $Y$ is a subset of $X$ and $a$ is an element of $A$ we set $aY:=\{ay\mid y\in Y\}$ and $AY:=\{ay\mid a\in A, y\in Y\}$. If $Y$ is constructible, then so are $aY$ and $AY$, by Chevalley's theorem. A subset $Y \subset X$ is \textit{stable under $A$} if $aY = Y$ for all $a \in A$. 

We use similar notation for actions of connected subgroups. 
Let $\GGG\subset \Aut(X)$ be such a connected subgroup. 
Proposition~\ref{Orbits.prop}(1) below implies that each orbit
$\GGG x$ is locally closed in $X$, 
and so $\dim \GGG x$ makes sense. 
We denote by $\mdo(\GGG;X)$ the  {\emph{maximal dimension of $\GGG$-orbits}} in $X$:
\begin{equation*}
\mdo(\GGG;X)\coloneqq \max\{\dim \GGG x\mid x \in X\}.
\end{equation*}
Then  $\Xm\subset X$ will be the union  of orbits of dimension $\mdo(\GGG ; X)$.
The following result can be found in Proposition~2.4 of \cite{Note_on_Ramanujam}.

\begin{prop}\label{Orbits.prop}
Let $\GGG \subset \Aut(X)$ be a connected  
subgroup.
\begin{enumerate}[\rm (1)]
\item
The $\GGG$-orbits in $X$ are open in their closure.
\item \label{algebraic-subset.item}
The subgroup $\GGG$ contains an irreducible constructible subset $V \subseteq \Aut(X)$
with $\id_X \in V$
such that $\GGG x = Vx$ for all $x \in X$. 
\item
The union $\Xm$ of orbits of maximal dimension is open in $X$.
\end{enumerate}
\end{prop}

\subsection{Algebraic subgroups}\label{algebraic_subgroups_and_basic_facts.sec}
We will constantly use the following facts, valid for any morphism $\rho\colon A \to \Aut(X)$ (see \cite[p. 26]{Ra1964A-note-on-automorp} and \cite[Prop.~3.3.2]{FuKr2018On-the-geometry-of}).

\begin{enumerate}[\rm (1)]
\item 
If the subvariety $Y \subset X$ is stable under $A$, the induced map $\rho_Y\colon A \to \Aut(Y)$ is a morphism.
\item 
The map $\rho^k\colon A^k \to \Aut(X)$, $(a_1,\ldots,a_k)\mapsto \rho(a_1)\circ\cdots\circ\rho(a_k)$, is a morphism, and  if  $\id_X \in \rho(A)$, then $\rho^k(A^k) \subset \rho^\ell(A^\ell)$ for any $k\leq \ell$.
\item\label{item2.5.3}
The map $(a,x)\mapsto (a,\rho(a)(x))$ is an automorphism of $A\times X$. In particular, the inverse family $A \ni a\mapsto \rho(a)^{-1} \in \Aut(X)$ is a morphism.
\end{enumerate}
Property~(3) is mentioned in \cite{Ra1964A-note-on-automorp}. 
A short proof can be found in the Appendix of \cite{Note_on_Ramanujam}.

Recall that \hyperlink{Ramanujam.thm}{Ramanujam's theorem} was used
to define the notion of algebraic subgroups of $\Aut(X)$ in \S\,\ref{par:intro_algebraic_families}.

\begin{lem}\label{lem.intersection-of-constructible}
If $C$ and $D$ are constructible subsets of $\Aut(X)$, then $C\cup D$ and $C\cap D$ are constructible as well. Thus, the intersection $G \cap H$ of two algebraic subgroups $G$, $H \subset \Aut(X)$ is an algebraic subgroup, closed in $G$ and in $H$.
\end{lem}
\begin{proof}
Write $C = \rho(A)$ and $D = \mu(B)$ for some morphisms $\rho\colon A \to \Aut(X)$ and $\mu\colon B \to \Aut(X)$. Then $C \cup D$ is the image of $A \cup B$, hence constructible. 
For $C\cap D$
consider the closed subset 
$Z \coloneqq \set{(a,b)}{\rho(a)(x) = \mu(b)(x) \ \forall x \in X}$ of $A \times B$.
Projecting $Z$ to $A$, we get a constructible subset
\[
\pr_A(Z)=\set{a\in A}{\rho(a)\in C\cap D} = \rho^{-1}(C\cap D)
\]
of $A$; hence $C \cap D = \rho(\pr_A(Z))$ is constructible as well.
The last claim follows from  the fact that a constructible subgroup of an algebraic group is  closed (see~Proposition~1.3(c) in \cite{Bo1991Linear-algebraic-g}).
\end{proof}

The proof above shows that the preimage of a constructible subset of $\Aut(X)$
under a morphism is constructible.

\begin{prop}\label{commutative-subgroup.prop}
Let $\GGG \subset \Aut(X)$ be a connected commutative subgroup
with a dense orbit in $X$. Then $\GGG$ is an algebraic subgroup and  $\dim(\GGG) = \dim(X)$.
\end{prop}

\begin{proof} Proposition~\ref{Orbits.prop} provides  a 
constructible subset $V \subseteq \Aut(X)$,  $V \subseteq \GGG$, such that
$Vx = \GGG x \subset X$ is open and dense for some $x \in X$. Since $\GGG$ is commutative, the map 
$\GGG \to \GGG x$, $g \mapsto gx$ is bijective. Indeed, if $gx=g'x$, then $ghx=g'hx$ for all $h\in \GGG$, hence $g=g'$. Thus, $V = \GGG$, and Corollary~1.3 of \cite{Note_on_Ramanujam} shows that $\GGG$ is an algebraic subgroup of $\Aut(X)$. Moreover $\dim(\GGG) =\dim(X)$, because $\GGG \to X$, $g \mapsto gx$, is an injective and dominant morphism.
\end{proof}

\begin{lem}\label{alg-generated.lem}
Let $\HHH\subset  \Aut(X)$ be a connected subgroup, 
$H$ an algebraic group, and $\phi \colon \HHH \to H$ a homomorphism 
of groups such that for each morphism $\mu\colon B \to \HHH$, the composition $\phi \circ \mu \colon B \to H$ is a morphism of varieties. 
If $\phi$ is injective, then $\HHH$ is a connected algebraic group of dimension $\leq \dim H$.
\end{lem}

\begin{proof}
For any injective morphism $\mu\colon B \to \GGG$, the composition $\phi\circ\mu\colon B \to H$ is injective, hence $\dim B \leq \dim H$. Thus, $\dim\GGG \leq \dim H$ and the claim follows from \hyperlink{Ramanujam.thm}{Ramanujam's theorem}.
\end{proof}
 
\begin{lem}\label{Lem.restriction-to-open}
Let $X_0 \subset X$ be locally closed and $\rho\colon A \to \Aut(X)$ a morphism.
\begin{enumerate}[\rm (1)]
\item \label{Lem.restriction-to-open1}
The subset $A_0:=\{a\in A \mid \rho(a)(X_0) = X_0\}$ is closed in $A$, and the induced map $\rho_0\colon A_0 \to \Aut(X_0)$ is a morphism.
\item \label{Lem.restriction-to-open2}
Assume that $X_0$ is dense in $X$. Let $\HHH \subset \Aut(X)$ be a connected subgroup stabilizing $X_0$. If the image $\HHH'$ of $\HHH$ in $\Aut(X_0)$ is contained in an algebraic group $H$, then $\HHH$ is an algebraic group of dimension $\leq \dim H$.
In particular, if $\HHH'$ is \algebraicallyTame{}, then  so is $\HHH$.
\end{enumerate}
\end{lem}

\begin{proof}
(1) 
The first statement follows from the observation that an element $a \in A$ lies in $A_0$ if and only if $\rho(a)$ and $\rho(a)^{-1}$ preserve the closed subsets $\overline{X_0}$ and $\overline{X_0} \setminus X_0$. This is a closed condition for $a$.
The second claim follows from property (1) listed at the beginning of 
the present \S\,\ref{algebraic_subgroups_and_basic_facts.sec}.
\ps
(2)
Apply Lemma~\ref{alg-generated.lem} to the homomorphism $\HHH \to H$.
\end{proof}

\subsection{Rosenlicht's decomposition and the unipotent radical} \label{Rosdecompandunipelem}
Unipotent elements of affine algebraic groups are discussed in \S 15 of 
\cite{Hu1975Linear-algebraic-g}. An algebraic group is \textit{unipotent} if it is affine and consists of unipotent elements. In characteristic zero a unipotent group is connected. In characteristic $p>0$ an element of an affine algebraic group is unipotent if and only if its order is finite and a power of $p$.

A variety is \textit{anti-affine} if $\OOO(Z)=\kk$.
Rosenlicht's decomposition theorem (see Theorem~1.2.4 in 
\cite{BrSaUm2013Lectures-on-the-st})
says that any connected algebraic group $G$ is a product $G = \Gaff \Gant$ where $\Gaff \subset G$ is the largest connected, affine and normal subgroup of $G$, and $\Gant\subset G$ is its largest anti-affine subgroup. 
The quotient $G/\Gaff$ is an abelian variety, and every connected affine subgroup is contained in $\Gaff$ (Proposition~3.1.1(i) in \cite{BrSaUm2013Lectures-on-the-st}).
The \emph{unipotent radical} $R_u(G)$ is  the largest closed, normal, connected and unipotent subgroup of $G$. Since unipotent groups are affine we get
$R_u(G) = R_u(\Gaff) \subset \Gaff$. 

\begin{rem}\label{unipotent-group.rem}
\textit{If $U$ is an algebraic group such that every element $u\in U$ is contained in some unipotent subgroup $U_u$, then $U$ is a unipotent group.}
Indeed, we show that the abelian variety $U^\circ/\Uaff^\circ$ 
is trivial, which gives that $U$ is affine. 
In characteristic zero $U_u$ is connected, hence contained in 
$\Uaff^\circ$, so $U^\circ = \Uaff^\circ$. 
In characteristic $p>0$ we use the fact that the order of every element of $U$ is a power of $p$, but a non-trivial abelian variety contains elements of order
prime to $p$.
\end{rem}

\begin{exa} \label{example_unipotent_char_p.exa}
Suppose $\Char(\kk)=p>0$. 
Let $E$ be an elliptic curve with  $p$-torsion $E[p] \simeq \ZZ/p\ZZ$. 
Define $G=(E\times \Ga)/(\ZZ/p\ZZ)$ where $\ZZ/p\ZZ$ is embedded diagonally.
This group is commutative and connected, $R_u(G)$ is the projection of $\{ 0\}\times \Ga$, and $R_u(G)$ intersects the projection of $E\times \{0\}$ in a subgroup $F_1\simeq \ZZ/p\ZZ$. Thus, $F_1$ is contained in a connected unipotent subgroup as well as in an elliptic curve. If  $z\in E$ is an element of order $p^r$ for some $r>1$, the finite group $F_r\subset G$ generated by $z$ is central, is not contained in $R_u(G)$, and is entirely made of elements of order a power of $p$.
\end{exa}

\begin{exa}\label{mostow_theorem.exa}
Let $G$ be an affine algebraic group. 
If $\Char(\kk) = 0$,  by a theorem of Mostow, $G$ is a semi-direct product $G = L \ltimes R_u(G)$ where $L$ is a maximal reductive subgroup and $R_u(G)$ is the unipotent radical of $G$ (see~\cite{Mo1956Fully-reducible-su}). Such an $L$ is called
a \emph{Levi subgroup}, and any two Levi subgroups are conjugate. 
This  decomposition no longer exists when $\Char(\kk) > 0$ (see~\cite[Prop.~A.6.4]{CoGaPr2010Pseudo-reductive-g}).
However, if $G$ is solvable and connected, then,  in any characteristic, 
$G = T \ltimes R_u(G)$ where $T \subseteq G$ is any torus of maximal dimension
(see \cite[\S19.3]{Hu1975Linear-algebraic-g}), and $R_u(G)$ consists of all
unipotent elements in $G$.
\end{exa}

\begin{lem}\label{quasiaffine-gives-affine.lem}
Let $G$ be an algebraic group acting faithfully on a quasi-affine variety $X$. Then $G$ is an affine algebraic group.
\end{lem}

\begin{proof} 
We may assume $G$ connected. 
Since $X$ is quasi-affine, $\OOO(X)$ separates points on $X$.
As every regular function on $\Gant$ is constant, each  $\Gant$-orbit is reduced to a point,
and so $\Gant$ is trivial, because the action is faithful.
\end{proof}

For $X$ quasi-affine, we say that $g\in \Aut(X)$ is {\emph{unipotent}} if it is
contained in a unipotent algebraic subgroup of $\Aut(X)$. 
A subgroup is unipotent if all its elements are unipotent. Since $X$ is quasi-affine, Lemma~\ref{quasiaffine-gives-affine.lem} and 
Lemma~\ref{lem.intersection-of-constructible} show that an
element in $\Aut(X)$ is unipotent if and only if it is algebraic and unipotent
in \emph{any} algebraic subgroup of $\Aut(X)$ containing it (cf. \S~\ref{connected_solvable_introduction.sec}).

With Example~\ref{example_unipotent_char_p.exa} in mind, we shall not define the concept of unipotent element (or subgroup) of $\Aut(X)$ when $X$ is not quasi-affine.

\subsection{Geometric quotient and Popov's theorem} 

\begin{defn} \label{geometric-quotient.def} 
Let $\GGG \subset \Aut(X)$ be a subgroup.
A morphism $\pi\colon X \to Y$ of varieties 
is  a \emph{geometric quotient} for $\GGG$ if the following properties hold:
\begin{enumerate}[\rm (i)]
\item
The fibers of $\pi$ are the $\GGG$-orbits. In particular $\pi$ is surjective.
\item
A subset $U \subset Y$ is open if and only if $\pi^{-1}(U) \subset X$ is open.
\item
For any open subset $U \subset Y$ the pull-back $\pi^*$ induces an isomorphism $\OOO_Y(U)\simto\OOO_X(\pi^{-1}(U))^\GGG$.
\end{enumerate}
\end{defn}

A geometric quotient satisfies the following universal property:
If $f \colon X \to Z$ is $\GGG$-invariant, then there exists a unique
factorization through $\pi \colon X \to Y$. As a consequence, 
if $A \to \Aut(X)$ is a morphism such that
its image normalizes $\GGG$, then it descends to a morphism $A \to \Aut(Y)$.

\begin{rem}
  \label{Rem.Geom_quotient_smooth}
\textit{If $\pi \colon X \to Y$  is a geometric quotient for a subgroup $\GGG \subseteq \Aut(X)$, then $\pi$ is smooth over some open dense subvariety of $Y$.} 
Indeed, we may assume that $Y$ is irreducible and $\GGG$ acts transitively on the set of irreducible components of $X$.
For an irreducible component $Z \subseteq X$ denote by $\GGG_Z$ the subgroup of elements in $\GGG$ that preserve $Z$. 
Then $\kk(Y) = \kk(Z)^{\GGG_Z}$, and $\kk(Z)/\kk(Z)^{\GGG_Z}$  is separable, by Lemma 1.5(ii) in \cite[\S\,IV]{Sp1989Aktionen-reduktive}. Hence, there is a dense open subset $Z_0 \subseteq Z$ such that $\pi|_{{Z_0}}$ is smooth (\cite[Proposition~18.79]{GoWe0Algebraic-geometry-II}).
It follows that $\pi$ is smooth on 
  $\GGG Z_0 = \pi^{-1}(\pi(Z_0))$, which proves the claim.
\end{rem}

With this classical notion in mind, Theorem~3 of \cite{Po2014On-infinite-dimens} gives the following version of Rosenlicht's theorem. 

\begin{thm}\label{popov-rosenlicht.thm}
Let $X$ be a 
variety, and let $\GGG\subset\Aut(X)$ be a connected subgroup.  
Then there is a  dense open and $\GGG$-stable subset $X' \subset X$ which admits  a geometric quotient $\pi\colon X' \to X'/\GGG$.  This subset $X'$ can be taken as a subset of $\bigcup_i X_{i, \text{\it max}}$ where the $X_i$ are the irreducible components of $X$.  
If $X$ is irreducible and $\kk(X)^\GGG = \kk$, then $\GGG$ acts with a dense
open orbit on $X$.
\end{thm}

\begin{cor}
\label{Cor.global_geom_quotient}
Let $\GGG, \HHH$ be connected subgroups of $\Aut(X)$ such that 
$\GGG$ normalizes $\HHH$. If $\GGG X' = X$ for some open subset $X' \subseteq X$
admitting a geometric quotient for $\HHH$, then there is a geometric  quotient $\pi \colon X \to Y$ for $\HHH$.
\end{cor}

\begin{proof}
Let $\pi' \colon X' \to Y'$ be a geometric quotient for $\HHH$.
If $g \in \GGG$, the composition
of $g^{-1}X' \to X'$, $x \mapsto gx$ with $\pi'$ is a geometric
quotient for  $\HHH$ on $g^{-1}X'$. Gluing these quotients for 
$g \in \GGG$ gives a geometric quotient $\pi \colon X \to Y$ for $\HHH$.
The image $\KKK \subseteq \Aut(Y)$ of $\GGG$ is connected and satisfies $\KKK Y' = Y$, 
hence $Y$ is a variety by Lemma~\ref{Prevariety.lem} below. 
\end{proof}

\begin{lem}
\label{Prevariety.lem}
If $Y$ is a prevariety such that $\Aut(Y)^\circ U = Y$ for some open affine subset $U \subseteq Y$, then $Y$ is a variety. 
\end{lem}

\begin{proof}
Let $\Delta\subset Y\times Y$ be the diagonal.  
Since $U$ is separated, $\Delta\cap (U\times U)$ is closed in $U\times U$. 
Thus, it is enough to show that the subsets $\varphi(U)\times \varphi(U)$ with $\varphi\in \Aut(Y)$ cover $Y\times Y$. 
For this, we fix a pair $(y_1, y_2)$ in $Y\times Y$.
By assumption, there are $\varphi_1, \varphi_2 \in \Aut(Y)^\circ$ such that $\varphi_i(y_i)\in U$ for $i=1,2$. Let $\rho\colon A\to \Aut(Y)^\circ$ be an irreducible algebraic family, the image of which contains $\varphi_1$ and $\varphi_2$. Then, $A_i:=\{ a\in A \; ; \; \rho(a)(y_i)\in U\}$ is open and non-empty, hence open and dense in $A$. Thus, $A_1\cap A_2$ is non-empty, and for any $a\in A_1\cap A_2$ the automorphism $\rho(a)$ maps simultaneously $y_1$ and $y_2$ in $U$.  
\end{proof}

\begin{rem}
    \label{Rem.global_geom_quotient}
    The proof of Corollary~\ref{Cor.global_geom_quotient} above shows the following:
    If the geometric quotient $X' \to X' / \HHH$ satsifies a property (P)
    which is local on the target (e.g. "affine", "smooth", $\ldots$), then $X \to X / \HHH$ satisfies (P) as well.
\end{rem}

\section{Rational points and nested ind-groups}\label{section:rational_points_nested}
\vspace{0.1cm}
\begin{center}
\begin{minipage}{12.0cm}
{\sl We discuss rational points of algebraic groups, describe examples of 
nested ind-groups, and give a version of the Lie-Kolchin theorem for subgroups of $\Aut(X)$. By convention, $\kk$-algebras will be {\emph{commutative}}.}
\end{minipage}
\end{center}

\subsection{Rational points of algebraic groups} 
Consider the group $\GL_2$ and the algebra $R:=\kk[u]$. The group $\GL_2(R)$ acts on the affine space $\bbA^3$ by automorphisms  of type
$(x, y, z)\mapsto (x, a(x)y+b(x)z, c(x)y+d(x)z)$,
where $a, b, c, d$ are elements of $\kk[u]$ such that $ad-bc\in \kk^\times$. 
The projection on the first coordinate is $\GL_2(R)$-invariant, and the action on each fiber coincides with the linear action of $\GL_2(\kk)$ on $\bbA^2$.

This can be generalized to any linear group $G$. Firstly, as shown in Proposition~2.5.1 of~\cite{FuKr2018On-the-geometry-of}, if $R$ is a $\kk$-algebra of countable dimension, then the group $G(R)$ of $R$-rational points of $G$ is naturally an ind-group. Secondly, this structure of ind-group is compatible with morphisms in the following sense (see \cite[Proposition 2.5.3]{FuKr2018On-the-geometry-of}).

\begin{prop}\label{functoriality-of-rational-points.prop}
Let $G$ and $H$ be linear algebraic groups. Let $R$ and $S$ be  $\kk$-algebras of countable dimension.
\begin{enumerate}[\rm (1)]
\item
Let $\phi\colon G \to H$ be a homomorphism of  algebraic groups. Then  the induced homomorphism
$
\phi_R \colon G(R) \to H(R)
$
is a homomorphism of ind-groups. If $\phi$ is injective, then $\phi_R$ is a closed immersion.
\item
Let $\psi\colon R \to S$ be
a homomorphism of $\kk$-algebras. 
Then the induced homomorphism $\psi_G\colon G(R) \to G(S)$ is a homomorphism of ind-groups. If $\psi$ is injective, then $\psi_G$ is a closed immersion.
\end{enumerate}
\end{prop}

\subsection{Endomorphisms of modules}\label{embedding-into-GL.subsec}
Let $R$ be a $\kk$-algebra of countable dimension, and let $M$ be a finitely generated $R$-module. Then $M$ and $\End_R(M)$ are $\kk$-vector spaces of countable dimension and thus carry a canonical structure of an affine ind-variety. Since the composition
$\End_R(M) \times \End_R(M) \to \End_R(M)$ is $\kk$-bilinear, it is a morphism of ind-varieties, and so $\End_R(M)$ is an ind-monoid.
The embedding $\Aut_R(M) \into \End_R(M) \times \End_R(M)$, $\phi\mapsto (\phi,\phi^{-1})$, identifies $\Aut_R(M)$ with the closed subset 
\begin{equation}
\label{eq:embedded_groups}
\PPP\coloneqq \{(\alpha,\beta) \in \End_R(M) \times \End_R(M) \mid \alpha\beta = \beta\alpha=\id_M\}
\end{equation} 
of $\End_R(M) \times \End_R(M)$. Since the composition  law is induced by the ind-morphism $((\alpha,\beta),(\alpha',\beta')) \mapsto (\alpha\alpha', \beta'\beta)$ and the inverse by $(\alpha,\beta)\mapsto (\beta,\alpha)$, both are ind-morphisms, and so $\Aut_R(M)$ is an ind-group.

\begin{exa}
\label{EndRM_free_M.exa}
Assume $M$ is a free $R$-module of rank $n < \infty$. \textit{Then $\Aut_R(M)$
and $\GL_n(R)$ are isomorphic as ind-groups.}
In fact, choosing a basis of $M$, Proposition~1.10.1 of~\cite{FuKr2018On-the-geometry-of} shows that the map $\GL_n(R) \to \End_n(R)\times\End_n(R)$, $g \mapsto (g,g^{-1})$ is a closed immersion. 
\end{exa}

\begin{lem}
\label{Lem:UP_Endos}
The ind-monoid $\End_R(M)$ satisfies the following universal property:
A map $\phi\colon Z \to \End_R(M)$ where $Z$ is an algebraic variety is a morphism of ind-varieties if and only if the map $\Phi\colon Z \times M \to M$, $\Phi(z,m)= \phi(z)(m)$, is a morphism of ind-varieties.
\end{lem}
Recall that $\Phi\colon Z \times M \to M$ is a morphism of ind-varieties if for any finite-dimensional $\kk$-subspace $V \subset M$ there is a finite-dimensional $\kk$-subspace $W \subset M$ such that $\Phi(Z \times V) \subset W$ and the induced map $Z \times V \to W$ is a morphism of varieties.

\begin{proof}
The ind-structure on $\End_R(M)$ is uniquely determined by the universal property. 
We have to see that the structure defined above has this property.
\ps
(a) If $\phi\colon Z \to \End_R(M)$ is an ind-morphism, there is a finite dimensional subspace $E \subset \End_R(M)$ such that $\phi(Z) \subset E$ and $\phi\colon Z \to E$ is a morphism. Let $V \subset M$ be a finite dimensional subspace. 
Then $W\coloneqq  \langle E(V) \rangle_{\kk} \subset M$ 
is finite dimensional, $\Phi(Z\times V) \subset W$, and we have a canonical $\kk$-linear map $E \to \Hom_\kk(V,W)$. Thus $Z \to \Hom_\kk(V,W)$ is a morphism and so  $\Phi\colon Z \times V \to W$ is a morphism as well.
\ps
(b) Now assume that $\Phi\colon Z \times M \to M$ is an ind-morphism.
For any finite-dimensional $V \subset M$ there is a finite-dimensional $W \subset M$ such that $\Phi(Z \times V) \subset W$ and $\Phi\colon Z\times V \to W$ is a morphism. Then, $\phi(Z)$ is contained in the  linear subspace $E\coloneqq \{\alpha\in\End_R(M)\mid\alpha(V) \subset W\}$, and we get a linear map $E \to \Hom_\kk(V,W)$ which is injective in  case $V$ generates $M$ as an $R$-module.
By assumption, the maps $\Phi(z,\cdot)\colon M \to M$ are $R$-linear, hence $\kk$-linear, and so the induced map $Z \to \Hom_\kk(V,W)$ is a morphism. Hence, $\phi\colon Z \to \End_R(M)$ is an ind-morphism.
\end{proof}

\begin{prop}\label{embedding-into-GLR.lem}
Let $A$ be a variety and  $A \to \Aut(X)$ a morphism.
\begin{enumerate}[\rm  (1)]
\item[\rm (1)] The induced map $\phi^*\colon A \times\OOO(X) \to \OOO(X)$, 
$(a,f)\mapsto a^* f$ is a morphism of ind-varieties. 
\end{enumerate}
 Set $R= \OOO(X)^A$, and let $M \subset \OOO(X)$ be an $A$-stable finitely generated $R$-module.
\begin{enumerate}[\rm  (1)]
\item[\rm (2)]  The induced map $\rho\colon A \to \Aut_R(M)$, $g \mapsto (g^*)^{-1}|_M$ is a 
morphism of ind-varieties.
\end{enumerate}
\end{prop}
\begin{proof}
The ring $\OOO(X)$ is an $R$-module,
and the homomorphisms $a^* \colon \OOO(X) \to \OOO(X)$ are $R$-linear for $a \in A$.
\ps
(a)
Let $\phi\colon A\times X \to X$ be the morphism induced by $A \to \Aut(X)$. 
The comorphism $\phi^*\colon \OOO(X) \to \OOO(A \times X)=\OOO(A)\otimes\OOO(X)$ has the following property:  if $f \in \OOO(X)$ and $\phi^*(f) = \sum_i h_i\otimes f_i$, then $a^*f = \sum_i h_i(a)f_i$ for $a \in A$. It follows that for any finite-dimensional subspace $V \subset \OOO(X)$ there is a finite-dimensional subspace $W$ 
such that $a^*V \subset W$ for all $a \in A$. 
For the first statement, it remains to see that the induced map $\rho\colon A \to \Hom_\kk(V,W)$ is a morphism. 
For this we choose a $\kk$-basis $(f_i)_{i\in I}$ of $\OOO(X)$ containing basis of $V$ and $W$. Then $a^* f_j = \sum_i h_{j,i}(a) f_i$, and so the matrix coefficients of $\rho(a)$ are regular functions on $A$.
\ps
(b) We get two maps $\rho_1,\rho_2\colon A \to \End_R(M)$, with $\rho_1(g)\coloneqq g^*|_M$ and $\rho_2(g)\coloneqq (g^{-1})^*|_M$. By  (a) and Lemma~\ref{Lem:UP_Endos} applied to $M$, both are ind-morphisms. Thus, by definition of the ind-structure on $\Aut_R(M)$,  $\rho\colon A \to \Aut_R(M)$ is a morphism of ind-varieties.
\end{proof}

The next result generalizes Proposition~\ref{functoriality-of-rational-points.prop}(2).
\begin{prop}\label{Hom-of-GLM.prop}
Let $R$ be a $\kk$-algebra of countable dimension, and let $S\supseteq R$ be a $\kk$-algebra which is finitely generated over $R$. Let $M$ be a finitely generated $R$-module and set $M_S:=M \otimes_R S$. Then the canonical map $\rho\colon\Aut_R(M) \to \Aut_S(M_S)$ is a homomorphism of ind-groups. If the natural map $M \to M_S$ is injective, then $\rho$ is a closed immersion.
\end{prop}
\begin{proof}
The map $\rho\colon \End_R(M) \to \End_S(M_S)$ is $\kk$-linear, hence a homomorphism of ind-monoids, and the map 
$\rho\times\rho\colon \End_R(M)\times\End_R(M) \to \End_S(M_S) \times \End_S(M_S)
$ sends the group $\Aut_R(M)$ to $\Aut_S(M_S)$ (where these groups are embedded as in Equation~\eqref{eq:embedded_groups}). This proves the first claim.
If the linear map $M \to M_S$ is injective,  the linear map $\End_R(M) \to \End_S(M_S)$ is injective as well and thus a closed immersion. The second claim follows.
\end{proof}

Let $R$ be a finitely generated  $\kk$-domain and let $\KK$ be an algebraic closure of the field of fractions
$\Quot(R)$. Let $M$ be a torsion-free $R$-module of finite rank $n$. Set $M_\KK=M\otimes_R\KK$, so that we get a canonical inclusion $\Aut_R(M) \subset \Aut_\KK(M_\KK) \simeq \GL_n(\KK)$.

\begin{lem}\label{connected-closure.lem}
If $\GGG \subset \Aut_R(M)$ is a connected subgroup, then  its closure  in $\Aut_\KK(M_\KK)$ is a connected algebraic $\KK$-group.
\end{lem}
\begin{proof}
The closure is an algebraic $\KK$-group. We have to show that it is connected.
Let $A$ be a variety and let $\rho\colon A \to \Aut_R(M)$ be a morphism. The $\kk$-morphism $\phi\colon A \times M \to M$ (Lemma~\ref{Lem:UP_Endos}) induces a $\KK$-morphism $\phi_\KK \colon A_\KK \times (M\otimes_\kk \KK) \to M \otimes_\kk \KK$, and thus a $\KK$-morphism $\bar\phi_\KK \colon A_\KK \times M_\KK \to M_\KK$. It follows that the $\rho$ defines a $\KK$-morphism $\rho_\KK \colon A_\KK \to\Aut_\KK(M_\KK)$, and so $\overline{\rho_\KK(A_\KK) } = \overline{\rho_\KK(A) } \subset \overline{\GGG} \subset \Aut_\KK(M_\KK)$. If $A$ is irreducible and $\id_M\in \rho(A)$, then $\rho(A) \subset \rho_\KK(A_\KK) \subset \overline{\GGG}^\circ$, and so $\GGG\subset\overline{\GGG}^\circ$, because $\GGG$ is connected.
\end{proof}

\subsection{Nested and strongly nested ind-groups}\label{nested.sec}
Recall from \S\,\ref{strongly_nested.sec} 
that an ind-group $\GGG$ is strongly nested if it admits an admissible
ascending filtration $\GGG = \bigcup_k \GGG_k$ such that all $\GGG_k$ are algebraic groups.

\begin{exas}\label{nested.exa}
\begin{enumerate}[\rm (1)]
\item 
\textit{If $\GGG$ is a nested (reps.\ strongly nested) ind-group,  then so is every
closed subgroup $\HHH \subset \GGG$. } 
\item
\textit{If $V$ is a $\kk$-vector space of countable dimension,  then the additive group $(V,+)$ is a strongly nested ind-group. We shall denote this group by $V^+$. }
\item\label{additive.item}
\textit{If $\Ga := (\bbA^1,+)$ is the additive group and $R$ is a $\kk$-algebra of countable dimension,  then the ind-group $\Ga(R) = (R,+)$ is strongly nested.}
\item\label{mult.item}
\textit{If $\Gm = (\bbA^1)^\times:=(\bbA^1\setminus\{0\},\times )$ is the multiplicative group and $R$ is a $\kk$-algebra, then $\Gm(R) = R^\times$,  the group of invertible elements of~$R$.} 
\end{enumerate}
\noindent
This is an ind-group in case $R$ has countable dimension, namely $R^\times=\GL_1(R)$. If $R = \OOO(Y)$ where $Y$ is an irreducible affine variety, Rosenlicht proved that $\OOO(Y)^\times/\kk^{\times}$ is a finitely generated torsion free abelian group 
(cf. \cite[Proposition~4.4.1]{FuKr2018On-the-geometry-of}).
This implies the following.
\begin{enumerate}[\rm (5)]
\item
\textit{For any variety $Y$ the connected component $(\OOO(Y)^\times)^\circ$ is isomorphic to $(\kst)^d$ where $d$ is the number of connected components of~$Y$.}
\end{enumerate}
\noindent
Indeed, if $Y$ is affine we apply Rosenlicht's result mentioned 
above,
and for the general case we choose an affine  
open and dense subset $U\subset Y$ to get a closed embedding 
$\OOO(Y)^\times \subset \OOO(U)^\times$ (see Proposition~\ref{functoriality-of-rational-points.prop}(2)).
\begin{enumerate}[\rm (6)]
\item\label{solvable.item}
\textit{Consider the group $\B_n \subset \GL_n$ of upper triangular matrices.
If $Y$ is any variety, then the closed ind-subgroup $B_n(\OOO(Y))^\circ \subset \GL(\OOO(Y))$ is connected and strongly nested.}
\end{enumerate}
\noindent
In fact, $B_n$ is a semi-direct product $B_n = T_n\ltimes U_n$ where $T_n$ is the diagonal torus and $U_n$ is the group of upper triangular unipotent matrices. By 
(5) and Lemma~\ref{semidirect-product.lem}(\ref{algebraic.item})
below it suffices to show that the ind-group $U_n(\OOO(Y))$ is strongly nested.  
Since $U_n$ has a chain $U_n(1) \supset U_n(2) \supset \cdots \supset \{e\}$ 
of closed normal subgroups such that $U_n(i) \to U_n(i) / U_n(i+1) \simeq \Ga$ splits
\cite[Exercise 7. \S17]{Hu1975Linear-algebraic-g},  this follows from  (\ref{additive.item}) 
and Lemma~\ref{semidirect-product.lem}(\ref{retraction.item}) below,
by induction on the dimension of $U_n$.
\end{exas}

\begin{lem}\label{semidirect-product.lem}
Let $\NNN, \HHH \subset \GGG$ be closed ind-subgroups of an ind-group $\GGG$. Assume that $\HHH$ normalizes $\NNN$,  $\GGG = \HHH \NNN$, and that $\NNN$ is strongly nested. 
\begin{enumerate}[\rm (i)]
\item 
If $\HHH$ is nested,  then so is $\GGG$. 
\item\label{algebraic.item}
If $\HHH$ is an algebraic group, then $\GGG$ is stronly nested.
\item\label{retraction.item}
If $\HHH$ is strongly nested and there is a retraction $\rho\colon \GGG\to\HHH$, i.e., a homomorphism of ind-groups which is the identity on $\HHH$,  then $\GGG$ is strongly nested.
\end{enumerate}
\end{lem}

\begin{proof}
(a) Write $\HHH=\bigcup_k H_k$ and $\NNN=\bigcup_\ell N_\ell$ as increasing unions of algebraic groups where the second is an admissible filtration.
The map $\gamma\colon\HHH\times \NNN \to \NNN$,  $(h,n) \mapsto hnh^{-1}$, is a morphism of ind-varieties. Hence, for each $k$, there is an $m\geq k$ such that $\gamma(H_k\times N_k) \subseteq N_m$. This implies that $N'_k:= \overline{\langle \gamma(H_k\times N_k)\rangle}$ is an algebraic subgroup of $N_m$, normalized by $H_k$. It follows that $G_k:=H_k N'_k\subset\GGG$  is an algebraic subgroup, and so $\GGG = \bigcup_kG_k$ is nested. This proves (i).

\ps
(b)
If $\HHH$ is an algebraic subgroup $H$, then, as in (a), 
we can can assume that the admissible filtration $\NNN=\bigcup_k N_k$ is stable under conjugation by $H$. We claim that the filtration $\GGG=H\NNN = \bigcup_k H N_k$ is admissible. In fact, 
let $A \subset \GGG$ be an algebraic subset. Enlarging $A$ we can assume that $e_\GGG\in A$ and that $A$ is stable under left-multiplication by $H$. Then $A = H (A \cap \NNN)$, and since $A\cap \NNN \subset N_k$ for some $k$ we get $A \subset H N_k$ showing that $\GGG=\bigcup_k H N_k$ is an admissible filtration. This proves (ii).
\ps
(c) 
As in (a) we can assume that $N_k$ is stable under conjugation by $H_k$, and thus we get a filtration $\GGG = \bigcup_k G_k$ by closed algebraic subgroups $G_k:=H_k N_k$. 
Let $A \subset \GGG$ be an algebraic subset. Enlarging $A$ we can assume that $e_\GGG\in A$. Its image  $\rho(A) \subset \HHH$ is contained in some $H_k$. Define  $\tilde A:=\overline{H_k A} \subset \GGG$. This is a closed algebraic subset contained in $H_k  (\tilde A \cap \NNN)$. 
If $\tilde A \cap \NNN \subset N_\ell$ and 
$m := \max\{k,\ell\}$, then $A \subset \tilde A \subset H_m N_m = G_m$. This proves (iii).
\end{proof}

\subsection{A version of the Lie-Kolchin theorem}\label{lie-kolchin.par}
The Lie-Kolchin Theorem tells us that a connected solvable linear algebraic group can be embedded for some $n$ into
$\B_n \subset \GL_n$, the group of upper triangular matrices (see Corollary~10.5 of \cite[\S III.10]{Bo1991Linear-algebraic-g}). This can be generalized in the following way.

\begin{mthm}\label{Lie-Kolchin.thm}
Let $R$ be a finitely generated $\kk$-domain, $M$ a finitely generated  and torsion-free $R$-module, and $\GGG \subseteq \Aut_R(M)$ a connected solvable subgroup. 
There is a finitely generated extension $S/R$ such that $M \otimes_R S$ is a free $S$-module of rank $n$ and that the image of $\GGG$ in $\GL_n(S)$ is conjugate to a subgroup of $\B_n(S)^\circ$. 

If, in addition, $\GGG = \langle \rho(A) \rangle$ for a morphism $\rho \colon A \to \Aut_R(M)$
with $A$ irreducible and $\id_M \in \rho(A) \subseteq \GGG$, 
then $\GGG$ is an algebraic subgroup of $\Aut_R(M)$.
\end{mthm}

\begin{proof}
Let $\KK$  be an algebraic closure of $\Quot(R)$, and  set $M_\KK := M\otimes_R \KK$. This is a  $\KK$-vector space of finite dimension $n$. Since $M$ is torsion-free, $M$ embeds into $M_\KK$, and we have an injective homomorphism 
of $\Aut_R(M)\into\GL(M_\KK)$.

The closure $\bGGG \subset \GL(M_\KK)$ is a solvable algebraic $\KK$-group which is connected by Lemma~\ref{connected-closure.lem}. The Lie-Kolchin Theorem provides a $\KK$-basis $(m_j)_{ 1\leq j \leq n}$ of $M_\KK$ such that $\bGGG \subseteq \B_n(\KK)$.
Let us fix a finitely generated extension $S$ of $R$ contained in $\KK$ such that $m_j \in M_S$ for all $j$ and  $M_S = \bigoplus_{j=1}^n S m_j$. 
Since $\B_n(\KK) \cap \GL_n(S) = \B_n(S)$ we get 
$\GGG \subseteq \B_n(S)$. 
By  Proposition~\ref{Hom-of-GLM.prop}, the homomorphism $\Aut_R(M) \to \Aut_S(M_S) = \GL_n(S)$ is a closed immersion of ind-groups, and so $\GGG \subset B_n(S)^\circ$, because $\GGG$ is connected.

Since $S$ is a finitely generated $\kk$-domain it is the coordinate ring
of an irreducible affine variety, and we can refer to  Example~\ref{nested.exa}\,(6) 
to conclude that 
$\B_n(S)^\circ$ is a  strongly nested closed ind-subgroup of $\GL_n(S)$. Hence, if $\GGG$ is generated by $\rho(A)$ for a morphism $\rho\colon A \to \Aut_R(M)$ with $A$ irreducible and $\id_M \in \rho(A)$,
then $\GGG$ is contained 
in an algebraic subgroup of $\B_n(S)^\circ$, and thus is a closed algebraic subgroup (see Lemma~\ref{alg-gen-subgroup.lem}).
\end{proof}

\section{Commutative automorphism groups}\label{quasi-affine.sec}
\vspace{0.1cm}
\begin{center}
\begin{minipage}{12.0cm}
{\textit{For a quasi-affine variety $X$, we establish a criterion ensuring that a connected solvable subgroup of $\Aut(X)$ is an algebraic subgroup
and apply it to derive the commutative case of Theorem~\ref{main-thm-A}.
}}
\end{minipage}
\end{center}

\subsection{Examples of quasi-affine varieties}\label{quasi-affine.subsec}

Recall, that a variety is \emph{quasi-affine} 
if it can be embedded as an open subset into an affine variety. For a quasi-affine variety
the canonical morphism $X \to \Spec(\OOO(X))$ is an open embedding 
(Proposition~5.1.2 of \cite{Gr1961Elements-de-geomet}). 
 
The plane minus the origin is quasi-affine and its algebra of regular functions $\OOO({\mathbb{A}}^2\setminus\{(0,0)\})$ is equal to $\OOO({\mathbb{A}}^2)=\kk[x,y]$. Now, let $D\subset {\mathbb {A}}^3$ be the union of the two planes $\{x=0\}$ and $\{y=0\}$. Let $X\subset D$ be the complement of the line $\{x=z=0\}$. Then, $X$ is quasi-affine and $\OOO(X)$ is not a finitely generated $\kk$-algebra (for this, consider the functions $y/z^m$ with $m\geq 1$).

Other interesting examples of quasi-affine varieties are given 
in~\cite{Neeman1988Steins, Jelonek2001, Vakil:example}.
Nagata's celebrated example provides a quasi-affine quotient $X$ for which $\OOO(X)$ is not finitely generated, see \cite{Nagata58}. 
We refer to~\cite{Wi2003Invariant-rings-an} and \cite{Nagata:lecture-notes} for general results on the existence of quasi-affine quotients.

\subsection{Base change and reduction to the generic fiber}

Let $\KK/\kk$ be an algebraically closed field extension.
If $\SSS$ is a set of morphisms $\rho \colon A \to \Aut(X)$ we denote by 
$\langle \SSS \rangle \subseteq \Aut(X)$ the subgroup generated by the images of the
$\rho \in \SSS$. Furthermore, $\SSS_{\KK}$ denotes the set of morphisms 
$\rho_{\KK} \colon A_{\KK} \to \Aut(X_{\KK})$ with $\rho \in \SSS$.

\begin{prop}
\label{Prop.Base-change}
Suppose $\SSS$ consists of morphisms $\rho\colon A \to \Aut(X)$ with $A$ irreducible and $\id_X\in \rho(A)$. 
If $\pi\colon X \to Y$ is a geometric quotient for the connected subgroup $\langle \SSS \rangle$, 
 then $\pi_\KK \colon X_{\KK} \to Y _{\KK}$
is a geometric quotient for $\langle \SSS_{\KK} \rangle$.
\end{prop}

\begin{proof}
By Theorem~1 in~\cite{Po2014On-infinite-dimens} 
(which is a refinement of Proposition~\ref{Orbits.prop}\eqref{algebraic-subset.item})
there exist $\rho_1, \ldots, \rho_n \in \SSS$ and 
$\varepsilon_1, \ldots, \varepsilon_n \in \{\pm1\}$, such that the orbits under
\[
    \tilde\rho \colon A \coloneqq A_1 \times \cdots \times A_n \to \Aut(X) \, , \quad
    (a_1, \ldots, a_n) \mapsto \rho_1(a_1)^{\varepsilon_1} \cdots \rho_n(a_n)^{\varepsilon_n}
\]
are exactly the $\langle \SSS \rangle$-orbits. Then
the morphism $\pi\colon X \to Y$, considered as a morphism of $\kk$-schemes, is a geometric quotient in the sense of \S1, Definition~0.6 of \cite{MuFoKi1994Geometric-invarian} when
we consider the action morphism 
$\sigma \colon A \times X \to X$, $(a, x) \mapsto \tilde\rho(a)(x)$. 
We have to show that $\pi_\KK \colon X_\KK \to Y_\KK$ satisfies Conditions (i)--(iv) of this definition for $\sigma_{\KK} \colon A_{\KK} \times X_{\KK} \to X_{\KK}$. 
For Conditions (i), (ii) and (iv) this follows from \S2\,(7) in 
\cite{MuFoKi1994Geometric-invarian} and for (iii) we can use \S2\,(4) loc. cit.
\end{proof}

Let $\pi \colon X\to Y$ be a dominant morphism, where $Y$ is irreducible.
Let $\KK/\kk(Y)$ be an algebraic closure. The $\kk(Y)$-scheme 
$\Xg = X \times_Y\Spec\kk(Y)$ is called the {\it generic fiber} of $\pi$.
The {\it geometric generic fiber} is the closed $\KK$-subscheme
\begin{equation*}
\Xgg \coloneqq \Xg \times_{\Spec \kk(Y)} \Spec\KK \subseteq X_{\KK}
\end{equation*}
which is a fiber of the morphism $\pi_\KK\colon X_\KK \to Y_\KK$. 
If $\pi$ is a geometric quotient, then $\Xgg$ is a  smooth $\KK$-variety by
Remark~\ref{Rem.Geom_quotient_smooth}. 

\begin{cor}
  \label{Cor.Geometric_fibre_of_Geom_quotient}
  Let $X$ be irreducible and 
  let $\GGG\subset\Aut(X)$ be a connected subgroup 
  admitting a geometric quotient $\pi\colon X \to Y$. Then, 
  \begin{enumerate}[\rm (1)]
    \item $\Xgg$ is a smooth irreducible $\KK$-variety.
        If $X$ is quasi-affine, then so is $\Xgg$.
    \item If $\GGG = \langle \SSS \rangle$ for some set $\SSS$ of morphisms $\rho \colon A \to \Aut(X)$     with irreducible $A$ and $\id_X \in \rho(A)$, then 
        $\langle \SSS_{\KK} \rangle$ acts transitively on $\Xgg$. \qed
  \end{enumerate}
\end{cor}

\begin{prop}
\label{G-stable-submodule.pro}
Let $X$ be an irreducible  quasi-affine variety, and let $\GGG \subset \Aut(X)$ be an algebraically generated and solvable subgroup. Assume that
\begin{enumerate}[\rm (i)]
\item there is a geometric quotient $\pi\colon X \to Y=X/\GGG$, and
\item the action of $\GGG$ on  $\OOO(X_{\bar\eta})$ is locally $\KK$-finite.
\end{enumerate}
Then $\GGG$ is an algebraic subgroup of $\Aut(X)$.
\end{prop}

\begin{proof}
We will construct an affine open subset
$Y' \subset Y$ and a  $\GGG$-stable, finitely generated free $\OOO(Y')$-submodule $M \subset\OOO(\pi^{-1}(Y'))$
which generates $\kk(X)$ as a field over $\kk$.
Then we apply Theorem~\ref{Lie-Kolchin.thm}
to the image of $\GGG$ in $\Aut_{\OOO(Y')}(M)$.
\ps
(a) Using Lemma~\ref{Lem.restriction-to-open} and Remark~\ref{Rem.Geom_quotient_smooth} we can assume $\pi\colon X \to Y$ smooth and $Y$ affine, with coordinate ring $R:=\OOO(Y)=\OOO(X)^\GGG \subset \OOO(X)$. We  set $K \coloneqq \kk(Y)$. 
Proposition 2.6 of \cite[Ch. I, \S\,2]{DeGa1970Groupes-algebrique} implies $\OOO(\Xg) = \OOO(X) \otimes_{R} K$ and $\OOO(\Xgg) = \OOO(\Xg) \otimes_{K} \KK$.

The action of $\GGG$ on $\OOO(X_\eta)$ is locally $K$-finite. 
Indeed, let $V \subset \OOO(X_\eta)$ be a finite dimensional $K$-subspace. 
Then $V_\KK:= \KK\otimes_{K} V\subset \OOO(X_{\bar\eta})$ is a finite dimensional subspace over  $\KK$. By assumption (ii), it is contained in a finite dimensional and $\GGG$-stable subspace $W$ of $\OOO(X_{\bar\eta})$. 
Hence, $W':=W \cap \OOO(X_\eta)$ is a finite dimensional and $\GGG$-stable $K$-subspace containing $V$. 
\ps
(b) 
Write $\GGG=\langle\rho(A)\rangle$ where $\rho \colon A \to \Aut(X)$ is a morphism with
irreducible $A$ and $\id_X \in \rho(A) \subseteq \GGG$.
Let $\theta \colon A \times X \to X, (a,x)\mapsto \rho(a)x$ be the corresponding morphism and 
$\theta_K \colon A_K \times_K \Xg \to \Xg$ the pull-back.
Choose  elements $f_1,\ldots,f_m \in \OOO(X)$ that generate $\kk(X)$ as a field over $\kk$.
By (a), the  subspace $\sum_i K f_i \subset \OOO(X_\eta)$ is contained in a finite-dimensional $\GGG$-stable $K$-subspace $W$. Let $x_{1},\ldots,x_n \in \OOO(X)$ be a  basis of~$W$.
There exist elements $h_{kj}$ of $\OOO(A_K) = \OOO(A) \otimes_{\kk} K$ and linearly independent elements $x_{n+1}, \ldots, x_{r}$ in $\OOO(X)$ such that 
\[
\langle x_{n+1}, \ldots, x_r\rangle_K \cap W = 0 \quad {\textrm{and}} \quad 
    \theta_K^\ast(x_k) = \sum_{j=1}^r h_{kj} \otimes_K x_j
    \quad \textrm{for all $k = 1, \ldots, n$} \, .
\]
The subset of $\kk$-rational points $A = A(\kk)$ is dense in
$A_K(K)$. For $a \in A(\kk)$ we have
$\theta_K^\ast(x_k)(a, \cdot ) \in W$, because $W$ is $\GGG$-stable. 
Thus $h_{kj}(a) = 0$ for all $j > n$, which shows that $h_{kj} = 0$ for $j > n$ and hence
$\theta^*_K(x_k) = \sum_{j=1}^n h_{kj}\otimes_K x_j$. 
Let us fix an $f \in R\setminus\{0\}$  such the 
$h_{kj}$ are contained in $\OOO(A) \otimes_\kk R_f$, where $R_f$ denotes the localization
of $R$ at $f$. Thus, $\rho(a)^*(x_k) \in \sum_{j=1}^n R_f x_j$ for all $a \in A(\kk)$.

\ps
(c) 
Set $Y' \coloneqq \Spec(R_f) \subseteq Y$.
The  $\GGG$-stable, free $R_f$-module 
$M \coloneq \sum_{j=1}^n R_f x_j \subseteq \OOO(\pi^{-1}(Y'))$ generates $\kk(X)$
as a field over $\kk$. This gives an injective 
homomorphism $\phi\colon\GGG \to \Aut_{R_f}(M)$.
By Proposition~\ref{embedding-into-GLR.lem}(2),
$\varphi \circ \rho \colon A \to \Aut_{R_f}(M)$ is a morphism. Since  $\varphi(\GGG)$ is generated by  
$\varphi \circ \rho (A)$ and   
$\GGG$ is solvable, Theorem~\ref{Lie-Kolchin.thm} 
implies that $G$ is an algebraic subgroup of $\Aut_{R_f}(M)$. Then, Lemma~\ref{alg-generated.lem} shows that $\GGG$ is an algebraic subgroup of $\Aut(X)$.
\end{proof}

\subsection{The case of a commutative group}\label{commutative-nested.par}

As a first application of Proposition~\ref{G-stable-submodule.pro} we extend
Theorem~B of \cite{CaReXi2023Families-of-commut} to quasi-affine varieties.
Define $\Ast:=\bbA^1\setminus\{0\}$, the affine line minus the origin.

\begin{mthm}\label{orbits-of-commutative-groups.thm} 
Let $X$ be a quasi-affine variety, and 
let $\HHH$ be connected commutative  subgroup of $\Aut(X)$. 
\begin{enumerate}[\rm (1)]
\item\label{transitive-action.item}
If $\HHH$ acts transitively on $X$, then $\HHH$ is an algebraic subgroup of $\Aut(X)$ of dimension equal to $\dim(X)$,  and $X$ is isomorphic to $\bbA^p \times (\Ast)^q$ for some integers $p, q$ with $p+q=\dim(X)$.
\item\label{nested.item}
$\HHH$ is \algebraicallyTame{}. 
\item\label{nested-characteristic-0.item} 
$\HHH$ is a direct product of an algebraic subtorus $T$ and a 
 \algebraicallyTame{} unipotent subgroup $\UUU$. 
\end{enumerate}
\end{mthm}

\begin{rem}\label{torus.faithful.rem}
The proof relies on the following fact: {\emph{if  a torus $T$ acts faithfully on an irreducible variety $X$, then 
$\dim(T)\leq \dim(X)$}}. See 
\cite[Corollaire~1, p.\,521]{De1970Sous-groupes-algeb} for a proof in arbitrary characteristic.
\end{rem}

\begin{proof}
(1) 
By Proposition~\ref{commutative-subgroup.prop}, $\HHH$ is an algebraic subgroup of $\Aut(X)$
of dimension $\dim (X)$. By Lemma~\ref{quasiaffine-gives-affine.lem}, $\HHH$ is affine 
and thus the statement follows from Theorem~1 of \cite{Br2021Homogeneous-variet}.
\ps
(2) We may assume that $X$ is irreducible. Let $\langle \rho \rangle \subset \HHH$ be an algebraically generated subgroup. 
By Lemma~\ref{Lem.restriction-to-open} and Theorem~\ref{popov-rosenlicht.thm}, replacing $X$ by an open dense subset, we can  assume that a geometric quotient 
$\pi \colon X \to X/\langle \rho \rangle$ exists.
Then the geometric generic fiber $\Xgg$ is an orbit under $\langle \rho_{\KK} \rangle$ 
(Corollary~\ref{Cor.Geometric_fibre_of_Geom_quotient}), and so, by (1), the image of $\langle \rho_{\KK} \rangle$ in $\Aut(\Xgg)$ is an algebraic group. Hence 
$\langle \rho \rangle$ acts locally $\KK$-finitely on $\OOO(\Xgg)$ (Proposition~8.6 in 
\cite{Hu1975Linear-algebraic-g}), and so 
$\langle \rho \rangle \subset \Aut(X)$ is an algebraic subgroup by 
Proposition~\ref{G-stable-submodule.pro}.

\ps
(3) 
By (2) and Lemma~\ref{quasiaffine-gives-affine.lem},  every algebraically generated subgroup of $\HHH$ is 
of the form $T \times U$ for an algebraic torus $T$ and 
a commutative connected unipotent algebraic group $U$. 
Remark~\ref{torus.faithful.rem} shows that the dimension of $T$ is bounded by the dimension of $X$, and thus $\HHH$ contains a unique torus $T_0$ of maximal dimension. This concludes the proof. 
\end{proof}

\section{Solvable automorphism groups}\label{maintheoremproof.sec}
\vspace{0.2cm}
\begin{center}
\begin{minipage}{12.0cm}
{\textit{
We prove Theorem~\ref{main-thm-A}. The first two sections provide preliminary tools.
}}
\end{minipage}
\end{center}
\vspace{0.2cm}

\subsection{Quotients by unipotent groups}
Let $Y$ be a variety. A \emph{$Y$-group scheme} is a morphism of varieties
$\pi \colon X \to Y$ together with a section $e \colon Y \to X$ (the identity, or neutral element), 
a $Y$-morphism $X \times_Y X \to X$ (the group operation) 
and a $Y$-morphism  $X \to X$ (the inverse), such that the usual
group axioms are satisfied (cf. \cite[\S\,1, Chapter~II]{DeGa1970Groupes-algebrique}). Note that the group operations are only defined for elements in the same 
fiber of $\pi$. 

\subsubsection{A (natural) quotient}
\label{Subsubsec.A_natural_quotient}
Let $X$ be an irreducible quasi-affine variety. 
Let $\UUU$ be a connected commutative and unipotent subgroup of $\Aut(X)$ (see~\S~\ref{connected_solvable_introduction.sec} and \S~\ref{Rosdecompandunipelem}).
By Theorem~\ref{orbits-of-commutative-groups.thm}\,(2),  $\UUU$ is \algebraicallyTame{}. Thus, Proposition~\ref{Orbits.prop}\,(2) provides a connected algebraic subgroup $U \subset \UUU$ with the same orbits as $\UUU$. In particular, $\OOO(X)^\UUU = \OOO(X)^U$. Since $U$ is unipotent and its action  on $\OOO(X)$ is locally finite and rational, we have
 $\kk(X)^U = \Quot(\OOO(X)^U)$ (see the Lemma on p.\,220 of \cite{Ro1961Quotient-Varieties}), hence  $\kk(X)^\UUU = \kk(X)^U$. This  proves the first assertion of the following proposition. 

\begin{prop}\label{commutative-quotient1.prop}
Let $X$ be an irreducible quasi-affine variety. 
Let $\UUU$ be a connected commutative and unipotent subgroup of $\Aut(X)$.
\begin{enumerate}[\rm (1)]
\item There is a connected unipotent algebraic subgroup $U \subset \UUU$ which has the same orbits in $X$ as $\UUU$. Moreover,
$\OOO(X)^\UUU = \OOO(X)^U$ and $\Quot(\OOO(X)^U) = \kk(X)^U = \kk(X)^\UUU.$
\end{enumerate}
Let $\pi\colon X \to \Spec\OOO(X)^\UUU$ be the canonical  morphism. 
There exists a smooth open dense and quasi-affine subvariety $Y_0 \subseteq Y$ such that,
for $X_0 \coloneqq \pi^{-1}(Y_0)$,
\begin{enumerate}[\rm (1)]
\setcounter{enumi}{1}
\item  
the morphism $\pi_0:=\pi|_{X_0} \colon X_0 \to Y_0$ is a  geometric quotient with respect to~$\UUU$ and is locally trivial in the Zariski topology, with fiber $\An$; 
\item the open subset $X_0$ is stable under the normalizer of $\UUU$ in $\Aut(X)$,
and the morphism $\pi_0\colon X_0\to Y_0$ is equivariant under this normalizer. 
\end{enumerate}
\end{prop}

Two propositions will complete this statement in the next paragraphs. 

\begin{proof} We prove (2) and (3) simultaneously, the main point being to construct the quotient $\pi_0\colon X_0\to Y_0$ in an equivariant way for the normalizer of $\UUU$.

Set $A=\OOO(X)$, and $Y= \Spec A^U= \Spec A^\UUU $. 
The morphism  $\pi$ is the composition of the open immersion $X \into \Spec A$ with the algebraic quotient $\Spec A\to Y$; it is equivariant for the action of the normalizer of $\UUU$

We apply Theorem~2 of~\cite{Br2021Homogeneous-variet}\footnote{Over an algebraically closed field $\kk$, solvable connected algebraic groups are split, see for instance \cite{Br2021Homogeneous-variet}, page 1141. We shall systematically use this fact when referring to  \cite{Br2021Homogeneous-variet}.}
to the action of $U$ on $X$: there exists a dense and $U$-stable open subset $X' \subseteq X$, a 
variety $Y'$, and a 
geometric quotient $\pi' \colon X' \to Y'$ for the action of $U$ on $X'$ such that $\pi'$ is a trivial 
$\An$-bundle for some $n \geq 0$. 
Since $\kk(Y') = \kk(X)^U = \Quot(A^U)$, 
we may further assume after shrinking $Y'$ (and thus $X'$), that $Y'$ is an open smooth
affine subvariety in $Y$, $\pi'$ coincides with $ \pi |_{X'}$, 
$\pi |_{\pi^{-1}(Y')}$ is
flat, and $Y' = \Spec(A^U_s)$, where  $A^U_s$ is  the localization at some $s \in A^U\setminus \{ 0\}$.

For $s' \in A^U$, we denote by $Y'_{s'}$ (resp. $X'_{s'}$) its non-vanishing locus in $Y'$ (resp. in $X'$).
The key step is to find an element $s' \in A^U$ such that the preimage of $Y'_{s'}$ under $\pi$ (instead of $\pi'$) coincides with $X'_{s'}$; that is, $\pi^{-1}(Y'_{s'}) = X'_{s'}$; in geometric terms, this means that 
over $Y'_{s'}$ the morphism $\pi \colon X \to Y$ is a geometric quotient
and a trivial $\An$-bundle.
For this, note that the algebraic group $U_{\kk(Y)}$ is connected and 
unipotent. The generic fiber of $\pi$ is quasi-affine
over $\kk(Y)$, with a dense $U_{\kk(Y)}$-orbit. Thus, the generic fibers of 
$\pi |_{\pi^{-1}(Y')} \colon \pi^{-1}(Y') \to Y'$ and $\pi' \colon X' \to Y'$ coincide. 
By flatness of $\pi |_{\pi^{-1}(Y')}$, the ring of regular functions of its 
generic fiber
is given by $A_s \otimes_{A_s^U} \kk(Y) =  A \otimes_{A^U} \kk(Y)$, because 
$A_s = \OOO(\pi^{-1}(Y'))$ 
(see~Proposition 2.6 of \cite[Ch. I, \S\,2]{DeGa1970Groupes-algebrique}).
Writing $\OOO(X') = A^U_s[t_1, \ldots, t_n]$
for algebraically independent $t_1, \ldots, t_n \in \OOO(X')$ over $A^U_s$, this  
gives us a non-zero
$s' \in A^U$ such that $t_i \in A_{ss'}$ for all $i$. Hence, we have
\[
A^U_{ss'}[t_1, \ldots, t_n] =  \OOO(X'_{s'}) 
\supseteq A_{ss'} \supseteq A^U_{ss'}[t_1, \ldots, t_n]
\]
and thus $\pi^{-1}(Y'_{s'}) = X'_{s'}$.

Now, consider the subset $Y_0 = N_\UUU (Y'_{s'})$, where $N_\UUU$
denotes the normalizer of~$\UUU$. It is an open, smooth and quasi-affine subvariety of $Y$.
Setting $X_0 = \pi^{-1}(Y_0)$, the claim follows.
\end{proof}

\subsubsection{A group scheme structure} 
In this paragraph and the next one, we keep the following notation: $X$ is an algebraic variety, $\UUU\subset \Aut(X)$
is a \algebraicallyTame{} and commutative group whose algebraic subgroups are unipotent, and $\pi \colon X \to Y$ is the geometric quotient for the action of $\UUU$. (In the setting of Proposition~\ref{commutative-quotient1.prop}, $X$ has now been implicitly replaced by the invariant open subset $X_0$.)

\begin{prop}\label{commutative-quotient2.prop}
Suppose the quotient $\pi\colon X \to Y$ has a section $e \colon Y \to X$.
Then $\pi$ has a unique structure of a commutative $Y$-group scheme with identity $e \colon Y \to X$
such that the maps $\UUU \to \pi^{-1}(y)$, $u\mapsto u(e(y))$, are group homomorphisms. 
The group operation has the following description:
for  $x, x' \in \pi^{-1}(y)$
\begin{equation*}\label{group-law}
x + x' = (u \circ u')(e(y)) = u(x') = u'(x) \quad \text{and} \quad
-x =  u^{-1}(e(y)),
\end{equation*}
where $u$ and $u'$ are any elements of $\UUU$ such that $u(e(y)) = x$ and  $u'(e(y)) = x'$.
\end{prop}
\begin{proof} 
Fix a connected unipotent algebraic subgroup $U \subset \UUU$ with the same orbits as $\UUU$ (Proposition~\ref{Orbits.prop}(2)). Let $\KK$ be an algebraic closure of $\kk(Y)$.
By Corollary~\ref{Cor.Geometric_fibre_of_Geom_quotient}, 
$U_{\KK} = U \times_{\Spec \kk} \Spec \KK$
acts transitively on the geometric generic fiber $\Xgg$ of 
$\pi$.
Hence, Propositions~2.1 and~3.1 from~\cite[Ch. III, \S\,3]{DeGa1970Groupes-algebrique}
give the first statement, and the description of the group law in the fibers follows by an explicit calculation.
\end{proof}

For readers not familiar with the language of sheaves in \cite{DeGa1970Groupes-algebrique} we give an elementary proof of this proposition in the Appendix, see Proposition~\ref{group-scheme.prop}.

\subsubsection{Description of the normalizer}\label{description_of_normalizer.sec}
We assume that the geometric quotient 
$\pi \colon X \to Y$ has a section $e \colon Y \to X$ and we endow 
$\pi$ with the structure of a $Y$-group scheme given by Proposition~\ref{group-law}.
We want to describe  the structure of the normalizer of $\UUU$. Here are two examples to keep in mind.

\begin{exa}
Consider the group of translations $\Tr_n \simeq(\Ga)^n$ of the affine space $\bbA^n$. In characteristic $0$, its normalizer in $\Aut(\bbA^n)$ is the group of affine transformations $g(z)=L_g(z)+T_g$, with  translation part $T_g\in \Tr_n(\kk)$ and  linear part $L_g\in \GL_n(\kk)$. In characteristic $p>0$, its normalizer is the group of polynomial automorphisms $g(z)=L_g(z)+T_g$ with again  $T_g\in \Tr_n(\kk)$, and with $L_g$ in the group $\GL_n(\kk[F])$ of 
automorphisms of the algebraic group $(\Ga)^n$. 
Here, $\kk[F]$ is the \emph{non-commutative} $\kk$-algebra of additive polynomials in one variable. It is generated by the Frobenius endomorphism $F\colon t\mapsto t^p$ (see~\cite{Lamy:Book}, \S Additive Automorphisms, in Chapter 7).
\end{exa}

\begin{exa} 
Consider the case $X = Y \times \bbA^n$, with $\pi \colon X \to Y$  the first projection.
The group $\UUU = \Tr_n(\OOO(Y))$ and its algebraic subgroup $U=\Tr_n(\kk)$ act on each fiber $\{y\}\times \bbA^n$ transitively. This gives an action of $\UUU$ on $X$ with quotient $\pi$. An element $g$ in $\Aut(X)$ that normalizes $\UUU$ can be written as
$g(y, z)=(\bar{g}(y), L_g(y)(z)+T_g(y))$ where $\bar{g}\in \Aut(Y)$, $T_g\in \Tr_n(\OOO(Y))$ and $L_g$ can be considered as a function from $Y$ to $\GL_n(\End(\Ga))$:
$L_g\in \GL_n(\OOO(Y))$ in characteristic $0$, and $L_g\in \GL_n(\OOO(Y)[F])$ in characteristic $p>0$.
\end{exa}

We will encapsulate this type of examples in Proposition~\ref{commutative-quotient3.prop}. 
For $x\in X$, we set $e_x=e(\pi(x))$.
If $s$ is a section of $\pi$, then $\phi(s)\colon x \mapsto x + s(\pi(x))$ is an automorphism of $X$. Two such automorphisms commute, hence  
\begin{equation*}
\Gamma(\pi) \coloneqq \set{\phi(s)}{s \colon Y \to X \
\textrm{is a section of  \,$\pi$}}  
\end{equation*}
is a commutative subgroup of $ \Aut(X)$. 
If $g\in \Aut(X)$ normalizes $\UUU$, the fibration $\pi$ is $g$-equivariant. Thus, the section $e$ is mapped by $g$ to a new section $s_g$. This section is given by 
\begin{equation*}
s_g(y) = g(e(\bar{g}^{-1}(y)))
\end{equation*}
where $\bar{g} \in \Aut(Y)$ is the image of $g$. Thus, to $g$ one can associate an automorphism $\TP_g:=\varphi(s_g)$ of $X$ such that $\pi\circ \TP_g=\pi$ and $\TP_g$ acts by translations on the fibers of $\pi$ (for the group scheme structure). It satisfies 
\begin{equation}\label{equation_for_Tg}
\TP_g\circ e \circ \bar{g}=g\circ e. 
\end{equation}
Now, composing $g$  with $\TP_{g^{-1}}$, we get an automorphism $\LP_g$ of $X$ that maps the image of $e$ onto itself. More precisely, setting
\begin{equation*}
\LP_g:= g \circ \TP_{g^{-1}}
\end{equation*}
and applying Equation~\eqref{equation_for_Tg} to $g^{-1}$, we obtain $\LP_g(e(y))=e(\bar{g}(y))$. 

More generally, denote by $c_g\colon \UUU\to \UUU$ the action by conjugation: 
$c_g(u)=g\circ u\circ g^{-1}$.
Then, if $x = u(e_x)$ for some $u \in \UUU$, we have
$\LP_g(x) = c_g(u)(e_{g(x)})$. Indeed, we get successively
\begin{align*}
\LP_g(x) &= g(x + s_{g^{-1}}(\pi(x))) =g(x + (g^{-1} \circ e \circ \bar{g})(\pi(x))) \\
&=g(u((g^{-1} \circ e \circ \bar{g})(\pi(x)))) = (g \circ u \circ g^{-1})(e_{g(x)}) \, .
\end{align*}
In other words, $\LP_g$ describes the action by conjugacy of $g$ on $\UUU$.

\begin{prop}\label{commutative-quotient3.prop}
Assume that $\pi \colon X \to Y$ has a section $e$,
as in Proposition~\ref{commutative-quotient2.prop}. Let  $\GGG \subseteq \Aut(X)$ be a subgroup that normalizes $\UUU$. 
With the notation introduced above, we have the following results.
\begin{enumerate}[\rm (1)]
\item
If $x = u(e_x)$ for some $u \in \UUU$, then 
$\LP_g(x) = (g \circ u \circ g^{-1})(e_{g(x)})$.
\item 
For $g \in \GGG$ we have
$\LP_g(x + x') = \LP_g(x) + \LP_g(x')$. 
\item
The map $\LP \colon \GGG \to \Aut(X)$ is a group homomorphism.
\item
The subgroup $\UUU \cap \GGG$ of $\GGG$ lies in the kernel of  the homomorphism $\LP$.
\item 
The subgroup $\Gamma(\pi)$ is normalized by $\LP(\GGG)$,
$\Gamma(\pi) \cap \LP(\GGG) = \{\id\}$, 
and $\GGG$ is contained in $\LP(\GGG) \ltimes \Gamma(\pi)$.
\item
If $\rho \colon A \to \Aut(X)$ is a morphism with 
image in $\GGG$, then the compositions $\TP \circ \rho \colon A \to \Aut(X)$ 
and $\LP \circ \rho \colon A \to \Aut(X)$ are morphisms.
\end{enumerate}
\end{prop} 
The map $u \mapsto \TP_u$ embeds $\UUU$ in a subgroup of $\Gamma(\pi)$. Property~(4) means that the action of $u$ on $X$ is indeed given by $\TP_u$. 

\begin{proof} Property (1) was proven above.
Let $x, x' \in X$ satisfy $\pi(x)=\pi(x')=:y$, and let $u$ and $u' \in \UUU$ satisfy $x = u(e_x)$ and $x' = u'(e_{x'})$.  For Property (2), we apply Property (1) to get $  \LP_g(x + x') = c_g (u \circ u')(e(\bar{g}(y))) $; then, 
by definition of the addition and by (1), we obtain
\begin{align*}
\LP_g(x + x') 
&= c_g (u) c_g(u')(e(\bar{g}(y))) 
 = \LP_g(x) + \LP_g(x').
\end{align*}
For Property (3), we fix $g$ and $h$ in $\GGG$. Our goal is to prove $L_{g\circ h}=L_g\circ L_h$. 
Let $u' \coloneqq c_h(u)$ and let $x' \coloneqq u'(e_{h(x)})$. Then $x' = u'(e_{x'})$ because $\pi(x')=\pi(x)$ and
\begin{equation*}
\LP_g( \LP_h(x) ) \stackrel{(1)}{=} \LP_g(c_h(u)(e_{h(x)})) 
  = \LP_g(x') \stackrel{(1)}{=}c_g (u') (e_{g(x')}).
\end{equation*}
Since $c_{g}\circ c_h=c_{g\circ h}$, the conclusion follows again by~(1). 

Since $\UUU$ is commutative, it acts trivially on itself by conjugation. Since it acts transitively on each fiber of $\pi$, Property (1) implies that $L_u=\id$ for every $u\in \UUU$, which proves~(4).

The first claim in Property~(5) is a direct consequence of (2). Indeed, if  $s \colon Y \to X$ is a section of $\pi$, then
\begin{align*}
(\LP_g \circ \varphi(s) \circ \LP_g^{-1})(x)
&= \LP_g \left(\LP_g^{-1}(x) + s(\pi(\LP_g^{-1}(x))) \right) \\
&\stackrel{(2)}{=}
x + \LP_g\left(s(\bar{g}^{-1}(\pi(x)))\right) =\varphi(s')(x),
\end{align*}
where $s' \colon Y \to X$ denotes the section
$\LP_g \circ s \circ \bar{g}^{-1}$.
For the  second claim it is enough to note that $\LP_g(e_x) = e_{g(x)}$ for all $x$ 
(see again~(1)). The last claim in~(5) follows directly from the definitions.

For Property~(6), note  that
the map $  \eta \colon A \times X \to A \times X$ defined by 
\[
(a, x) \mapsto 
\left(a, x + \left(\rho(a) \circ e \circ \overline{\rho(a)}^{-1}\right)(\pi(x)) \right)
\]
is an $A$-isomorphism. This proves the first claim,  and the second follows from the first. 
\end{proof}

\subsection{Connected subgroups for which the commutator is a torus}
\label{Subsec.Commutator_torus}
If $H$ is a connected solvable algebraic group, then $[H,H] = [H_{\text{\Tiny aff}}, H_{\text{\Tiny aff}}]$ since $H_{\text{\Tiny ant}}$ is central in $H$, and thus $[H, H]$ is unipotent
(see~Theorem~19.3(b) in \cite{Hu1975Linear-algebraic-g}). In particular, if a connected 
algebraic group $G$ contains a normal torus $T$ such that $G/T$ is commutative, then $G$ is commutative. This result generalizes as follows (a key step towards Theorem~\ref{main-thm-A} and Corollary~\ref{corB}).

\begin{thm}\label{normal-torus.thm}
Let $\GGG \subseteq \Aut(X)$ be a connected subgroup where $X$ is 
irreducible. Assume that $\GGG$  contains a torus $T$ that is normal in $\GGG$. Then $\GGG$ is commutative if and only if the quotient $\GGG/T$ is commutative.
\end{thm}

To prove it, we collect three lemmas. Recall that the action of an algebraic 
group on a variety is \emph{free} if the stabilizer of each point is 
\emph{schematically} trivial.

\begin{lem}\label{quotient-by-torus.lem}
Let $T$ be a torus acting on a normal variety~$X$.
There is a finite covering $X = \bigcup_i X_i$ by affine open and $T$-stable subsets $X_i\subset X$.
If the action of $T$ is free, 
there is a geometric quotient $\pi\colon X \to Y$ which is a principal $T$-bundle over a prevariety $Y$.
\end{lem}

\begin{proof}
The first statement follows from Corollary 2 in \cite{Su1974Equivariant-comple}.
For a free action of a torus $T$ on an affine variety $Z$, it follows from Luna's Slice Theorem that the quotient morphism 
$Z \to Z\quot T := \Spec\OOO(Z)^T
$
is a principal $T$-bundle, see Proposition~7.6 of \cite{BaRi1985Etale-slices-for-a}. 
Gluing the principal bundles $X_i \to X_i\quot T$ we obtain a principal $T$-bundle $\pi\colon X \to X\quot T$ where $X\quot T$ is a prevariety.
\end{proof}

\noindent{\bf{Notation.}} If $\pi \colon X \to Y$ is a morphism, we denote by $\Aut_Y(X)$ the
subgroup of automorphisms $\varphi$
of $X$ such that $\pi \circ \varphi = \pi$. 
As in \S\,\ref{connected_subgroups.sec}, we denote by $\Aut_Y(X)^\circ$ the connected component
of the identity of $\Aut_Y(X)$. 
 
\begin{lem}\label{AutT(X/Y).lem}
Let $T \subseteq \Aut(X)$ be a torus, and let $\pi\colon X \to Y$ be a principal $T$-bundle 
of irreducible varieties.  
Then $\Aut_Y(X)^\circ = T$.
\end{lem}

\begin{proof}
Let $\rho\colon A \to \Aut(X)$ be a morphism where $A$ is irreducible and $\id_X \in \rho(A) \subset \Aut_Y(X)$. We have to show that $\rho(A) \subset T$.
The morphism $\rho$ induces an $A$-isomorphism $\phi \colon A \times X \to A \times X$
(see Property~(\ref{item2.5.3}) in \S\,\ref{algebraic_subgroups_and_basic_facts.sec}).
Let $Y' \subseteq Y$ be an affine and dense open subset such that $X':=\pi^{-1}(Y') \to Y'$  is a trivial bundle, so that we can identify $X' = Y' \times T$. 
Then $\varphi$ restricts to an $A$-isomorphism $A \times Y' \times T \to A \times Y' \times T$.
By a theorem of Rosenlicht's (Theorem~2 in \cite{Ro1961Toroidal-algebraic}), the composition with the projection onto $T$ is of the form
\[
A \times Y' \times T \to T , \quad (a, y, t) \mapsto f_1(a) f_2(y) f_3(t)
\]
for morphisms $f_1 \colon A \to T$, $f_2 \colon Y' \to T$, and $f_3 \colon T \to T$. At least 
one element of $A$ acts as the identity on $X'$, so $f_2$ is constant
and $f_3(t) = \lambda t$ for some $\lambda \in T$.
Hence, the image of $A$ in $\Aut(X')$ is contained in $T$. 
Since 
$\Aut_Y(X) \to \Aut_{Y'}(X')$, $\psi \mapsto \psi |_{X'}$ is injective,  
we conclude that $\rho(A)$ is contained in $T$.
\end{proof}

\begin{lem}
    \label{Lem.Further_crit_to_be_alg}
    Let $\pi \colon X \to Y$ be a geometric quotient for a connected subgroup 
    $\HHH \subseteq \Aut(X)$, normalized by a connected 
    subgroup $\GGG \subseteq \Aut(X)$. Assume that 
    the image $G \subseteq \Aut(Y)$ of $\GGG$ is an algebraic 
    group and the kernel of $\GGG \to G$ is finite dimensional. Then $\GGG$ is
    an algebraic subgroup of $\Aut(X)$.
\end{lem}

\begin{proof}
    Let $\mu\colon B\to \Aut(X)$ be an injective morphism with 
    $\mu(B)\subset \GGG$, and let $\phi\colon B\to G$ 
    be the composition of $\mu$ with the projection $\GGG \to G$. 
    Then $\phi$ is a morphism of varieties. 
    Let $C$ be an irreducible component of $\phi^{-1}(g)$ for some $g \in G$.   
Pick $c_0$ in $C$. Then the morphism $C\to \Aut(X)$, $c\mapsto \mu(c_0)^{-1}\circ \mu(c)$, is injective and takes values in $\KKK = \ker(\GGG \to \Aut(X))$. 
Thus, $\dim(C)\leq \dim(\KKK)$, and so $\dim(B)\leq \dim(\KKK)+\dim(H)$.
Now \hyperlink{Ramanujam.thm}{Ramanujam's theorem}
implies that $\GGG$ is algebraic.
\end{proof}

\begin{proof}[Proof of Theorem~\ref{normal-torus.thm}] 
We have to show that $\GGG$ is commutative when $\GGG/T$ is.
\ps
(a) We first show that $\GGG$ centralizes $T$. 
Let $A$ be an irreducible variety, and let $\eta \colon A \to \Aut(X)$ be a morphism such that $\id_X \in \eta(A)  \subseteq \GGG$. 
Consider the morphism $\mu\colon A\times T \to \Aut(X)$ defined by $\mu(a,t)=\eta(a)\,t\,\eta(a)^{-1}$. 
As $T$ is normal in $\GGG$, $\mu(A\times T)\subset T$. By \hyperlink{Ramanujam.thm}{Ramanujam's theorem}, $\mu$ defines a morphism $A\times T \to T$ of algebraic varieties, hence a morphism 
$\rho\colon  A \to \Aut(T)$ of ind-varieties with image in $\Aut_{\tiny gr}(T)$. Since $A$ is irreducible and 
$\Aut_{\tiny gr}(T)^\circ = \{\id_T\}$ we see that $\rho(A)$ centralizes $T$, and the claim follows.
\ps
(b) Now assume that $\GGG$ acts transitively on $X$, in particular $X$ is smooth.
To prove the theorem in this case, we show that $\GGG$ is an algebraic group. 

Since $T$ is central in $\GGG$, all schematic stabilizers $T_x$, $x \in X$, are equal, hence are trivial\footnote{For this, one can apply Lemma~4.5 of \cite{Note_on_Ramanujam} to $\set{(t, y)}{ty = y}$ in $N \times X$, where $N=T_x$ is the common schematic stabilizer, and use Remark~1.2 of \cite{Note_on_Ramanujam}},
so $T$ acts freely on $X$. By Lemma~\ref{quotient-by-torus.lem}, the geometric quotient exists and is a principal $T$-bundle $\pi\colon X \to Y$, where $Y$ is a prevariety. 
Let  $G\subseteq \Aut(Y)$ be the image  of $\GGG$ in $\Aut(Y)$. 
Since $G$ is connected and acts transitively on $Y$, Lemma~\ref{Prevariety.lem}
shows that $Y$ is a variety. Moreover, $G$ is commutative, as $\GGG/T$ is commutative. 
Thus,  Proposition~\ref{commutative-subgroup.prop} shows that $G$ is an algebraic group.
The kernel of $\GGG \to G$ is contained in $\Aut_Y(X)$, which is finite dimensional, 
as $\Aut_Y(X)^\circ = T$ (Lemma~\ref{AutT(X/Y).lem}). Hence, 
$\GGG$ is an algebraic group (Lemma~\ref{Lem.Further_crit_to_be_alg}).
\ps
(c) In the general case we use (b) to show that for every $\GGG$-orbit $O\subset X$, the restriction $\GGG_{\vert \OOO}\subset \Aut(O)$ is commutative, which implies that $\GGG$ itself is commutative.
\end{proof}

\subsection{Proof of Theorem~\ref{main-thm-A}}
\label{Subsec.Proof_Main_Thm}
The following is a more specific version of Theorem~\ref{main-thm-A}.

\begin{thm} 
\label{thm.Proof_of_them_A}
Let $X$ be a quasi-affine variety of dimension $n$. 
If $\GGG$ is a connected and solvable subgroup of $\Aut(X)$, then $\GGG$ is \algebraicallyTame{}.
If furthermore $\GGG$ acts transitively on $X$, then  $X$ is affine. 
\end{thm}

\begin{proof}
We can assume that $\GGG$ is algebraically generated, 
and our goal is to prove that it is an algebraic subgroup of $\Aut(X)$. 
Since $\GGG$ is connected, it preserves the irreducible components of $X$, and hence 
we can assume that $X$ is irreducible by \hyperlink{Ramanujam.thm}{Ramanujam's theorem}.
We proceed by induction on the pair $(n, d)$ where $n=\dim(X)$ and $d=\dl(\GGG)$ are positive integers. 
We use the lexicographic order and  start with $n=1$. 
\ps
(a)
In dimension $1$, $X$ is affine and its automorphism group is either finite, the group $\Aff_1$ of affine transformations if $X\simeq \bbA^1$, or the group $\Gm$ of  homotheties  if $X\simeq \bbA^1\setminus\{0\}$.
Thus, the result is established for all pairs $(1, d)$.
\ps
(b)
 Now we assume that $n = \dim X \geq 2$ and that the result holds  for all pairs $(m,d)$ with $m\leq n-1$.
  
To prove the result in dimension $n$, we start with $d=1$: 
this is the case of commutative groups.  Assertions~\eqref{nested.item} and~\eqref{transitive-action.item}  of Theorem~\ref{orbits-of-commutative-groups.thm}  from \S\,\ref{commutative-nested.par} imply respectively that $\GGG$ is an algebraic subgroup and  that $X$ is affine if the action is transitive.
This proves our result for  the pair  $(n, 1)$. 
 
From now on, we still work in dimension $n$, but we assume  that $d\geq 2$ 
and that the result holds for pairs up to $(n, d-1)$. 
Define $r=d-1$. Then $r\geq 1$ and the derived series reads
$$
\GGG^{(0)}=\GGG  \supset \GGG^{(1)} = [\GGG,\GGG] \supset \cdots\supset \GGG^{(j)} \supset \cdots \supset \GGG^{(r)} \supset \{ \id_X \}.
$$
Let $\HHH := \GGG^{(r)}$ denote the last non-trivial member, 
it is a connected and commutative subgroup of  $\Aut(X)$. 
The assertions~\eqref{nested.item} and~\eqref{nested-characteristic-0.item} of Theorem~\ref{orbits-of-commutative-groups.thm} show that 
$\HHH$ is a product $T \times \UUU$ where $T$ is a torus and $\UUU$ is 
unipotent, commutative, and \algebraicallyTame{}. 
Now, we split the proof in two cases.

\ps
(c) 
First, we assume  that  $\GGG$ has a dense orbit. 
Using Lemma~\ref{Lem.restriction-to-open} we can assume that $X = \Xm$, i.e.\ $\GGG$ acts transitively.
Applied to $\GGG^{(r-1)}$,  Theorem~\ref{normal-torus.thm} shows that the unipotent part $\UUU$ of $\HHH$ is non-trivial. 
By Proposition~\ref{commutative-quotient1.prop} 
and the fact that $\GGG$ acts transitively on $X$,  we infer that 
\begin{enumerate}[\rm $\quad \bullet$]
\item there exists a geometric quotient $\pi\colon X \to Y$
for $\UUU$,  
\item $\pi$ is a $\GGG$-equivariant $\bbA^k$-bundle, 
locally trivial in the Zariski topology,
and
\item the image $\GGG' \subseteq \Aut(Y)$
of $\GGG$ is an algebraically generated subgroup.
\end{enumerate}
Since $\dim(Y) < \dim (X)$ and the action is transitive, the induction hypothesis implies that $Y$ is affine and $\GGG'$ is an algebraic subgroup of $\Aut(Y)$. 
The morphism $\pi$ being affine, $X$ is also affine. 
Moreover, by \cite{BaCoWr1976Locally-polynomial}, the $\bbA^k$-bundle $\pi \colon X \to Y$ has a section, and thus a natural $Y$-group scheme structure 
(Proposition~\ref{commutative-quotient2.prop}).
We use the homomorphism 
\[
\LP \colon \GGG \to \Aut(X)
\]
of Proposition~\ref{commutative-quotient3.prop}.
The images $\LP(\GGG)$ and  $\LP(\HHH)$ are connected subgroups of $\Aut(X)$, and $\LP(\UUU)$ is trivial, thus $\LP(\HHH) = \LP(T)$ is an algebraic subtorus of $\Aut(X)$.
Since $\HHH = \GGG^{(r)}$ is the commutator of $\GGG^{(r-1)}$, Theorem~\ref{normal-torus.thm} 
applied to $\LP(\GGG^{(r-1)})$ and $\LP(\HHH)$ implies that $\LP(\HHH)$ is trivial. Thus, the derived length of $\LP(\GGG)$ is strictly less than $d$, and,
by induction, $\LP(\GGG)$  is  a connected algebraic subgroup of $\Aut(X)$. 

By Proposition~\ref{commutative-quotient3.prop},  
$\GGG$ is contained in $\LP(\GGG) \ltimes \Gamma(\pi)$
where $\Gamma(\pi) \subseteq \Aut(X)$ is the commutative group given by 
all sections of $\pi$.
Denote by $\rho \colon A \to \Aut(X)$ a morphism from an irreducible variety $A$
such that $\id \in \rho(A)$ and $\langle \rho(A)\rangle = \GGG$.
Let $S \subseteq \Gamma(\pi)$ be the subgroup generated by the image of the morphism 
\[
A \times \LP(\GGG) \to \Aut(X) \, , \quad (a, f) \mapsto 
f \circ \TP_{\rho(a)^{-1}}^{-1} \circ f^{-1} \, ,
\]
where $\TP_g$ is the ``translation part'' of $g$, as in  \S\,\ref{description_of_normalizer.sec} and
Proposition~\ref{commutative-quotient3.prop}.
This
group $S$
is algebraically generated and, 
$\Gamma(\pi)$ being commutative, Theorem~\ref{orbits-of-commutative-groups.thm} shows that $S$  is an algebraic subgroup of $\Aut(X)$.
By construction, $\LP(\GGG)$ normalizes $S$.
Since $\GGG$ is generated by $\rho(A)$ and $\rho(a) = \LP_{\rho(a)} \circ \TP_{\rho(a)^{-1}}^{-1} \in \LP(\GGG) \ltimes S$ 
for all $a \in A$,  
$\GGG$ is contained in the algebraic subgroup 
$\LP(\GGG) \ltimes S \subseteq \Aut(X)$.  
By Lemma~\ref{alg-gen-subgroup.lem} $\GGG$  is a closed algebraic group, which proves  the result for the pair $(n,d)$ when $\GGG$ has a dense orbit.  
\ps
(d) Now we assume that  the result holds for pairs $(n',d')<(n,d)$ in lexicographic order, as well as for $(n,d)$ when $\mdo(\GGG)=n$. 
We want to prove it for $(n,d)$ when $1\leq \mdo(\GGG) \leq n-1$.

Theorem~\ref{popov-rosenlicht.thm} and Remark~\ref{Rem.Geom_quotient_smooth}
provide a dense, $\GGG$-stable and open subset $X'$ of $X$ that admits a geometric quotient 
for $\GGG$. Using Lemma~\ref{Lem.restriction-to-open} we are free to replace $X$ by $X'$
and thus we may assume that there exists a geometric quotient 
$\pi\colon X \to Y\coloneqq X/\GGG$.
The geometric generic fibre $\Xgg$ is an irreducible quasi-affine $\KK$-variety, 
where $\KK$ is an algebraic closure of $\kk(Y)$  (Corollary~\ref{Cor.Geometric_fibre_of_Geom_quotient}).

Let $\rho \colon A \to \Aut(X)$ be a morphism, where $A$ is irreducible, $\id \in \rho(A) \subseteq \GGG$ and $\rho(A)$ generates $\GGG$ as a group. Denote by $\GGG_{\KK} \subseteq \Aut(\Xgg)$
the subgroup generated by the image of $\rho_{\KK} \colon A_{\KK} \to \Aut(\Xgg)$.
By Corollary~\ref{Cor.Geometric_fibre_of_Geom_quotient}, 
$\GGG_{\KK}$ acts transitively on $\Xgg$, so we derive from our standing assumption 
that $\GGG_{\KK}$ is an algebraic subgroup.
Hence, by Proposition~\ref{G-stable-submodule.pro} together with 
Proposition~8.6 of \cite{Hu1975Linear-algebraic-g}, 
$\GGG$ is an algebraic subgroup of $\Aut(X)$.

Thus, the result is obtained for the pair $(n,d)$, which  concludes the proof by induction. 
\end{proof}

\section{Applications}\label{applications.sec}

\vspace{0.1cm}
 \begin{center}
 \begin{minipage}{12.0cm}
 {\textit{We present our applications of Theorem~\ref{main-thm-A}, as well as a general definition of the notion of solvable and unipotent radicals. We recommend to read Example~\ref{exa.perepechko} below to get an idea of
some of the phenomena arising in positive characteristic.
 }}
 \end{minipage}
 \end{center}

\subsection{Levi decomposition and proof of Corollary~\ref{corB}}\label{sec.Levi}
Recall the notion of unipotent radical $R_u(G)$ for an algebraic group $G$ from \S~\ref{Rosdecompandunipelem}.

\begin{lem}\label{Quotient_by_unipotent_radical_bounded.lem}
Let $X$ be a variety. There exists an integer $d$, depending only on $\dim(X)$, such that  $\dim (G/R_u(G)) \leq d$ for any
algebraic subgroup $G \subset \Aut(X)$.
\end{lem}

\begin{proof}
First, assume $G$ is affine. By Remark~\ref{torus.faithful.rem},  the dimension of a torus acting faithfully on $X$ is at most $\dim (X)$. Thus, the rank of $G/R_u(G)$ is at most $\dim (X)$ which implies
$\dim (G/R_u(G))$ is bounded by a quadratic function of $\dim (X)$. Indeed, 
let $H$ be a reductive algebraic group: for $r = \mathrm{rank}(H) \geq 15$ we have $\dim(H) \leq 2r^2 + r$ and this bound is realized by $\SO_{2r+1}$; for $r \leq 15$ we have  $\dim(H) \leq 31 r$, but this bound is not optimal.
The case of affine groups gives an upper bound for $\dim \Gaff/R_u(\Gaff)$.
By \cite[proof of Prop.~5.5.4]{BrSaUm2013Lectures-on-the-st},  $\dim \Gant\leq 3\dim(X)$. 
Since $R_u(\Gaff) = R_u(G)$ we get
$\dim (G / R_u(G)) \leq \dim \Gant + \dim (\Gaff /R_u(\Gaff))$
and the conclusion follows.
\end{proof}

Let $\GGG$ be a group. A set $\SSS$ of subgroups of $\GGG$ is \emph{covering} if  $\GGG = \bigcup_{G \in \SSS} G$. It is \emph{directed} if, for any pair $G_1,G_2 \in \SSS$, there is a $G \in \SSS$ containing both. 

\begin{rem}\label{nested_alg_tame_covering.rem}
If  $\GGG $ is \algebraicallyTame{} the set $\MMM$ of all its connected
algebraic subgroups  is covering and directed. 
If $\GGG$ is nested by connected algebraic subgroups 
$G_1 \subseteq G_2 \subseteq \cdots$, then $\SSS = \{G_1, G_2, \ldots \}$
is covering and directed. 
\end{rem}

\begin{prop}\label{Pro.Unipotent_radical}
Let $\GGG \subset \Aut(X)$ be a subgroup. 
Suppose that $\SSS$ is a set of connected and algebraic subgroups of $\GGG$ that is covering and directed.
Then,
\begin{enumerate}[\rm (1)]
\item \label{Pro.Unipotent_radical1} 
The subset $\SSS'$ of those $G \in \SSS$ such that $\dim( G / R_u(G) )$ is maximal is covering and directed.
\item \label{Pro.Unipotent_radical2} 
The set of unipotent radicals $R_u(G)$ with $G \in \SSS'$ is directed.
\item \label{Pro.Unipotent_radical3} The group $R_u^{\SSS}(\GGG) :=\bigcup_{G \in \SSS'} R_u(G)$ 
is  connected and normal in $\GGG$, it is a union of connected unipotent algebraic groups, and  $\GGG = G \,R_u^{\SSS}(\GGG)$ for any $G \in \SSS'$.
\end{enumerate}
\end{prop}

\begin{proof}
(a)
If $G_1, G_2 \in \SSS$ and $G_1 \subset G_2$, then $G_1 \cap R_u(G_2)$ is  normal in $G_1$, so 
$(G_1 \cap R_u(G_2))^\circ \subset R_u(G_1)$. 
It follows that $\dim G_2/R_u(G_2) \geq \dim G_1/R_u(G_1)$. 

In case of equality,  $(G_1\cap R_u(G_2))^\circ$ has finite index in $R_u(G_1)$, hence $(G_1\cap R_u(G_2))^\circ= R_u(G_1)$ by  connectedness. In particular, $R_u(G_1)\subset R_u(G_2)$ and $R_u(G_1)$ has finite index in $G_1\cap R_u(G_2)$.
\ps
(b) For $G \in \SSS$ and $G' \in \SSS'$,  there is some $G'' \in \SSS$ containing $G$ and $G'$, and by (a) we deduce that $G'' \in \SSS'$. This proves Assertions~\eqref{Pro.Unipotent_radical1} and~\eqref{Pro.Unipotent_radical2}.

\ps
(c)
By construction, $R_u^\SSS(\GGG)$ is a union of connected and unipotent algebraic groups.  
If $g\in \GGG$ and $h\in R_u^\SSS(\GGG)$, then by (b) there is an element $G$ of $\SSS'$ 
containing both $g$ and $h$, with $h\in R_u(G)$. Since $R_u(G)$ is normal in $G$, we deduce that $ghg^{-1}\in R_u(\GGG)$.
To conclude the proof of Assertion~\eqref{Pro.Unipotent_radical3}, note that  
$\GGG = G R_u^\SSS(\GGG)$ for any $G \in \SSS'$ because $\SSS'$ covers $\GGG$.
\end{proof}

\begin{rem}\label{dim-G/R_u.rem}
Part (a) of the proof above shows that for two connected algebraic subgroup $G_1 \subset G_2$ we have $\dim G_1/R_u(G_1) \leq \dim G_2/R_u(G_2)$.
\end{rem}

\begin{rem} 
\label{first_definition_unipotent_radical.rem}
Suppose 
$\GGG$ is \algebraicallyTame{}. By Remark~\ref{nested_alg_tame_covering.rem}, Proposition~\ref{Pro.Unipotent_radical} can be applied to the set $\MMM$ of {\emph{all}} its algebraic and connected subgroups. The subgroup $R_u^\MMM(\GGG)$ will coincide
with the {\emph{unipotent radical}} of $\GGG$ defined below in \S~\ref{unipotent_radical.sec}. So for simplicity, we denote it by $R_u(\GGG)$.
\end{rem}

\begin{proof}[Proof of Corollary~\ref{corB}]
By Theorem~\ref{main-thm-A} the connected solvable subgroup $\HHH$ is \algebraicallyTame{}, and all its algebraic subgroups are affine (Lemma~\ref{quasiaffine-gives-affine.lem}).  
Let $\MMM$ be the covering directed set of all connected algebraic subgroups of $\HHH$, and let $\MMM'$ be the subset of those $G \in \MMM$ such that
$\dim G/ R_u(G)$ is maximal. By Proposition~\ref{Pro.Unipotent_radical}\,\eqref{Pro.Unipotent_radical1}, $\MMM'$ is also covering and directed.
Fix a torus $T \subseteq \HHH$  of  maximal dimension. Then $T \in \MMM'$.
By Proposition~\ref{Pro.Unipotent_radical}\,\eqref{Pro.Unipotent_radical3}, 
$\HHH = T R_u(\HHH)$. Every algebraic subgroup in $R_u(\HHH)$ 
is unipotent and 
the unipotent elements of $\HHH$ are exactly the elements in $R_u(\HHH)$.
Hence, $T \cap R_u(\HHH)$ is trivial and so 
$\HHH = T \ltimes R_u(\HHH)$.
\end{proof}

\subsection{Solvable subgroups of maximal derived length}\label{solvable.subgroups.sec}

The Jonquières subgroup $\Jonq(n) \subset \Aut(\An)$ was introduced in \S\,\ref{jonquieres.sec}. 
It is connected, its derived length is equal to $n+1$, and its  derived subgroup  $\Jonq_u(n)$ coincides with $R_u(\Jonq(n))$.
The following proposition and its corollary generalize Theorems~D and~C of \cite{KZ24} to arbitrary characteristic.  

\begin{prop}\label{Prop.U_transitive}
Let $X$ be a quasi-affine variety and $\HHH \subset \Aut(X)$ a connected solvable subgroup. If $R_u(\HHH)$ has a dense orbit, then $X \simeq \An$ and $\HHH$ is conjugate to a subgroup of $\Jonq(n)$. In particular, $\dl(\HHH) \leq n+1$ and $\dl(R_u(\HHH)) \leq n$. 
\end{prop}

For the proof we need the following result, see Proposition~3.8 in 
\cite[Ch.\,IV, \S\,2]{DeGa1970Groupes-algebrique}.

\begin{prop}\label{Thm.Torus_actions_on_U}
Let $U$ be a connected unipotent algebraic group. Assume that a torus $T$ acts on 
$U$ by algebraic group automorphisms. If $U$ is non-trivial, then 
$U$ contains a central, $T$-stable subgroup that is isomorphic to $\Ga$.
\end{prop}

\begin{proof}[Proof of Proposition~\ref{Prop.U_transitive}]
We argue by induction on $n:=\dim X$, the case $n=1$ being trivial. 
Let $T\subset \HHH$ be a torus of maximal dimension. From
Corollary~\ref{corB} 
we get  $\HHH=T\ltimes \UUU$ where $\UUU:=R_u(\HHH)$ is \algebraicallyTame{}.
\ps
(a) Let $U \subset \UUU$ be a connected algebraic subgroup with the same orbits and the same invariants as $\UUU$ (see Proposition~\ref{Orbits.prop}\,(2)). 
Since $U$ is unipotent its orbits are closed (Proposition~4.10 in 
\cite{Bo1991Linear-algebraic-g})
and isomorphic to $\An$ 
(Theorem~1 in \cite{Br2021Homogeneous-variet}), so $X$ is isomorphic to $\An$. 
\ps
(b) Choose a point $x \in X$ and denote by $\UUU_x$ its stabilizer.  The fixed point set $X^{\UUU_x}$ of $\UUU_x$ coincides with $\bigcap_U X^{U_x}$, where $U$  runs through the set of connected algebraic subgroups of $\UUU$. Choosing $U$ big enough, we can assume that $X^{U_x} = X^{\UUU_x}$ and that $T$ normalizes~$U$. 
\ps
(c) Choose a central $T$-stable subgroup $N$ of $U$ which is isomorphic to $\Ga$, as in Proposition~\ref{Thm.Torus_actions_on_U}. 
Since $U \subseteq \Aut(X)$
acts transitively on $X$ and $N$ is central in $U$ the schematic intersection $N \cap U_x$ acts trivially on $X$ and thus it is trivial
(see Remark~1.2 in \cite{Note_on_Ramanujam}).
An element $g$ of $U$ normalizes $U_x$ (resp. $\UUU_x$) if and only if $gx$ is in $X^{U_x}$  (resp. in  $X^{\UUU_x}$). Thus, if $g$ normalizes $U_x$, then it normalizes $\UUU_x$ too. Hence,  $N$ normalizes $\UUU_x$.
\ps
(d) Now, we identify $X$ with $U/U_x$ and with $\bbA^n$. The quotient morphism 
\[
\pi\colon \An = U/U_x \to U/NU_x
\]
is $U$-equivariant and $U/NU_x \simeq \bbA^{n-1}$, as in (a).
Since $N$ is central in $U$ and acts freely on $X$,   
Lemma~3.2\,(v) in \cite{Br2021Homogeneous-variet} shows that $\pi$ is a principal  $N$-bundle. Every principal $\Ga$-bundle over an affine variety is trivial (see Corollaire 6.6 \cite[III, \S\,4]{DeGa1970Groupes-algebrique}), thus $\pi$ can be identified with a linear projection $\An \to \bbA^{n-1}$.
\ps
(e) Since $N$ normalizes $\UUU_x$ the quotient morphism $\pi$ can be identified with the quotient $\UUU/\UUU_x \to \UUU/N\UUU_x$; as such, it is $\UUU$-equivari\-ant. 
Since $\pi$ is a geometric quotient for $N$ and $T$ normalizes $N$, 
$\pi$ is also $T$-equivariant.
Thus, $\pi$ is $\HHH$-equivariant, and the image of $\HHH$ in $\Aut(\bbA^{n-1})$ is a connected  solvable subgroup whose unipotent part acts with a dense orbit. Now, the conclusion follows by induction.
\end{proof}

\begin{cor}
\label{Cor.bound_dl}
If $X$ is quasi-affine and
$\UUU \subseteq \Aut(X)$ a connected unipotent solvable subgroup, then
$\dl(\UUU) \leq \max \set{\dim (\UUU x)}{x \in X}$.
\end{cor}

\begin{proof}
For all $x \in X$, the orbit $\UUU x$ is locally closed in $X$ and irreducible 
(Proposition~\ref{Orbits.prop}). 
Let $\UUU^{(x)}$ be the image of $\UUU$ in $\Aut(\UUU x)$. It is a connected unipotent solvable subgroup of $\Aut(\UUU x)$
acting transitively on $\UUU x$. By Proposition~\ref{Prop.U_transitive}, 
$\dl(\UUU^{(x)}) \leq \dim \UUU x$. Since 
$\dl(\UUU) = \max \{\dl(\UUU^{(x)})\mid x \in X\}$, the claim follows.
\end{proof}

\begin{lem}
    \label{Lem.affine-morphism} 
    Let $\pi \colon X \to Y$ be a dominant morphism of irreducible varieties
    such that $\Xgg$ is isomorphic to the affine space 
    of dimension $k$.
    \begin{enumerate}[\rm (1)]
        \item  There is a flat morphism $U \to Y$ with 
               $X \times_Y U \simeq \bbA^k \times U$.
        \item  The morphism $\pi$ is affine over an open dense subset of $Y$. 
    \end{enumerate}
\end{lem}

\begin{proof}
  Assertion (2) follows from (1) by faithfully-flat descent
  (Proposition 14.53 in \cite{GoWe0Algebraic-geometry-I}).
  Let us prove Assertion (1).
  We may assume that $\pi$ is flat and $Y$ is affine and irreducible.
  Denote by $\KK / \kk(Y)$ an algebraic closure.
  Let $f_1, \ldots, f_k \in \OOO(\Xgg) = \OOO(X) \otimes_{\OOO(Y)} \KK$
  be the coordinate functions of an isomorphism 
  $f \colon X_{\KK} \xrightarrow{\sim} \bbA^k_{\KK}$. 
  Replacing $\pi$ with
  the pull-back by a flat morphism $Y' \to Y$, 
  we may assume that $f_1, \ldots, f_k \in \OOO(X)$, i.e.,
  $f$ is equal to the pull-back $\varphi_{\KK}$ of a 
  $Y$-morphism $\varphi \colon X \to \bbA^k \times Y$. Faithfully-flat descent
  applied to the pull-back via 
  $\bbA^k_{\KK} \to \bbA^k_{\kk(Y)}$ 
  gives that $\varphi_{\kk(Y)} \colon X_{\kk(Y)} \xrightarrow{\sim} \bbA^k_{\kk(Y)}$ is an isomorphism.
  By Proposition 5.7 in Exposition~I of \cite{GrRa2003Revetements-etales}
  and Proposition~9.6.1 of \cite{Gr1966Elements-de-geomet-28} we find an 
  open dense subset $U \subseteq Y$ such that the restriction of
  the $Y$-morphism $\varphi$ over $U$ gives a $U$-isomorphism
  $\pi^{-1}(U) \simeq \bbA^k \times U$, hence Assertion (1).
  %
\end{proof}

\begin{proof}[Proof of Theorem~\ref{main-thm-derived-length}] Assertion (1) claims that a connected solvable subgroup $\GGG \subset \Aut(X)$ has derived length $\leq \dim X + 1$.
For the proof we argue by induction on $n:=\dim(X)$, starting with the trivial case $n=0$. Now, we assume $n \geq 1$ and that the statement has been obtained in all dimensions $\leq n-1$.
Considering the action of $\GGG$ on each of its orbits, we may assume that $\GGG$ acts transitively on $X$. Arguing by contradiction, we assume $d \coloneqq \dl(\GGG) \geq n + 2$.

  \ps
  (a)
  Corollary~\ref{Cor.global_geom_quotient}
  gives a geometric quotient $\pi \colon X \to Y$ for $\HHH \coloneqq \GGG^{(d-1)}$.
  Let $\KK / \kk(Y)$ be an algebraic closure.
  For a connected subgroup
  $\BBB \subset \Aut_Y(X)$ 
  we denote by 
  $\BBB_{\KK} \subseteq \Aut(\Xgg)$ the subgroup generated by the morphisms
  $\rho_{\KK} \colon A_{\KK} \to \Aut(\Xgg)$ for all morphisms 
  $\rho \colon A \to \Aut(X)$ with irreducible $A$ and 
  $\id_X \in \rho(A) \subseteq \BBB$. 
  By Proposition~\ref{commutative-subgroup.prop} and
  Corollary~\ref{Cor.Geometric_fibre_of_Geom_quotient}, $\HHH_{\KK}$ is a commutative
  algebraic group acting transitively on $\Xgg$ and $\dim \HHH_{\KK} = \dim \Xgg$.

\ps
(b)
Set $m := \dim Y \leq n-1$ and $\KKK:=\GGG^{(m+1)}$. By the induction, $\KKK$ acts trivially on $Y$, and $ \KKK^{(1)}=\GGG^{(m+2)}$ contains $\HHH$, because $m+2 \leq n+1\leq d-1$.  
The (commutative) algebraic group $\HHH_{\KK} / R_u(\HHH_{\KK})$ is a semi-abelian variety, i.e., an extension of an abelian variety with a torus, which acts transitively
on $Z \coloneqq \Xgg / R_u(\HHH_{\KK})$.  
In particular, $Z$ is semi-abelian, and so $\Aut(Z)^\circ = Z$. 
In fact, every automorphism of $Z$ is a group automorphism up to translation with an element of $Z$ (see Theorem~3 in \cite{Ro1961Toroidal-algebraic}), and the 
group automorphisms of $Z$ form a discrete group 
(cf. Proposition~4.3.2 of \cite{BrSaUm2013Lectures-on-the-st}).
It follows that
$\KKK_{\KK}^{(1)}$ acts trivially on $Z$. 
Since $\HHH\subset \KKK^{(1)}$, $\KKK^{(1)}_{\KK}$ acts transitively on $\Xgg$, and we conclude that 
$\HHH_{\KK} = R_u(\HHH_{\KK})$, and so $\Xgg \simeq \bbA_{\KK}^{n-m}$ 
(by Theorem~1 in \cite{Br2021Homogeneous-variet}). 

 By Corollary~\ref{corB}, $\KKK_{\KK}^{(1)} \subseteq R_u(\KKK_{\KK})$, and so $R_u(\KKK_{\KK})$ acts transitively on $\Xgg$. Now, up to conjugation, $\KKK_{\KK} \subseteq \Jonq_{\KK}(n-m)$ by Proposition~\ref{Prop.U_transitive}. In particular, $\KKK_{\KK}^{(2)}=\GGG^{(m+3)}$ does not act transitively on $\Xgg$, and so $m+3 \geq d$, which implies $n = \dim X =m+1 = \dim Y + 1$. Then, by Lemma~\ref{Lem.affine-morphism} 
  and the $\GGG$-transitivity, 
  all $\HHH$-orbits are isomorphic to~$\bbA^1$.
 
  \ps
  (c)  
  Let $X' \subseteq X$ be an orbit of $\GGG^{(n-1)}$ and 
  $\FFF \subseteq \Aut(X')$ be the image of $\GGG^{(n-1)}$.
  We will show
  $\FFF^{(2)} = \{\id\}$, implying the contradiction $\GGG^{(n+1)} = \{ \id \}$.

  Let $\pi' \colon X' \to Y'$ be a geometric quotient 
  for the image of $\HHH$ in $\Aut(X')$ (Corollary~\ref{Cor.global_geom_quotient}). 
  As all $\HHH$-orbits are isomorphic to $\mathbb{A}^1$,
  $\pi'$ has  fibers isomorphic to $\bbA^1$.
  By Lemma~\ref{Lem.affine-morphism} and the $\FFF$-transitivity, $\pi' \colon X' \to Y'$ is 
  affine and locally a trivial $\bbA^1$-bundle after a flat base-change, as in (b). 
  By induction, $\KKK' \coloneqq \FFF^{(1)}$ acts trivially on $Y'$.
  
  Restricting the $\KKK'$-action to $(\pi')^{-1}(Y_0)$ for an
  open dense affine subset $Y_0 \subseteq Y'$, Theorem~\ref{main-thm-A} and Lemma~\ref{Lem.restriction-to-open} yield that $\KKK'$
  is \algebraicallyTame{},
  and $\KKK' = T \ltimes R_u(\KKK')$ for a torus $T \subseteq \KKK'$.
  As the fibers of $\pi'$ are affine lines, $R_u(\KKK')$ is commutative.
  
  If $T = \{\id\}$, then $\FFF^{(2)} = \{ \id \}$. 
  Otherwise,  $\dim(T) \geq 1$.
  Given any flat morphism  $U \to Y'$ such that $U$ is irreducible and 
  $X' \times_{Y'} U \simeq \bbA^1 \times U$, 
  we have $\Aut_U(\bbA^1 \times U)^\circ = (\OOO(U)^\ast)^\circ \ltimes \OOO(U)$
  and $(\OOO(U)^\ast)^\circ \simeq \mathbb{G}_m$ (see Example~\ref{nested.exa}(5));
  thus, the fixed points  of $T$ on 
  $X' \times_{Y'} U \simeq \bbA^1 \times U$ 
  give a section $s_U$  and 
  $T = \set{\varphi \in \Aut_U(\bbA^1 \times U) }{\varphi(s_U) = s_U}^\circ \simeq \mathbb{G}_m  $. 
  Hence, the fixed points of  $T$ on $X'$ provide a section 
  $e \colon Y' \to X'$ of $\pi'$
  and $T = \set{\varphi \in \Aut_Y(X)}{e = \varphi \circ e}^\circ$.
  
  Note that $\pi'$ is a geometric quotient for
  $R_u(\KKK')$. 
  By Propositions~\ref{commutative-quotient2.prop} 
  and~\ref{commutative-quotient3.prop}, $\pi'$
  has the structure of a $Y'$-group scheme with identity $e$ and 
  $\FFF \subseteq \LP(\FFF) \ltimes \Gamma(\pi')$.
  The image $F \subseteq \Aut(Y')$ of $\FFF$ is an algebraic group, 
  as it is commutative and acts transitively on $Y'$
  (Proposition~\ref{commutative-subgroup.prop}). 
  The kernel of $\LP(\FFF) \to F$ is finite dimensional, 
  for its connected component is $T$. By Lemma~\ref{Lem.Further_crit_to_be_alg}, $\LP(\FFF)$ is an algebraic group
  and $\LP(\FFF)^{(1)} \subseteq T$. Thus 
  $\LP(\FFF)$ is commutative (see Section~\ref{Subsec.Commutator_torus}) and 
  hence $\FFF^{(2)} = \{ \id \}$. This gives Assertion (1).

\ps
(d) Now suppose that $\dl(\UUU) \geq n$ and $X$ is connected and quasi-affine. 
By Corollary~\ref{Cor.bound_dl},  $\dl(\UUU)=n$ and $\UUU$ has an orbit $O$ of dimension~$n$.
Since $O \subset X$ is closed, $O$ is a connected component of $X$ and so $O = X$.
Then, Proposition~\ref{Prop.U_transitive} gives Assertion~(2), namely 
$X \simeq \An$ and $\HHH$ is conjugate to a subgroup of $\Jonq(n)$.
\end{proof}

\subsection{Nested and \algebraicallyTame{} subgroups}\label{sec,nested-and-tame}
The following result gives Theorem~\ref{thm-nested-tame}.

\begin{thm}
\label{Thm.Equivalence_alg_tame_nested}
For a subgroup $\GGG \subseteq \Aut(X)$ the following are equivalent:
\begin{enumerate}[\rm (1)]
\item \label{Eq1.nested_alg_tame} $\GGG$ is \algebraicallyTame{};
\item \label{Eq2.nested_alg_tame} $\GGG$ admits a covering directed set of connected
algebraic subgroups;
\item \label{Eq3.nested_alg_tame} $\GGG$ is nested by connected algebraic subgroups.
\end{enumerate}
\end{thm}

\begin{cor}\label{nested-strongly-nested.cor}
If $X$ is affine and $\GGG \subset \Aut(X)$ is a closed subgroup that is nested by connected algebraic groups, then $\GGG$ is strongly nested and \algebraicallyTame{}.
\end{cor}
\begin{proof}
By Theorem~\ref{Thm.Equivalence_alg_tame_nested}, $\GGG$ is \algebraicallyTame{}. Let $\GGG = \bigcup_k\GGG_k$ be an admissible filtration where we can assume that each $\GGG_k$ is irreducible and contains $e_\GGG$ (Lemma~\ref{connected-indgroups.lem}). Then $G_k:=\langle \GGG_k\rangle \subset \GGG$ is a connected algebraic group and $\GGG = \bigcup_k G_k$ is an admissible filtration.
\end{proof}

Before proving Theorem~\ref{Thm.Equivalence_alg_tame_nested} we start with a family of examples to illustrate some of the difficulties. 

\subsubsection{A construction and an example}\label{par:NewPerepechko}
The following construction is inspired by Example~5 in \cite{St2012On-the-topologies-}.

\begin{cons}
  \label{cons.NewPerepechko}
  We define $ \Ga^\infty$ as the increasing union $\bigcup_{n\geq 1}\Ga^n$, 
  and we denote simply by $\Ga\subset \Ga^\infty$ (resp.\ $\Ga\subset \Ga^n$) the subgroup $\Ga\times\{(0, \ldots, 0, \ldots)\}$. 

  \textit{Assume $\Char(\kk) = p > 0$. Let $W$ be an $\FF_p$-subspace 
  of $\Ga$ of positive dimension such that  $\Ga / W$ is countable. There is 
  a subgroup $\UUU \subseteq \Ga^\infty$, nested by connected unipotent algebraic subgroups, with $\UUU \oplus W = \Ga^\infty$
  (direct sum as $\FF_p$-vectors spaces).}

   The index of $W$ being countable, we can choose an increasing sequence 
    $V_1 = \{ 0 \} \subseteq V_2 \subseteq \cdots \subseteq \Ga$ of finite
  $\FF_p$-subspaces such that their union $V = \bigcup_n V_n$ is an $\FF_p$-complement
  of $W$ in $\Ga$: $V \oplus W = \Ga$. We set $V_n'=V_n\times \{(0, \ldots, 0) \} \subseteq \Ga^n$.

   Let us   define a sequence of connected algebraic subgroups
  $U_{n+1} \subseteq \Ga^{n+1}$, each of codimension $1$ in $\Ga^{n+1}$. 
  We start with $U_1 = \{0\} \subseteq \Ga$ and  define $U_{n+1}$ by induction. 
  For this, remark that $U_n$ and $U_n+V_{n+1}'$ have codimension $1$ in $\Ga^n$ by construction.
  By Proposition~20.3 in \cite{Hu1975Linear-algebraic-g} there is an algebraic homomorphism
  $\varphi_n \colon \Ga^n \to \Ga$ such that ${\mathrm{Ker}}(\varphi_n)=U_n+V_{n+1}'$, and we set 
  \[U_{n+1}=\set{ (x_1, \ldots, x_{n},  \varphi_n(x_1, \ldots, x_n))}{(x_1, \ldots, x_{n})\in \Ga^n } \subset \Ga^{n+1}
  \]
  Obviously $U_{n} \times \{0\}\subset U_{n+1} \cap (\Ga^n \times \{0\})$.
  Moreover, $(x_1, 0 \ldots, 0)$  is in $U_{n+1}$ if and only if $\varphi_n(x_1, 0, \ldots, 0)=0$
  which is equivalent to $x_1\in V_{n+1}'$, by induction. Thus
    \begin{equation*}
    \label{Eq.containment}
    U_{n+1} \cap \Ga \times \{(0, \ldots, 0) \} 
    = V_{n+1}' \subseteq \Ga^{n+1} \, .
  \end{equation*}
 This implies that the projection 
  $U_{n+1} \to \Ga^n$, $(x_1, \ldots, x_{n+1}) \mapsto (x_2, \ldots, x_{n+1})$
  is surjective (for $n\geq 1$), because its kernel is finite. 
 
 Now, consider the group $\UUU = \bigcup_n U_n \subseteq \Ga^\infty$. 
  By construction, $V$ is the intersection of $\UUU$ with
  $\Ga $, so 
  $\Ga \subseteq \UUU + W$ and 
  $\UUU \oplus W = \Ga^\infty$. 
\end{cons}

  \begin{exa} 
  \label{exa.perepechko}
  Here, we identify $\kk[x]$ with $\Ga^\infty$  by 
  $ p(x)=\sum_{i\geq 0}^\infty a_i x^i\mapsto (a_i)_{i\geq 0}$.
  Doing so, the closed subgroup   $\set{(x, y) \mapsto (x, y + p(x))}{p \in \kk[x]}$ of $\Aut(\bbA^2)$ is identified to $\Ga^\infty$.
  Using Construction~\ref{cons.NewPerepechko} we obtain   
  \begin{itemize}
    \item[{\rm(a)}] If $\kk$ is countable of characteristic $p > 0$ and if we choose $W = \{0\}$, 
  then $\UUU = \bigcup_n U_n = \Ga^{\infty}$ is closed in $\Aut(\bbA^2)$,
  and $\UUU$ contains an algebraic subgroup $\Ga$ which is not contained 
  in any $U_n$.
   \item[{\rm(b)}] If  $\kk$ is of characteristic $p > 0$ and if we choose $W\neq \{0\}$ 
   with $\Ga/W$ infinite  countable, then $\UUU$ is not closed in the ind-group 
  $\Aut(\bbA^2)$, because $\UUU \cap \Ga$ is reduced to a proper dense subgroup of $\Ga(\kk)$.
  \end{itemize}
  In both cases, one can take $\Ga$ to be the group of translations 
  $(x,y)\mapsto (x+t,y)$, $t\in \kk$.
  \end{exa}

The  phenomenon described in Example~\ref{exa.perepechko}~(a) does not occur when $\kk$ is uncountable or when $\Char(\kk) = 0$. 
Indeed, if $\HHH = \bigcup_{k} H_k$ is a nested subgroup of $\Aut(X)$ and $G\subset \Aut(X)$ an algebraic
subgroup of $\HHH$, then $G$ is contained in one of the $H_i$, because 
$G$ is the increasing union of the closed subgroups 
$G \cap H_i$, and this sequence stabilizes
by Lemma~\ref{Lem.Increading_seq_of_algebraic_subgroups}.
The phenomenon described in Example~\ref{exa.perepechko}~(b) does not occur when 
$\Char(\kk) = 0$, by Theorem~\ref{main-thm-Perepechko}.

\subsubsection{Nestedness of $R_u^\SSS(\GGG)$}

\begin{prop}\label{unipotent-indgroup-is-tame.prop}
Let $X$ be an affine variety. Let $\UUU \subset \Aut(X)$ be a connected subgroup which admits a covering and directed set $\SSS$ of connected unipotent algebraic groups. Then $\UUU$ and $\bUUU$ are solvable and
\algebraicallyTame{}, $\UUU$ is nested by elements of $\SSS$ and $\bUUU$ is strongly nested by connected unipotent algebraic groups.
\end{prop}
\begin{proof}
If $U$ is a connected unipotent algebraic subgroup of $\UUU$, then  $\dl (U) \leq \dim X$ by Corollary~\ref{Cor.bound_dl}, and so $\UUU$ and $\bUUU$ are solvable of derived length $\leq \dim X$. Theorem~\ref{main-thm-A} implies that $\UUU$ and $\bUUU$ are \algebraicallyTame{}. 
As $\bUUU$ is a connected ind-group it admits an admissible filtration $W_1 \subseteq W_2 \subseteq \ldots$ where each $W_k$ is irreducible and contains
the identity (Lemma~\ref{connected-indgroups.lem}). 
Hence, $\bUUU$ is strongly nested, $\bUUU = \bigcup_k U_k$, where each $U_k = \langle W_k \rangle$ is a connected algebraic subgroup. 
The group $R_u(\bUUU)$ is closed in $\bUUU$, because 
$R_u(\bUUU) \cap U_k = R_u(U_k)$ for all $k$. 
As $\UUU \subseteq R_u(\bUUU)$,
we get $R_u(\bUUU) = \bUUU$ and thus $U_k$ is unipotent for all $k$.

It remains to see that $\UUU$ is nested by elements of $\SSS$. 
The subgroup of $U_k$ generated by all $V \in \SSS_k:=\set{V \in \SSS}{V\subset U_k}$ is  
generated by finitely many
$V_{k, 1}, \ldots V_{k, n_k} \in \SSS$ (Proposition 7.5~in \cite{Hu1975Linear-algebraic-g}).
Since $\SSS$ is directed, there exists an increasing sequence 
$V_1 \subseteq V_2 \subseteq \cdots$ of elements in $\SSS$ such that $V_k$ contains $V_{k, i}$ for all $i$. As $\SSS$ covers $\UUU$ we conclude $\UUU = \bigcup_k V_k$.
\end{proof}

\begin{prop}
\label{Pro.Unipotent_radical_second_part}
Let $\GGG$ be a subgroup of $\Aut(X)$.
Suppose that $\SSS$ is a covering and directed set of connected algebraic subgroups of $\GGG$.
Then $R_u^{\SSS}(\GGG)$ is solvable, \algebraicallyTame{}, and nested by connected unipotent algebraic 
subgroups. 
\end{prop}

\begin{proof}
Choose a connected, unipotent and algebraic subgroup $U \subset R_u^\SSS(\GGG)$ which has the same orbits as $R_u^\SSS(\GGG)$ (Proposition 2.4\,(4) in \cite{Note_on_Ramanujam}).
By Theorem~2 in \cite{Br2021Homogeneous-variet}, there is a dense affine open subset $X_0 \subset X$ which is stable under~$U$, hence also under $R_u^\SSS(\GGG)$.
By Lemma~\ref{Lem.restriction-to-open} the image $\UUU \subset \Aut(X_0)$ of $R_u^\SSS(\GGG)$ is connected and admits a covering directed set of connected unipotent algebraic groups. Now the claim follows from Proposition~\ref{unipotent-indgroup-is-tame.prop}.
\end{proof}

\subsubsection{Proof of Theorem~\ref{Thm.Equivalence_alg_tame_nested}}
(a) Remark~\ref{nested_alg_tame_covering.rem} shows that Assertions~\eqref{Eq1.nested_alg_tame} 
and~\eqref{Eq3.nested_alg_tame} both imply Assertion~\eqref{Eq2.nested_alg_tame}.
Now, we  assume that Assertion~\eqref{Eq2.nested_alg_tame} is satisfied. 
Let $\SSS$ be a covering and directed set of connected algebraic subgroups of $\GGG$.
By 
Propositions~\ref{Pro.Unipotent_radical}   and~\ref{Pro.Unipotent_radical_second_part},
$\GGG = G \UUU$ where $G\in \SSS$ is a connected algebraic subgroup normalizing $\UUU$, and $\UUU:=R_u^\SSS(\GGG)$ is 
\algebraicallyTame{}, solvable and  nested: $\UUU = \bigcup_kU_k$ where the $U_k$ are connected unipotent algebraic subgroups, $U_k\subset U_{k+1}$. Since $\UUU$ is \algebraicallyTame{} and normal we can assume that the $U_k$ are normalized by $G$. 
Thus, $\GGG = \bigcup_k GU_k$ is nested by the connected algebraic subgroups $GU_k$, which proves Assertion~\eqref{Eq3.nested_alg_tame}.
\ps
(b)
It remains to see that $\GGG$ is \algebraicallyTame{}. Let $\rho\colon A \to \Aut(X)$ be a morphism with 
$A$ irreducible and $\id_X\in\rho(A) \subset \GGG$. We can assume that $\rho(A)$ is stable under left and right multiplication by $G$. Then $\rho(A) = G (\rho(A)\cap \UUU)$, and $\rho(A)\cap\UUU$ is stable under conjugation  by $G$. Thus the subgroup $V:=\langle\rho(A)\cap\UUU\rangle$ is normalized by $G$ and $\langle \rho(A) \rangle = G \,V$. We claim that $V$ is contained in an algebraic subgroup of $\UUU$. Then $\langle \rho(A)\rangle$ is finite dimensional, 
hence an algebraic group, and we are done.

In order to prove the claim  we apply Theorem~2 of~ \cite{Br2021Homogeneous-variet} to get a dense and affine open subset $X_0 \subset X$ which is stable under $\UUU$. The image $\UUU'$ of $\UUU$ in $\Aut(X_0)$ is connected and solvable, so $\UUU'$ and its closure $\overline{\UUU'} \subset \Aut(X_0)$ are \algebraicallyTame{}, by Theorem~\ref{main-thm-A} and Lemma~\ref{connected-indgroups.lem}. 
By Lemma~\ref{Lem.restriction-to-open}, the set 
\[
A_0:=\{a \in A\mid \rho(a)(X_0) = X_0 \; \text{ and } \; \rho(a)|_{X_0} \in \overline{\UUU'}\}
\]
is closed in $A$. And the induced map $\rho_0\colon A_0 \to \Aut(X_0)$ is a morphism. Hence $\langle\rho_0(A_0)\rangle \subset \overline{\UUU'}$ is an algebraic group. By construction  $\rho(A) \cap \UUU \subset \rho(A_0)$, thus $V$ is contained in an algebraic group, hence the claim. \hfill $\square$

\subsection{Solvable and unipotent radicals} \label{unipotent_radical.sec}

\begin{thm} 
\label{existence_of_solvable_radical.thm}
Any subgroup $\GGG$ of $\Aut(X)$ contains a unique maximal connected normal and solvable subgroup $R(\GGG)$, called the solvable radical of $\GGG$.
\end{thm}

\begin{proof}
  Let $\HHH_1$ and $\HHH_2$ be connected normal solvable subgroups of $\GGG$. Then
  $\langle \HHH_1, \HHH_2 \rangle$ is again a connected normal solvable subgroup of $\GGG$. Hence, the set $\SSS$ of connected normal solvable subgroups $\HHH$ of $\GGG$ is directed.
  By Theorem~\ref{main-thm-derived-length}, 
  every $\HHH \in \SSS$ has derived length 
  $\leq \dim X + 1$.
  This shows that $\bigcup_{\HHH \in \SSS} \HHH$ is a connected normal solvable subgroup
  of $\GGG$, hence equal to $R(\GGG)$.
\end{proof}

\begin{rem}
The heart of the proof of Theorem~\ref{existence_of_solvable_radical.thm} is the upper bound $\dl(\HHH)\leq \dim(X)+1$ from  Theorem~\ref{main-thm-derived-length}. This can be replaced by the  simpler inequality $\dl(\HHH)\leq 2\dim(X)$. Indeed, for $n\geq 0$ let $d(n)$ be the maximum of the derived length of connected and solvable subgroups of $\Aut(X)$, where $X$ is any variety of dimension $n$. We want to prove $d(n)\leq 2n$. For $n=1$, $d(1)=2$, a bound reached by the group of affine transformations of the line. Now, suppose $d(k)\leq 2k$ for any $k\leq n-1$. 
 Pick $X$ of dimension $n$ and $\GGG$ a connected and solvable subgroup of $\Aut(X)$. 
 Set $d=\dl(\GGG)$ and $\HHH=\GGG^{(d-1)}$. Since $\HHH$ is normal, 
 Theorem~\ref{popov-rosenlicht.thm} and
 Corollary~\ref{Cor.global_geom_quotient} give a dense open subset $X_0\subset X$ and a $\GGG$-equivariant geometric quotient $\pi \colon X_0 \to Y_0$ for $\HHH$. 
 Set $\KKK= \GGG^{(d(m))}$ where $m=\dim(Y)$. Then, 
 the image of $\KKK$ in $\Aut(Y)$ is trivial
 and $\KKK^{(d(n-m))}=\{\id\}$ in restriction to any fiber of $\pi$. Thus, $d\leq d(m)+d(n-m)$ and finally $d(n)\leq 2n$ for every $n$. 
\end{rem}

\begin{thm}
  \label{thm.Existence_radical}
  Any subgroup $\GGG \subseteq \Aut(X)$ contains a unique maximal subgroup
  $R_u(\GGG)$ which is
  normal solvable and 
  generated by connected unipotent algebraic subgroups. It is 
  nested by connected unipotent algebraic subgroups,
  and its derived length is bounded by $\dim X$.
\end{thm}

We call $R_u(\GGG)$ \emph{the unipotent radical} of $\GGG$.
To prove this result, we start with the following proposition.

\begin{prop}
  \label{Prop.Estimate_dl}
  Let $\GGG \subseteq \Aut(X)$ be a solvable 
  subgroup such that 
  \begin{equation}
\tag{$\ast$}
\label{Covering_property}
\begin{minipage}{0.8\textwidth}
\textit{%
$\GGG$ is generated by images of morphisms $\mathbb{A}^\ell \to \Aut(X)$ containing $\id_X$ where $\ell \in \NN$.
}
\end{minipage}
\end{equation}
Then $X$ contains a $\GGG$-stable open affine dense subset.
\end{prop}

\begin{cor}
  \label{Cor.unipotent_solvable_subgroup}
  If $\GGG \subseteq \Aut(X)$ is a solvable subgroup, generated by 
  connected unipotent algebraic subgroups, then $\GGG$ is \algebraicallyTame{}, $\dl(\GGG) \leq \dim X$ and all connected algebraic subgroups are unipotent.
\end{cor}

\begin{proof}
  By Proposition~\ref{Prop.Estimate_dl} above and  Lemma~\ref{Lem.restriction-to-open} we can assume that $X$ is affine. Then the claims follow from Corollary~\ref{corB} and 
  Corollary~\ref{Cor.bound_dl}.
\end{proof}

\begin{proof}[Proof of Proposition~\ref{Prop.Estimate_dl}]
   We may assume $X$ irreducible and $\GGG \neq \{\id\}$. 
  We proceed by induction on $n:=\dim(X)$. 
  The case $n = 1$ is clear, thus now $n \geq 2$.
  We set $\HHH = \GGG^{(d-1)}$ for $d = \dl(\GGG)$: 
  it is non-trivial, commutative, and  satisfies~\eqref{Covering_property}. 
  
  By Theorem~\ref{popov-rosenlicht.thm} and Remark~\ref{Rem.Geom_quotient_smooth} 
  there is an 
  $\HHH$-stable open dense $X' \subseteq X$ admitting a smooth geometric quotient 
  $\pi' \colon X' \to Y'$ for $\HHH$. 
  Set $m=\dim(Y')$ and denote by $\KK / \kk(Y)$ an algebraic closure.
  Let $\HHH_{\KK} \subseteq \Aut(\Xgg')$ be the subgroup generated by 
  the images of all $\rho_{\KK} \colon \mathbb{A}^l_{\KK} \to \Aut(\Xgg')$, 
  for all morphisms $\rho \colon \mathbb{A}^l\to \Aut(X)$ with $\id \in \rho( {\mathbb{A}}^l)\subset \HHH$. By Proposition~\ref{commutative-subgroup.prop} and Corollary~\ref{Cor.Geometric_fibre_of_Geom_quotient},  
  $\HHH_{\KK} \subseteq \Aut(\Xgg')$ is a commutative algebraic subgroup that acts transitively on~$\Xgg'$. 
  And $\HHH_{\KK}$ is unipotent because every morphism from 
  $\mathbb{A}^\ell$ to $\Gm$ or an abelian variety is constant
  (see Proposition~3.10 in \cite{milneAV}).
  Hence $\Xgg' \simeq \bbA^{n-m}_{\KK}$ and,  
  after shrinking $Y'$, we may assume that $\pi'$ is affine (Lemma~\ref{Lem.affine-morphism}).
  Then a $\GGG$-equivariant 
  affine geometric quotient $\pi_0 \colon X_0 \coloneqq \GGG X' \to Y_0$ for $\HHH$ exists by Corollary~\ref{Cor.global_geom_quotient} and Remark~\ref{Rem.global_geom_quotient}.
  
  As the image of $\GGG$ in $\Aut(Y_0)$ satisfies~\eqref{Covering_property},
  by induction there is an open dense affine $\GGG$-stable $V \subseteq Y_0$
  and  $\pi_0^{-1}(V) \subseteq X_0$ is our desired subset.
\end{proof}
 
\begin{proof}[Proof of Theorem~\ref{thm.Existence_radical}]
  Let $\SSS$ be the set of normal solvable subgroups of $\GGG$
  that are generated by connected unipotent algebraic subgroups.
  If $\HHH_1, \HHH_2 \in \SSS$, then 
  $\langle \HHH_1, \HHH_2 \rangle \in \SSS$, i.e.\ $\SSS$ is directed. 
  By Corollary~\ref{Cor.unipotent_solvable_subgroup}, 
  every $\HHH \in \SSS$ is \algebraicallyTame{}, $\dl(\HHH) \leq \dim X$ 
  and all connected algebraic 
  subgroups are unipotent. Thus
  $\HHH_0 \coloneq \bigcup_{\HHH \in \SSS} \HHH \in \SSS$
  and $\HHH_0 = R_u(\GGG)$. 
  Theorem~\ref{Thm.Equivalence_alg_tame_nested} completes the proof.
\end{proof}

When $\GGG$ is \algebraicallyTame, the definition of the 
unipotent radical from Theorem~\ref{thm.Existence_radical} coincides with the one used 
in Remark~\ref{first_definition_unipotent_radical.rem}. Moreover, if $\SSS$
is any covering directed set of connected algebraic subgroups of $\GGG$, one can show that (a) $R_u^\SSS(\GGG) = R_u(\GGG)$ if $\Char(\kk)=0$ or $\kk$ is uncountable and (b) $R_u^\SSS(\GGG)\subset R_u(\GGG)$ has index bounded by $p^{3\dim(X)^2}$ if $\Char(\kk)=p>0$. See Corollary~\ref{Cor.R_u^SSS(GGG)_R_u(GGG)} in Section~\ref{Appendix}. The following example illustrates this phenomenon.

\begin{exa}
  \label{Exa.Unipotent_radical}
  \textit{Assume $\Char(\kk) = p > 0$ and  $\kk$ is countable.
  There is a smooth variety $X$ of dimension $3$ and a 
  subgroup $\GGG = \bigcup_n G_n \subseteq \Aut(X)$,
  nested by a countable family $\SSS =  \{G_1, G_2, \ldots \}$ of connected algebraic subgroups, such that 
  $R_u^\SSS(\GGG)$ has index $p$ in $R_u(\GGG)$.}

  Let $E$ be an elliptic curve containing a finite cyclic subgroup of order $p$.
  Then the $p$-torsion $E[p]$ is isomorphic to $\ZZ/p\ZZ$. The prime field of $\kk$ 
  being equal to $\FF_p$, we obtain two embeddings of the
  cyclic group $(\ZZ/p\ZZ, +)$, one in $(E, +)$ and one in $(\kk, +)$:
  \begin{align*}
  \iota_E \colon \ZZ/p\ZZ\to E[p]\subset E \; {\textrm{ and }} \; 
  \iota_\kk \colon \ZZ/p\ZZ\to \FF_p\subset \Ga\simeq \kk.
  \end{align*}
  Let $X$ be the geometric quotient of $E \times \bbA^2$ by the $\ZZ/p\ZZ$-action
  $\gamma \cdot (e, x, y) = (e + \iota_E(\gamma), x - \iota_\kk(\gamma), y)$, with $\gamma \in \ZZ/p\ZZ$. 
    
  Then, we identify $\Ga^{\infty}\simeq \kk[y]$ with 
  $\set{(e, x, y) \mapsto (e, x + p(y), y)}{p \in \kk[y]} \subseteq \Aut(X)$
  and $E$ with $\set{(e, x, y) \mapsto (e + f, x, y)}{f \in E} \subseteq \Aut(X)$. 
  The intersection $E \cap \Ga^\infty \subset \Aut(X)$ is identified to $E[p]$ and $\FF_p$: it consists precisely of the automorphisms
  $(e, x, y) \mapsto (e + \iota_E(\gamma), x, y) = (e, x + \iota_\kk(\gamma), y) \; \textrm{ for $\gamma \in \ZZ/p\ZZ$}$.
    
Let $\GGG \subseteq \Aut(X)$ be the subgroup generated by $E$ and $\Ga^\infty$.
This is a commutative and connected group.  Its unipotent radical coincides with the image of  $\Ga^\infty$
in $\Aut(X)$. 

Set $W=\FF_p=\iota_\kk(\ZZ/p\ZZ)\subset \Ga$, and apply Construction~\ref{cons.NewPerepechko}. 
This gives an increasing union of algebraic subgroups $U_n\subset \Ga^{n}$ such that their union $\UUU=\bigcup_nU_n$ satisfies $\UUU\oplus W=\Ga^\infty$ and hence
$(\UUU\cap \Ga)\oplus W=\Ga$ (as $\FF_p$-linear subspaces).

  Define $G_n = E \times U_n \subseteq \Aut(X)$. Each $G_n$ contains the finite group $E[p]$. Thus
  $\GGG = \bigcup_n G_n \subseteq \Aut(X)$ and $\SSS = \{G_1, G_2, \ldots \}$ is a covering and directed set of connected algebraic subgroups for $\GGG$.  On the other hand, $R_u^\SSS(\GGG) = \UUU$ for $\SSS = \{G_1, G_2, \ldots \}$ and its index in $R_u(\GGG)$ is equal to $p$.
\end{exa}

\subsection{Borel subgroups of $\bbA^n$}\label{par:Borel_Subgroups}

Recall from \S~\ref{Borel-subgroups} that a Borel subgroup of $\Aut(X)$ is a connected solvable subgroup which is maximal among all such subgroups for inclusion. The next result is a consequence of Theorem~\ref{main-thm-derived-length}
and Zorn's lemma.

\begin{cor}
\label{Cor:Borel}
Every connected solvable subgroup of $\Aut(X)$ is contained in a Borel subgroup. 
\qed
\end{cor}

\begin{cor}
The group $\Jonq(n)$ is its own normalizer and is maximal among all solvable subgroups of $\Aut(\An)$.
In particular, $\Jonq(n)$ is a Borel-subgroup of $\Aut(\An)$.
\end{cor}
\begin{proof} 
We provide a characteristic-free proof (see also \cite{FuPo2018On-the-maximality-} when $\kk=\CC$). 

(a) For each  integer $ j\geq 0$,  set $D^j=\Jonq_u(n)^{(j)}$, 
the $j$-th derived subgroup. The projection 
$\An\to \bbA^{j}$  from Equation~\eqref{Eq.coflag} is invariant under $D^{j}$ and is a quotient map, that is
$\kk[x_1,\ldots,x_{n}]^{D^{j}} = \kk[x_1,\ldots,x_{j}]$.
Similarly,  $\kk[x_1,\ldots, x_{j+1}]^{D^{j} } = \kk[x_1,\ldots,x_{j}]$.
Let $\UUU \subset \Jonq_u(n)$ be any subgroup of derived length~$n$. 
Since $\UUU^{(j)}$ is a non-trivial subgroup of $\Jonq_u(n)^{(j)}$, we see that 
$\kk[x_1,\ldots,x_{n}]^{\UUU^{j}} = \kk[x_1,\ldots,x_{j}]$ and $\kk[x_1,\ldots, x_{j+1}]^{\UUU^{j} } = \kk[x_1,\ldots,x_{j}]$.

 (b) If $\varphi\in \Aut(\An)$ conjugates $\Jonq_u(n)$ to a subgroup $\UUU:=\varphi \Jonq_u(n)\varphi^{-1}$ of $\Jonq_u(n)$, 
Step (a) implies that $\varphi\in \Jonq(n)$, because it preserves the coflag from Equation~\eqref{Eq.coflag}. 
 In particular, $\Jonq(n)$ is its own normalizer in $\Aut(\An)$.

(c) Let $\HHH \supseteq\Jonq(n)$ be a solvable subgroup of $\Aut(\An)$.
Its closure $\bHHH$ is solvable as well and $\bHHH^\circ \supseteq \Jonq(n)$.
Theorem~\ref{main-thm-derived-length}\,(2) gives a $\varphi\in \Aut(\An)$
with $\varphi \bHHH^\circ \varphi^{-1} \subseteq \Jonq(n)$. 
Since $\Jonq(n)\subset  \bHHH^\circ$, we obtain 
$\varphi \Jonq(n)\varphi^{-1}\subset \Jonq(n)$ and we conclude with Step (b):   $\varphi$ is in $\Jonq(n)$, $\bHHH^\circ=\Jonq(n)$, and then $\HHH=\Jonq(n)$ 
because it normalizes $\bHHH^\circ$. 
\end{proof}

\begin{prop}\label{Prop.Borel_dim2}
Every connected solvable subgroup of $\Aut(\Atwo)$ is conjugate to a subgroup of $\Jonq(2)$. All Borel-subgroups of $\Aut(\Atwo)$ are conjugate.
\end{prop} 
 When $\kk=\CC$, this   is Theorem~1 of~\cite{Berest-Eshmatov-Eshmatov}. The proof given by Berest and Eshmatov is based on Jung's theorem and Bass-Serre theory of amalgamated products. 
We  provide a characteristic-free proof based on Theorem~\ref{main-thm-derived-length}. We refer 
to Proposition~\ref{Prop.Exotic_Borels} and Theorem~\ref{infmanyBorels} for the existence of other Borel subgroups in dimension $n\geq 3$.

\begin{proof}
As in Corollary~\ref{corB}, write $\HHH = T\ltimes \UUU$ with a torus $T$ of dimension $\leq 2$.
If $\UUU$ is trivial, then $\HHH$ is a torus and the claim follows from \cite{Gu1962The-action-of-an-a}. 
If $\UUU$ is not commutative, then $\UUU$ has a dense orbit, by Corollary~\ref{Cor.bound_dl}, and the claim follows from Theorem~\ref{main-thm-derived-length}. 
Thus we can assume that the connected, unipotent group $\UUU$ is  non-trivial and commutative. 
As $\HHH$ is \algebraicallyTame{}, Proposition~\ref{Thm.Torus_actions_on_U} provides an
 algebraic subgroup $N$ of $\UUU$, which is isomorphic to $\Ga$  and is normalized by $T$.
 As $\UUU$ is commutative, $N$ is normalized by $\HHH$.
By \cite{Re1968Operations-du-grou} (for $\Char(\kk) = 0$) resp. by \cite{Mi1971Ga-action-of-the-a} (for $\Char \kk >0$) the group $N$ is conjugate to a subgroup fixing the variable $x$, and so $\HHH$ stabilizes the coflag $x\colon \Atwo \to \Aone$.
\end{proof}

\subsection{Connected nested subgroups are closed}\label{proof_perepechko_sec}

Here, we prove Theorem~\ref{main-thm-Perepechko} which states that a connected nested subgroup of $\Aut(X)$  is closed when $X$ is affine and $\Char(\kk) = 0$. This doesn't hold when $\Char(\kk) > 0$, see Example~\ref{exa.perepechko}.

\begin{proof}[Proof of Theorem~\ref{main-thm-Perepechko}]
(a) 
By Theorem~\ref{main-thm-A}, $\GGG$ is \algebraicallyTame{}. 
Let $L$ be a connected reductive algebraic group of maximal dimension in $\GGG$ (see Lemma~\ref{Quotient_by_unipotent_radical_bounded.lem}). 
Every connected algebraic subgroup $G$ in $\GGG$ is affine
(see Lemma~\ref{quasiaffine-gives-affine.lem}) and 
from Mostow's theorem we get 
$G = L \ltimes R_u(G)$ if $G$ contains $L$ (Example~\ref{mostow_theorem.exa}). 
Hence $\GGG = L \ltimes \UUU$, where
$\UUU :=R_u(\GGG)$ is the unipotent radical.

From Proposition~\ref{unipotent-indgroup-is-tame.prop}, $\UUU$ and its closure $\bUUU \subset \Aut(X)$ are \algebraicallyTame{}, $\UUU$ is nested ($\UUU = \bigcup_k U_k$ with the $U_k$ unipotent), and $\bUUU$ is strongly nested by unipotent algebraic subgroups. If $U \subset \bUUU$ is a closed unipotent algebraic subgroup, then $U \cap \UUU = \bigcup_k U\cap U_k$, which is an ascending filtration of $U$ by connected closed unipotent subgroups. Hence, $U \cap \UUU = U \cap U_k$ for $k$ large enough. Since $\bUUU$ is strongly nested this implies that $\UUU$ is closed in $\bUUU$ and so $\UUU = \bUUU$. 
\ps
(b) 
We can now assume that
$\UUU$ is strongly nested by some unipotent algebraic groups $U_k$.
Let $\Aut(X) = \bigcup_k A_k$ be an admissible filtration.  
We may assume further that $A_k$ is invariant under left multiplication by $L$ 
and that $L$ normalizes each $U_k$. 
For each $k$ there is some $j$ such that $\UUU \cap A_k \subseteq U_j$.
Hence, $\GGG \cap A_k = L(\UUU \cap A_k)$ is a closed subset of the algebraic 
subgroup $L \ltimes U_j$ and 
thus $\GGG$ is closed in $\Aut(X)$.
\end{proof}

\subsection{Solvable subgroups of $\Aut(\An)$ containing an $n$-dimensional torus}

\begin{thm}\label{Tn-subgroup.thm}
If $\HHH \subset \Aut(\An)$ is a connected solvable 
subgroup containing a torus $T$ of dimension $n$,  then $\HHH$ is conjugate to a subgroup of 
$\Jonq(n)$.
\end{thm}

To prove this theorem, denote by $T_n\subset\GL_n\subset \Aut(\An)$ the maximal torus of diagonal matrices, and recall that a {\emph{root subgroup}} $U \subseteq \Aut(\bbA^n)$ is a connected 
one-dimensional unipotent subgroup normalized by $T_n$.

\begin{exa}
Let $(a,b)$ be a pair of non-negative integers.
For $s\in \kk$, let $f_s$ be the automorphism of $\bbA^3$ defined by $f_s(x_1,x_2,x_3)=(x_1,x_2,x_3 + s x_1^ax_2^b)$.
If $t=\diag(t_1, t_2, t_3) \in T_3$, then
$ t^{-1} f_st= f_{s'} $
with $s'=st_1^{a}t_2^{b}t_{3}^{-1}$. In other words, the additive group $\{f_s\mid s\in \Ga\}$ 
is a root subgroup.
\end{exa}

This example is typical. Indeed, 
\cite[Theorem~3.5]{LaLi2016Additive-group-act}
shows that every root subgroup of $\Aut(\bbA^n)$ is of the form
\begin{equation}
\label{Eq.Root_subgroup}
U=\set{ (x_1,\ldots,x_j+s m_j, x_{j+1},\ldots,x_n)}{ s \in \kk}
\end{equation}
for some $j \in \{1, 2 \ldots,  n\}$ where $m_j$ is a monomial in the variables $x_i$ with $i\neq j$.

\begin{proof}[Proof of Theorem~\ref{Tn-subgroup.thm}]
By Corollary~\ref{corB},  $\HHH$ is  equal to a semi-direct product $T \ltimes \UUU$ where $T$ is a torus of dimension $n$ and $\UUU$ is connected unipotent and solvable. 
Since all tori of dimension $n$ in $\Aut(\An)$ are conjugate (see \cite{Bi1966Remarks-on-the-act}) we can assume $T =T_n$. 
Let 
$q_j \colon \An \to \bbA^{n-1}$ be the projection that ``forgets'' the variable $x_j$. 
By induction on $n$, it suffices to find an index $j$ such that   $q_j$ is $\HHH$-equivariant. 
Arguing by contradiction, we find elements $u_j\in \UUU$ such that $u_j$ does not preserve the fibration defined by $q_j$. By connectedness and Theorem~\ref{main-thm-A},  there is a connected algebraic subgroup $H\subset \HHH$ containing $T$ and the $u_j$. Write 
$H=T\ltimes U$ for some connected unipotent algebraic group $U$ containing the~$u_j$.  Proposition~\ref{Thm.Torus_actions_on_U}
provides a central root subgroup $V\subset U$ for~$T$. In particular, $V$ is normalized by $H$.
The explicit description of root subgroups given in Equation~\eqref{Eq.Root_subgroup} shows that, for some index $j$, the general orbit of $V$ is a fiber of $q_j$. But then, $H$, hence also $u_j$, should preserve the fibration defined by~$q_j$. This contradiction concludes the proof.  
\end{proof}

\section{Appendix}\label{Appendix}
{\small{

\subsection{Non-conjugate Borel subgroups} \label{Subsec.Exotic_Borel}

The following proposition extends Corollary~1.5 
of~\cite{ReUrSa2025Group-theoretical-} to any field.

\begin{prop}
  \label{Prop.Exotic_Borels}
  For $n \geq 3$, $\Aut(\An)$ contains Borel subgrous 
  that are non-conjugate to $\Jonq(n)$.
\end{prop}

\begin{proof} The proof is the same as in~\cite{ReUrSa2025Group-theoretical-}, except that we define $f = x_1 x_3 + x_2^2 + x_4^k + x_5^k 
  + \ldots + x_n^k$, where $k = 2$ if $\Char(\kk) \neq 2$ and
  $k = 3$ if $\Char(\kk) = 2$.
\end{proof}

\begin{thm}\label{infmanyBorels}
If $\Char(\kk) > 0$ and $n\geq 4$, there
are infinitely many pairwise non-conjugate Borel subgroups in $\Aut(\mathbb{A}^n)$.
\end{thm}

\begin{proof}
We rely on Gupta's counterexample to Zariski's Cancellation Problem~\cite{GuptaCansel_counterex_any_dim}.	
Gupta considered the $\kk$-algebras:
 \[
 A_{(r,h)} =\kk[x_1,\dots,
 x_m,y,z,t]/(y x_1^{r_1}\dots
  x_m^{r_m}
 - h(z,t))
 \]
 where $m \ge 1$, the list of  exponents $r=(r_1, \ldots, r_m)$ satisfies $r_i\geq 2$ for all indices, and 
  $\kk[z,t]/\langle h\rangle \simeq \kk[x]$, but
  $h(z,t) \in \kk[z,t]$ is not a variable of $\kk[z,t]$ (i.e.,\   it is not the image of the variable $z$ by an automorphism of $\bbA^2$). For such an algebra $A=A_{(r,h)}$, 
the affine variety $X_A = \Spec(A)$ has dimension $m+2 \ge 3$. 
Gupta  proved  that
$X_A \times \mathbb{A}^1$ is isomorphic to  $\mathbb{A}^{m+3}$ but
 $X_A$ is not isomorphic to $\mathbb{A}^{m+2}$.
 Moreover, according to Theorem 4.3 and Corollary 4.4 of \cite{GhoshGuptatriv_of_family_hypers}, distinct parameters $r=(r_i)$ and $h$ provide non-isomorphic algebras.
 
 \ps
 (a) Fix such an algebra $A$. Let $U_A$ be the subgroup of $\Aut(X_A \times \bbA^1)$ generated by the automorphisms
\[ 
\TP_f \colon  (x, w) \mapsto 
(x,w + f(x)),
\]
where $f \in \OOO(X_A) = A$. 
First note that $U_A$ coincides with its centralizer, and this group is not a Borel subgroup (because of the automorphisms $(x,w)\mapsto (x, aw)$ with $a\in \kk^\times$). So, 
by Corollary \ref{Cor:Borel}, $U_A$ is contained in
a Borel subgroup $B_A$ of $\Aut(X_A \times \bbA^1)$, which is not commutative.
Consider the derived  series
$
B_A \supset B_A^{(1)} \supset \dots \supset B_A^{(k)}  \supset  \{ \id \}
$, where $k+1 =\dl(B_A)\geq 2$.
Denote by $W_A \subset B_A^{(k)}$ the connected component
  of the subgroup 
  of  elements of order $\leq p = \Char(\kk)$. Both $B_A^{(k)}$ and $W_A$ are normal in $B_A$. Since $B_A^{(k)}\subset [B_A,B_A]$, Corollary \ref{corB} implies that $B_A^{(k)}$ is unipotent; and 
  $B_A^{(k)}$ being~\algebraicallyTame{},  
  $W_A$ is non-trivial.

\ps
(b) We first claim that
 $W_A \cap U_A$ contains a non-trivial connected subgroup. 
 Assume this is not the case. Then
  we have an algebraic subgroup $U' \subset W_A$ that is isomorphic to $\Ga$ and satisfies $U' \not\subseteq U_A$. Since $U_A$ coincides with its own centralizer, 
  there exists some $\TP_f \in U_A$ that does not commute with $U'$.
  Then, 
 \[
U_f:=\set{\TP_{cf} u' ( \TP_{cf})^{-1}}{u' \in U' \, , \ c \in \kk} \subset W_A
 \]
 is an irreducible constructible subset of positive dimension of commuting automorphisms. By Theorem \ref{main-thm-A}, the subgroup $U \subset W_A$ generated by $U_f$ is commutative and algebraic,  hence   a vector group $\subset W_A$, and by construction $U$ is normalized by
 $
 \mathrm{T} := \{  \TP_{cf} \mid   c \in \kk \}  \simeq \mathbb{G}_a.
 $
By Corollary~4 of \cite{Bi1966Remarks-on-the-act}, 
there is a subgroup $V \subset U \subset W_A$ which is centralized by $\mathrm{T}$ and isomorphic to $\Ga$.
By assumption the connected component of $U \cap W_A$ is the identity, so $V$ is not contained in  $U_A$.
Since it commutes with $\mathrm{T}$,  $V$ induces a non-trivial action on $A = \OOO(X_A \times \mathbb{A}^1)^{ \mathrm{T}}$ and $V$ normalizes $U_A$. 
And
the commutator $[V,U_A]$  is a non-trivial subgroup of $ U_A$ because $U_A$ coincides with its centralizer.
Since $U_A$ normalizes  $W_A$ and $V \subset W_A$, $[V,U_A] $ is a subgroup of $ W_A$.
Hence,  $[V,U_A]$ is a non-trivial connected subgroup of 
$W_A \cap U_A$ which proves the claim.

\ps
(c) Let us prove that $W_A \subset U_A$.  
Otherwise we can  pick a subgroup $V \subset W_A$ which is isomorphic to $\Ga$ and is not contained in $U_A$.
Note that $V$ commutes with the non-trivial connected subgroup $(W_A \cap U_A)^\circ$ of $U_A$, hence it
induces an action on $A = \OOO(X_A \times \mathbb{A}^1)^{ (W_A \cap U_A)^\circ}$ and consequently on $X_A$. Thus, $V$ permutes the fibers of
$
\pi \colon X_A \times \mathbb{A}^1 \to X_A
$: every   $v \in V$ can be written as $v(x,w) =(\varphi_v(x), w + \alpha_v(x))$ for some pair $(\varphi_v, \alpha_v)\in \Aut(X_A)\times A$.  
Then,  $v^{-1} T_f v = T_{f \circ \varphi_v}$ and
$[v^{-1},T_f] =  T_{f \circ \varphi_v - f}$.
Since $[V,U_A]$ is a subgroup of the commutative group $W_A$, 
$[V,[V,U_A]] = \{ \id \}$. 
For $\Tilde{v} \in V$, this gives
\[
    \id
    [\Tilde{v}^{-1},T_{f \circ \varphi_v - f}] = T_{(f \circ \varphi_v - f)\circ \varphi_{\Tilde{v}} - (f \circ \varphi_v - f)}
\]
which implies   
\[
    (f \circ \varphi_v - f)\circ \varphi_{\Tilde{v}} = (f \circ \varphi_v - f)
\]
for all $v,\Tilde{v} \in V$, $f \in A$.
This contradicts Lemma \ref{expressionnotzerousingBrion} below and proves the claim.

\ps
(d) Pick two algebras $A=A_{(r,g)}$ and $A'=A_{(r',g')}$. If 
 $B_A$ is conjugate to $B_{A'}$ in $\Aut(X_A \times \mathbb{A}^1) \simeq \Aut(\mathbb{A}^{m+3})$ then $W_A$ is conjugate to $W_{A'}$, 
 hence $A=\OOO(X_A \times \mathbb{A}^1)^{W_A} \simeq 
 \OOO(X_{A'} \times \mathbb{A}^1)^{W_{A'}}=A'$ as algebras. Since the 
 $A_{(r,g)}$ are pairwise non-isomorphic, this concludes the proof.
\end{proof}

\begin{lem}\label{expressionnotzerousingBrion}
Let $X$ be an irreducible affine variety 
endowed with a faithful $\mathbb{G}_a$-action. Then there exist 
$s, \tilde{s} \in \Ga$ and $f \in \mathcal{O}(X)$ such that 
\[
  \tilde{s} \cdot (s \cdot f - f) \neq (s \cdot f - f) \, ,
\]
where $s \cdot f$ denotes the regular function on $X$ given by $x \mapsto f(sx)$.
\end{lem}
\begin{proof} 
Let $p$ be the characteristic exponent of $\kk$, i.e.\ $p = \max(1, \Char(\kk))$.
By Corollary~4 of \cite{Br2021Homogeneous-variet},  
$X$ contains a $\Ga$-stable open subset $X_0$ which is $\Ga$-isomorphic to 
$\mathbb{A}^1 \times Z$ for some affine variety $Z$ with $\kk(Z) = \kk(X)^{\Ga}$ and 
$\Ga$ acts on $\mathbb{A}^1 \times Z$ via
$
  s \cdot (t,z) = (t + P(z,s), z),
$
where $P \in \mathcal{O}(Z)[s]$ is a monic additive polynomial in $s$.
As $\kk(X)^{\Ga}$ is the quotient field of $\OOO(X)^{\Ga}$, after shrinking $X_0$ we may
assume that $X_0 = \set{x \in X}{h(x) \neq 0}$ for some non-zero invariant 
$h \in \mathcal{O}(X)^{\Ga}$. We identify $X_0 = \mathbb{A}^1 \times Z$.

There exists $l \geq 0$ such that $f \coloneqq h^l t^{p+1}$
is a regular function on $X$, for $\OOO(X_0)$ is the localization of $\OOO(X)$ at $h$.
For $s \in \Ga$ we have
\begin{equation*} 
    s \cdot f
    - f = h^l(t+P_s)^{p+1} - h^lt^{p+1} =
    h^l(P_s t^p + P_s^p t + P_s^{p+1}) \; ,
\end{equation*}
where $P_s \in \OOO(Z)$ denotes the function $z \mapsto P(z, s)$.
Then for $s, \tilde{s} \in \Ga$ we compute
\begin{align*}
\tilde{s} \cdot (s \cdot f - f) - (s \cdot f - f) 
= h^l P_s P_{\tilde{s}}
(P_s^{p-1} + P_{\tilde{s}}^{p-1}) \; .
\end{align*}
Hence, for any given $s \in \Ga\setminus \{0\}$ we
find some $\tilde{s} \in \Ga\setminus \{0\}$ with $ 
  P_s^{p-1} + P_{\tilde{s}}^{p-1} \neq 0$, 
(because $P_s^{p-1} + T^{p-1} \in \kk(Z)[T]$ 
has only finitely many roots in $\kk(Z)$). Since $s, \tilde{s}$
are non-zero, $P_s, P_{\tilde{s}}$ are non-zero and the proof is finished. 
\end{proof}

\subsection{Group scheme structures on quotients}

\begin{prop}\label{group-scheme.prop}
Let $X$ be an irreducible variety, 
and let $U$ be a connected commutative unipotent algebraic subgroup of $\Aut(X)$. 
Assume that $X$ admits a geometric quotient  $\pi\colon X \to Y:=X/U$ such that $Y$ is smooth and $\pi$  is locally trivial in the Zariski-topology with fiber $\An$. If $\pi$ has a section $e\colon Y \to X$, then $X \to Y$ has a unique structure of a commutative group scheme with identity $e$
such that $\Phi \colon U \times Y \to X$, $(u,y)\mapsto ue(y)$, is a homomorphism of $Y$-group schemes. 
\par
\end{prop}

\begin{proof}
We have to show that the addition $\alpha\colon X \times_Y X \to X$ and the inverse $\iota\colon X \to X$ are morphisms. 
We have a commutative diagram
\[
\begin{CD}
U \times U \times Y @>{(u_1,u_2,y)\mapsto (u_1+u_2,y)}>{\tilde\alpha}> U \times Y\\
@VV{\Phi\times_Y\Phi}V @VV{\Phi}V \\
X\times_Y X @>{\alpha}>> X\\
@VV{\pi\times_Y\pi}V @VV{\pi}V\\
Y@=Y
\end{CD}
\]
Since the question is local on the base, we can assume that $Y$ is affine.
Since $\pi$ is a smooth and affine morphism, all varieties in this diagram 
are smooth and affine.
Note that the induced maps $\alpha_y\colon X_y \times X_y \to X_y$ on the fibers are morphisms for any $y \in Y$. 

To simplify the notation set $A:= U \times U \times Y$, $B:=X\times_Y X$, $\phi:=\Phi\times_Y\Phi\colon A \to B$ and $\psi:=\Phi\circ\tilde\alpha\colon A \to X$. Denoting by $\Gamma_{\psi}$  the graph of the $Y$-morphism $\psi$ and by $\Gamma_{\alpha}$ the graph of the $Y$-map $\alpha$  we get the following commutative diagram:
\[
\begin{CD}
A @>{\simeq}>> \Gamma_{\psi} @>{\subseteq}>> A \times_Y X \\
@VV{\phi}V @VV{\text{surjective}}V @VV{\phi \times_Y\id}V \\
B @>\beta>{\text{bijective}}> \Gamma_{\alpha} @>{\subseteq}>> B \times_Y X \\
@VV{p_B}V @VV{p_\Gamma}V @VVV\\
Y @= Y @= Y
\end{CD}
\]
It follows that $\Gamma_{\alpha}$ is the image of the irreducible
variety $\Gamma_{\psi}$ under the morphism $\phi\times_Y\id$, hence it is a constructible  subset of $B\times_Y X$. By construction, $p_\Gamma^{-1}(y)$ is the graph of the morphism $\alpha_y\colon B_y \to X_y$, and so  $\beta$ induces an isomorphism on the fibers $B_y \simto (\Gamma_\alpha)_y$.

As a consequence, the projection $\pr_{B}\colon B \times_Y X \to B$ induces a surjective morphism $p\colon \overline{\Gamma_{\alpha}} \to B$ which is injective on an open dense subset and restricts to isomorphisms $(\Gamma_\alpha)_y \simto B_y$ on the fibers for all $y \in Y$. Hence the differential $dp_{(b, x)}$ is an isomorphism on an open subset of $\Gamma_\alpha$ which implies that $p$ is birational.  
Since  $\overline{\Gamma_{\alpha}}$ is irreducible and 
$B$ is normal this implies that $p$ is an isomorphism, by the Lemma ~\ref{Igusa.lem} below. Hence
$\Gamma_{\alpha}=\overline{\Gamma_{\alpha}}$ and $\alpha = \pr_{X}\circ p^{-1}$ is a morphism.
\par
The proof that the inverse $\iota\colon X \to X$ is a morphism goes along the same lines.
\end{proof}

The following result is due to Igusa who used it in his proof of 
\cite[Lemma~4]{Ig1973Geometry-of-absolu}.
Another proof can be found in \cite[II.3.4 Lemma and Bemerkung on p.\,106]{Kr1984Geometrische-Metho}.

\begin{lem}  \label{Igusa.lem}
Let $A,B$ be irreducible affine varieties. Let $\phi\colon A \to B$ be a birational morphism. Assume that $B$ is normal and that $\codim_{B}\overline{B\setminus\phi(A)} \geq 2$. Then $\phi$ is an isomorphism.
\end{lem}

\begin{proof}
Let $r \in\kk(B)$ be a rational function on $B$. If $r$ is not regular on $B$, then $r^{-1}$ vanishes on a hypersurface $H \subset B$, because $B$ is normal. Since $\overline{B\setminus\phi(A)}$ has codimension at least 2 the inverse image $\phi^{-1}(H)$ contains a hypersurface $L \subset A$, and $\phi^*(r^{-1})$ vanishes on $L$. It follows that $\phi^*(r)$ is not regular on $A$. Hence, if the pull-back of a rational function $r$ on $B$ is regular on $A$, then $r$ is regular on $B$. Hence, $\phi^*\colon \OOO(B) \to \OOO(A)$ is surjective and thus an isomorphism.
\end{proof}

\subsection{Increasing sequences of algebraic subgroups.}

\begin{lem}
\label{Lem.Increading_seq_of_algebraic_subgroups}
Assume $\Char(\kk) = 0$ or $\kk$ is uncountable.
Let $G$ be an algebraic group that is an increasing union of 
closed subgroups $H_i$, $i\in \NN$. Then 
$G = H_i$ for some $i$.
\end{lem}

This  lemma does not hold if $\kk$ is  countable  and of positive characteristic. For example,
$\SL_2(\overline{\mathbb{F}_p})$ is the increasing union of the finite subgroups $\SL_2(\mathbb{F}_{p^n})$.
 
\begin{proof} 
The statement follows from \cite[Lemma~1.3.1]{FuKr2018On-the-geometry-of} when $\kk$ is uncountbale, so we assume 
$\Char(\kk) = 0$.
It suffices to show that $G$ contains a finitely generated 
dense subgroup, and we may assume $G$ connected.  
Since $\Gaff/R_u(\Gaff)$ is reductive,
it is generated by its root subgroups and a maximal torus 
(Theorem 26.3 in \cite{Hu1975Linear-algebraic-g}). Since $\Char(\kk) = 0$, both $\Ga(\kk)$ and $\Gm(\kk)$
contain elements that generate a dense subgroup. We deduce that both $R_u(\Gaff)$ and $\Gaff/R_u(\Gaff)$ contain a finitely generated dense subgroup, and so does $\Gaff$.

Using the short exact sequence $1 \to \Gaff \to G \to G/\Gaff \to 1$,
we may now assume that $G$ is an abelian variety. By  Corollary~1 in 
\cite[Ch. IV, \S\,18]{Mu2008Abelian-varieties} any abelian variety
is isogenous to a finite product of simple ones, thus
we may assume that $G$ is simple. Then, again because $\Char(\kk) = 0$, there is  an element of infinite order in $G$
(see~\cite[Theorem~10.1]{FrJa1974Approximation-theo}) and the claim follows.
\end{proof}

\subsection{Complements on the unipotent radical} In this section, the results will concern \algebraicallyTame{} subgroups $\GGG$ of $\Aut(X)$. Thus, the unipotent radical $R_u(\GGG)$ coincides with $R_u^\MMM(\GGG)$ where $\MMM$ is the set of all connected algebraic subgroups of $\GGG$.

\begin{rem} 
\label{Rem.Unipotent_radical}
\textit{Let $G$ be a connected  affine algebraic group. Then $R_u(G)$ contains every 
normal subgroup  $U$ of $G$ consisting of unipotent elements.} 
Indeed, since the unipotent elements form a closed subset of $G$ 
we may assume that $U$ is closed. 
Its connected component $U^\circ$ is unipotent and normal, hence contained in $R_u(G)$.
It follows that  $U / (U \cap R_u(G))$ is a finite unipotent normal subgroup of the reductive group $L = G / R_u(G)$. Such a subgroup must be central in $L$, hence contained in a maximal torus of $L$ (Corollary~A in \cite[\S26.2]{Hu1975Linear-algebraic-g}). Therefore, it is trivial, and so $U \subseteq R_u(G)$.
\end{rem}

\begin{lem}
\label{Lem.pr-torsion}
Assume $\Char(\kk) = p > 0$. If $U$ is a connected and unipotent 
algebraic subgroup of $\Aut(X)$, every element of $U$ is $p^{\dim X}$-torsion.  
\end{lem}

\begin{proof}
By considering the image of $U$ in $\Aut(Ux)$ for all $U$-orbits $U x$, 
we are reduced to the case where $U$ acts transitively on $X$, thus $X = U/U_x$
where
$x\in X$. Set  $n=\dim X$. Corollaire~3.15 of \cite[IV, \S4]{DeGa1970Groupes-algebrique} implies that $u^{p^n} \in U_x$ for every $u \in U$. Thus, for $v\in U$, 
we get 
$u^{p^n}  vx = v (v^{-1} u v)^{p^n} x= vx$,  
and the claim follows by transitivity of $U$.
\end{proof}

If $\GGG \subseteq \Aut(X)$ is a subgroup with a covering and directed set $\SSS$
of connected algebraic subgroups, 
we define $\GGG_{\aff}^{\SSS} \coloneqq \bigcup_{G \in \SSS} \Gaff$. It is a subgroup of $\GGG$, because $\Gaff$ contains all affine connected subgroups of $G$ (see \S~\ref{Rosdecompandunipelem}).

\begin{prop}
  \label{Prop.Key_unipotent_radical}
  Let $\GGG \subseteq \Aut(X)$ be a subgroup with a
  covering and directed set $\SSS$ of connected algebraic subgroups. 
  Let $\VVV$ be a subgroup of $\GGG$ that is a union of connected and unipotent algebraic subgroups.
\begin{enumerate}[\rm (a)]
    \item If $\Char(\kk) = p > 0$, then 
  $\VVV \cap \GGG_{\aff}^{\SSS}$ has index $\leq p^{3 \dim(X)^2}$ in $\VVV$.
  \item If $\Char(\kk) = 0$ or $\kk$ is uncountable or $X$ is quasi-affine, 
  then $\VVV \subseteq \GGG_{\aff}^{\SSS}$.
\end{enumerate}
\end{prop}

\begin{proof}
  By Theorem~\ref{Thm.Equivalence_alg_tame_nested} $\GGG$ is \algebraicallyTame{}, hence $\VVV$ too.
  Let $V \subseteq \VVV$ be a connected algebraic subgroup. Then by Remark~\ref{unipotent-group.rem}, $V$ is unipotent.
  
  \ps
  (a) Set $n = \dim X$.
  By Lemma~\ref{Lem.pr-torsion}, every element of $V$ is of order $p^k$ for some $k\leq n$.
  By \cite[proof of Prop.~5.5.4]{BrSaUm2013Lectures-on-the-st}
  the dimension of the abelian variety 
  $A \coloneqq G/G_{\aff}$ is bounded by $3n$ for any $G \in \SSS$. 
  In particular, the $p^n$-torsion $A[p^n]$ contains at most
  $p^{3n^2}$ elements. 
  Thus, the index of $V \cap G_{\aff}$ in $V \cap G$ is bounded by $p^{3n^2}$, so the index of $V \cap \GGG_{\aff}^{\SSS}$ in $V$ is also bounded by $p^{3n^2}$. This implies the first claim.

  \ps
  (b) If $\Char(\kk) = 0$, then $V \cap G$ is connected and affine, hence
  contained in $G_{\aff}$  for all 
  $G \in \SSS$. Thus $\VVV \subseteq \GGG_{\aff}^{\SSS}$. 
  If $X$ is quasi-affine, then $\GGG_{\aff}^{\SSS} = \GGG$, by 
  Lemma~\ref{quasiaffine-gives-affine.lem}.
  Now, suppose $\kk$ is uncountable, and 
  write $\GGG_{\aff}^{\SSS} = \bigcup_k G_k$
  as an increasing countable union of connected algebraic subgroups
  (Theorem~\ref{Thm.Equivalence_alg_tame_nested}). 
  Then, $V = V \cap G_k$ for some $k$ because
  $V$ is the union of at most countably many closed $V \cap G_k$-cosets (for variying $k$), and $\kk$ is uncountable; 
  here we use (a). Hence, the second claim follows.
\end{proof}
\begin{cor}
  \label{Cor.R_u^SSS(GGG)_R_u(GGG)}
  If $\GGG \subseteq \Aut(X)$ admits a covering and directed set $\SSS$ of connected algebraic subgroups, then $R_u(\GGG) \cap \GGG_{\aff}^{\SSS} \subseteq R_u^{\SSS}(\GGG)$ and the index of
  $R_u^{\SSS}(\GGG)$  in $R_u(\GGG)$ is bounded by 
  $p^{3\dim(X)^2}$, where $p$ 
  is the characterstic exponent of $\kk$.
  If, moreover, $X$ is quasi-affine or $\Char(\kk) = 0$ or $\kk$ is uncountable, then 
  $R_u^{\SSS}(\GGG) = R_u(\GGG)$.
\end{cor}

\begin{proof}
  By Proposition~\ref{Pro.Unipotent_radical} 
  the subgroup
  $R_u^{\SSS}(\GGG) \subset \GGG$  is normal and is a union of
  connected unipotent algebraic groups.
  As $\GGG$ is \algebraicallyTame{} (Theorem~\ref{Thm.Equivalence_alg_tame_nested}),
  the same holds for $R_u(\GGG)^{\SSS}$ and hence $R_u(\GGG)^{\SSS}$ is a countable 
  increasing union of connected algebraic subgroups $U_k$, all of which are unipotent 
  (Remark~\ref{unipotent-group.rem}). 
  By Corollary~\ref{Cor.unipotent_solvable_subgroup}, 
  $\dl(U_k) \leq \dim X$, thus $R_u(\GGG)^{\SSS}$ is solvable.
  Hence, by Theorem~\ref{thm.Existence_radical}, 
  $R_u^{\SSS}(\GGG) \subseteq R_u(\GGG)$. 
  For any $G \in \SSS$ such that $\dim G/G_{\aff}$
  is maximal, Remark~\ref{Rem.Unipotent_radical} implies $R_u(\GGG) \cap G_{\aff} \subseteq R_u(G)$,
  and then $R_u(\GGG) \cap \GGG_{\aff}^{\SSS} \subseteq R_u^{\SSS}(\GGG)$. 
  Now the claims follow from Proposition~\ref{Prop.Key_unipotent_radical}.
\end{proof}

\subsection{Maximal derived length solvable subgroups}
The following example shows that Theorem~\ref{main-thm-derived-length}(2) fails
for non-quasi-affine varieties:

\begin{exa}
  \label{Exa.Smooth_proj_dl_n_plus_1}
  \textit{There is a smooth projective variety of dimension $n$ with a faithful
  action of a connected solvable algebraic group of derived length $n+1$
  (its unipotent radical has derived length $n$).}
  Let $d \geq 0$. The group $\Gm^n$ acts on $(\bbA^2 \setminus \{0 \})^n$ by  
  \[ 
   \lambda\cdot z = 
    (\lambda_i \lambda_{i-1}^d \cdots \lambda_1^{d^{i-1}} x_i, \lambda_i y_i )_{_{i = 1, \ldots, n}},
  \]
  where $\lambda=(\lambda_1, \cdots, \lambda_n)\in \Gm^n$ and 
  $z=((x_1, y_1), \ldots, (x_n, y_n))\in (\bbA^2 \setminus \{0 \})^n$. 
  Let $X_n$ be the corresponding geometric quotient. 
  Forgetting successively the last pair of coordinates yields a composition
  $
    X_n \to X_{n-1} \to \cdots \to X_1 = \mathbb{P}^1
  $
  of $\mathbb{P}^1$-bundles, hence $X_n$ is a smooth projective variety.
  Let $\Jonq_{n, d} \subseteq \Jonq_n \subseteq \Aut(\mathbb{A}^n)$ 
  be the connected algebraic subgroup consisting of those elements
  \begin{equation}\label{eq:example_jonq_n,d}
    (x_1, \ldots, x_n) \mapsto (a_1 x_1 + p_1, 
    a_2 x_2 + p_2(x_1), \ldots, a_n x_n + p_n(x_1, \ldots, x_{n-1})) \, ,
  \end{equation}
  where each monomial $x_1^{e_1} \cdots x_{i-1}^{e_{i-1}}$ in $p_i$ 
  satisfies 
  $e_1 + d e_2 + \ldots + d^{i-2} e_{i-1} \leq d^{i-1}$
  (see Proposition~15.2.5 in \cite{FuKr2018On-the-geometry-of}).
  The element from~\eqref{eq:example_jonq_n,d} acts on $(\bbA^2 \setminus \{0 \})^n$ via
  \[
    (x_i, y_i)_{i=1, \ldots, n} \mapsto 
    (a_i x_i + y_iP_{i}(x_1, y_1, \ldots, x_{i-1}, y_{i-1}), y_i)_{i=1, \ldots, n}
  \]
  where $P_i \in \kk[x_1, y_1, \ldots, x_{i-1}, y_{i-1}]$ is the unique 
  semi-invariant of weight $\lambda_{i-1}^d \cdots \lambda_1^{d^{i-1}}$
  for the given $\mathbb{G}_m^n$-action on $(\bbA^2 \setminus \{0\})^n$ with
  $p_i(x_1, \ldots, x_{i-1}) = P_i(x_1, 1, \ldots, x_{i-1}, 1)$. This action commutes with the $\Gm^n$-action and induces
  an algebraic action of $\Jonq_{n, d}$  on $X_n$, which is  faithful because its restriction to the 
  affine chart $y_1 = \cdots = y_n = 1$ is the natural action on $\mathbb{A}^n$.
  Now, $\dl(\Jonq_{n, d}) = n+1$ if $d$ is big enough since
  $\cup_{d \geq 0} \Jonq_{n, d} = \Jonq_n$.
\end{exa}
}}

\par\bigskip
\bibliographystyle{alpha}
\bibliography{biblio-solvable}

\end{document}